\documentclass[9pt]{amsart}
\usepackage{amsmath}
\usepackage{amsxtra}
\usepackage{amscd}
\usepackage{amsthm}
\usepackage{amsfonts}
\usepackage{amssymb}
\usepackage{eucal}
\usepackage[all]{xy}
\usepackage{graphicx}

\newtheorem{cor}[subsubsection]{Corollary}
\newtheorem{lem}[subsubsection]{Lemma}
\newtheorem{prop}[subsubsection]{Proposition}

\newtheorem{conj}[subsubsection]{Conjecture}
\newtheorem{thm}[subsubsection]{Theorem}

\theoremstyle{definition}

\theoremstyle{remark}
\newtheorem{rem}[subsubsection]{Remark}

\newcommand{\thmref}[1]{Theorem~\ref{#1}}

\newcommand{\secref}[1]{Sect.~\ref{#1}}
\newcommand{\lemref}[1]{Lemma~\ref{#1}}
\newcommand{\propref}[1]{Proposition~\ref{#1}}
\newcommand{\corref}[1]{Corollary~\ref{#1}}
\newcommand{\conjref}[1]{Conjecture~\ref{#1}}

\numberwithin{equation}{section}

\newcommand{\nc}{\newcommand}
\nc{\renc}{\renewcommand}
\nc{\ssec}{\subsection}
\nc{\sssec}{\subsubsection}
\nc{\on}{\operatorname}

\nc\ol{\overline}
\nc\wt{\widetilde}
\nc\tboxtimes{\wt{\boxtimes}}
\nc\tstar{\wt{\star}}
\nc{\alp}{a}

\nc{\ZZ}{{\mathbb Z}}
\nc{\NN}{{\mathbb N}}
\nc{\OO}{{\mathbb O}}
\renc{\SS}{{\mathbb S}}
\nc{\DD}{{\mathbb D}}
\nc{\GG}{{\mathbb G}}

\nc{\Fq}{{\mathbb F}_q}
\nc{\Fqb}{\ol{{\mathbb F}_q}}
\nc{\Ql}{\ol{{\mathbb Q}_\ell}}
\nc{\id}{\text{id}}
\nc\X{\mathcal X}

\nc{\Hom}{\on{Hom}}
\nc{\Lie}{\on{Lie}}
\nc{\Loc}{\on{Loc}}
\nc{\Pic}{\on{Pic}}
\nc{\Bun}{\on{Bun}}
\nc{\IC}{\on{IC}}
\nc{\Aut}{\on{Aut}}
\nc{\rk}{\on{rk}}
\nc{\Sh}{\on{Sh}}
\nc{\Perv}{\on{Perv}}
\nc{\pos}{{\on{pos}}}
\nc{\Conv}{\on{Conv}}
\nc{\Sph}{\on{Sph}}
\nc{\Sym}{\on{Sym}}
\nc{\BunBb}{\overline{\Bun}_B}
\nc{\BunNb}{\overline{\Bun}_N}
\nc{\BunTb}{\overline{\Bun}_T}
\nc{\BunBbm}{\overline{\Bun}_{B^-}}
\nc{\BunBbel}{\overline{\Bun}_{B,el}}
\nc{\BunBbmel}{\overline{\Bun}_{B^-,el}}
\nc{\Buno}{\overset{o}{\Bun}}
\nc{\BunPb}{{\overline{\Bun}_P}}
\nc{\BunBM}{\Bun_{B(M)}}
\nc{\BunBMb}{\overline{\Bun}_{B(M)}}
\nc{\BunPbw}{{\widetilde{\Bun}_P}}
\nc{\BunBP}{\widetilde{\Bun}_{B,P}}
\nc{\GUb}{\overline{G/U}}
\nc{\GUPb}{\overline{G/U(P)}}

\nc{\Hhom}{\underline{\on{Hom}}}
\nc\syminfty{\on{Sym}^{\infty}}
\nc\lal{\ol{\kappa_x}}
\nc\xl{\ol{x}}
\nc\thl{\ol{\theta}}
\nc\nul{\ol{\nu}}
\nc\mul{\ol{\mu}}
\nc\Sum\Sigma
\nc{\oX}{\overset{o}{X}{}}
\nc{\hl}{\overset{\leftarrow}h{}}
\nc{\hr}{\overset{\rightarrow}h{}}
\nc{\M}{{\mathcal M}}
\nc{\N}{{\mathcal N}}
\nc{\F}{{\mathcal F}}
\nc{\D}{{\mathcal D}}
\nc{\Q}{{\mathcal Q}}
\nc{\Y}{{\mathcal Y}}
\nc{\G}{{\mathcal G}}
\nc{\E}{{\mathcal E}}
\nc{\CalC}{{\mathcal C}}
\nc\Dh{\widehat{\D}}

\nc{\C}{{\mathcal C}}
\nc{\K}{{\mathcal K}}
\renewcommand{\H}{{\mathcal H}}

\nc{\T}{{\mathcal T}}
\nc{\V}{{\mathcal V}}
\renc{\P}{{\mathcal P}}
\nc{\A}{{\mathcal A}}
\nc{\B}{{\mathcal B}}
\nc{\U}{{\mathcal U}}

\nc{\Gr}{{\on{Gr}}}

\nc{\frn}{{\check{\mathfrak u}(P)}}

\nc{\fC}{\mathfrak C}
\nc{\p}{\mathfrak p}
\nc{\q}{\mathfrak q}
\nc\f{{\mathfrak f}}

\nc{\qo}{{\mathfrak q}}
\nc{\po}{{\mathfrak p}}
\nc{\s}{{\mathfrak s}}
\nc\w{\text{w}}

\renewcommand{\mod}{{\on{-mod}}}

\nc\mathi\iota
\nc\Spec{\on{Spec}}
\nc\Mod{\on{Mod}}
\nc{\tw}{\widetilde{\mathfrak t}}
\nc{\pw}{\widetilde{\mathfrak p}}
\nc{\qw}{\widetilde{\mathfrak q}}
\nc{\jw}{\widetilde j}

\nc{\grb}{\overline{\Gr}}
\nc{\I}{\mathcal I}

\nc{\kappach}{{\check\kappa_x}}
\nc{\Lambdach}{{\check\Lambda}{}}
\nc{\much}{{\check\mu}}
\nc{\omegach}{{\check\omega}}
\nc{\nuch}{{\check\nu}}
\nc{\etach}{{\check\eta}}
\nc{\alphach}{{\checka}}
\nc{\oblvtach}{{\check\oblvta}}
\nc{\pich}{{\check\pi}}
\nc{\ch}{{\check h}}

\nc{\Hb}{\overline{\H}}

\nc{\BA}{{\mathbb{A}}}
\nc{\BC}{{\mathbb{C}}}
\nc{\BE}{{\mathbb{E}}}
\nc{\BF}{{\mathbb{F}}}
\nc{\BG}{{\mathbb{G}}}
\nc{\BM}{{\mathbb{M}}}
\nc{\BO}{{\mathbb{O}}}
\nc{\BD}{{\mathbb{D}}}
\nc{\BN}{{\mathbb{N}}}
\nc{\BP}{{\mathbb{P}}}
\nc{\BQ}{{\mathbb{Q}}}
\nc{\BR}{{\mathbb{R}}}
\nc{\BZ}{{\mathbb{Z}}}
\nc{\BS}{{\mathbb{S}}}

\nc{\CA}{{\mathcal{A}}}
\nc{\CB}{{\mathcal{B}}}

\nc{\CE}{{\mathcal{E}}}
\nc{\CF}{{\mathcal{F}}}
\nc{\CG}{{\mathcal{G}}}
\nc{\CH}{{\mathcal{H}}}

\nc{\CL}{{\mathcal{L}}}
\nc{\CC}{{\mathcal{C}}}
\nc{\CM}{{\mathcal{M}}}
\nc{\CN}{{\mathcal{N}}}
\nc{\CK}{{\mathcal{K}}}
\nc{\CO}{{\mathcal{O}}}
\nc{\CP}{{\mathcal{P}}}
\nc{\CQ}{{\mathcal{Q}}}
\nc{\CR}{{\mathcal{R}}}
\nc{\CS}{{\mathcal{S}}}
\nc{\CT}{{\mathcal{T}}}
\nc{\CU}{{\mathcal{U}}}
\nc{\CV}{{\mathcal{V}}}
\nc{\CW}{{\mathcal{W}}}
\nc{\CX}{{\mathcal{X}}}
\nc{\CY}{{\mathcal{Y}}}
\nc{\CZ}{{\mathcal{Z}}}
\nc{\CI}{{\mathcal{I}}}
\nc{\CJ}{{\mathcal{J}}}

\nc{\csM}{{\check{\mathcal A}}{}}
\nc{\oM}{{\overset{\circ}{\mathcal M}}{}}
\nc{\obM}{{\overset{\circ}{\mathbf M}}{}}
\nc{\oCA}{{\overset{\circ}{\mathcal A}}{}}
\nc{\obA}{{\overset{\circ}{\mathbf A}}{}}
\nc{\ooM}{{\overset{\circ}{M}}{}}
\nc{\osM}{{\overset{\circ}{\mathsf M}}{}}
\nc{\vM}{{\overset{\bullet}{\mathcal M}}{}}
\nc{\nM}{{\underset{\bullet}{\mathcal M}}{}}
\nc{\oD}{{\overset{\circ}{\mathcal D}}{}}
\nc{\obD}{{\overset{\circ}{\mathbf D}}{}}
\nc{\oA}{{\overset{\circ}{\mathbb A}}{}}
\nc{\op}{{\overset{\bullet}{\mathbf p}}{}}
\nc{\cp}{{\overset{\circ}{\mathbf p}}{}}
\nc{\oU}{{\overset{\bullet}{\mathcal U}}{}}
\nc{\oZ}{{\overset{\circ}{\mathcal Z}}{}}
\nc{\ofZ}{{\overset{\circ}{\mathfrak Z}}{}}
\nc{\oF}{{\overset{\circ}{\fF}}}

\nc{\fa}{{\mathfrak{a}}}
\nc{\fb}{{\mathfrak{b}}}
\nc{\fd}{{\mathfrak{d}}}
\nc{\ff}{{\mathfrak{f}}}
\nc{\fg}{{\mathfrak{g}}}
\nc{\fgl}{{\mathfrak{gl}}}
\nc{\fh}{{\mathfrak{h}}}
\nc{\fj}{{\mathfrak{j}}}
\nc{\fl}{{\mathfrak{l}}}
\nc{\fm}{{\mathfrak{m}}}
\nc{\fn}{{\mathfrak{n}}}
\nc{\fu}{{\mathfrak{u}}}
\nc{\fp}{{\mathfrak{p}}}
\nc{\fr}{{\mathfrak{r}}}
\nc{\fs}{{\mathfrak{s}}}
\nc{\ft}{{\mathfrak{t}}}
\nc{\fz}{{\mathfrak{z}}}
\nc{\fsl}{{\mathfrak{sl}}}
\nc{\hsl}{{\widehat{\mathfrak{sl}}}}
\nc{\hgl}{{\widehat{\mathfrak{gl}}}}
\nc{\hg}{{\widehat{\mathfrak{g}}}}
\nc{\chg}{{\widehat{\mathfrak{g}}}{}^\vee}
\nc{\hn}{{\widehat{\mathfrak{n}}}}
\nc{\chn}{{\widehat{\mathfrak{n}}}{}^\vee}

\nc{\fA}{{\mathfrak{A}}}
\nc{\fB}{{\mathfrak{B}}}
\nc{\fD}{{\mathfrak{D}}}
\nc{\fE}{{\mathfrak{E}}}
\nc{\fF}{{\mathfrak{F}}}
\nc{\fG}{{\mathfrak{G}}}
\nc{\fK}{{\mathfrak{K}}}
\nc{\fL}{{\mathfrak{L}}}
\nc{\fM}{{\mathfrak{M}}}
\nc{\fN}{{\mathfrak{N}}}
\nc{\fP}{{\mathfrak{P}}}
\nc{\fU}{{\mathfrak{U}}}
\nc{\fV}{{\mathfrak{V}}}
\nc{\fZ}{{\mathfrak{Z}}}

\nc{\bb}{{\mathbf{b}}}
\nc{\bc}{{\mathbf{c}}}
\nc{\bd}{{\mathbf{d}}}
\nc{\bbf}{{\mathbf{f}}}
\nc{\be}{{\mathbf{e}}}
\nc{\bi}{{\mathbf{i}}}
\nc{\bj}{{\mathbf{j}}}
\nc{\bn}{{\mathbf{n}}}
\nc{\bp}{{\mathbf{p}}}
\nc{\bq}{{\mathbf{q}}}
\nc{\bu}{{\mathbf{u}}}
\nc{\bv}{{\mathbf{v}}}
\nc{\bx}{{\mathbf{x}}}
\nc{\bs}{{\mathbf{s}}}
\nc{\by}{{\mathbf{y}}}
\nc{\bw}{{\mathbf{w}}}
\nc{\bA}{{\mathbf{A}}}
\nc{\bK}{{\mathbf{K}}}
\nc{\bB}{{\mathbf{B}}}
\nc{\bF}{{\mathbf{F}}}
\nc{\bC}{{\mathbf{C}}}
\nc{\bG}{{\mathbf{G}}}
\nc{\bD}{{\mathbf{D}}}
\nc{\bE}{{\mathbf{E}}}
\nc{\bH}{{\mathbf{H}}}
\nc{\bO}{{\mathbf{O}}}
\nc{\bI}{{\mathbf{I}}}
\nc{\bM}{{\mathbf{M}}}
\nc{\bN}{{\mathbf{N}}}
\nc{\bV}{{\mathbf{V}}}
\nc{\bW}{{\mathbf{W}}}
\nc{\bX}{{\mathbf{X}}}
\nc{\bZ}{{\mathbf{Z}}}
\nc{\bS}{{\mathbf{S}}}

\nc{\sA}{{\mathsf{A}}}
\nc{\sB}{{\mathsf{B}}}
\nc{\sC}{{\mathsf{C}}}
\nc{\sD}{{\mathsf{D}}}
\nc{\sF}{{\mathsf{F}}}
\nc{\sK}{{\mathsf{K}}}
\nc{\sM}{{\mathsf{M}}}
\nc{\sO}{{\mathsf{O}}}
\nc{\sW}{{\mathsf{W}}}
\nc{\sQ}{{\mathsf{Q}}}
\nc{\sP}{{\mathsf{P}}}
\nc{\sZ}{{\mathsf{Z}}}
\nc{\sr}{{\mathsf{r}}}
\nc{\bk}{{\mathsf{k}}}
\nc{\sg}{{\mathsf{g}}}
\nc{\sff}{{\mathsf{f}}}
\nc{\sfb}{{\mathsf{b}}}
\nc{\sfc}{{\mathsf{c}}}
\nc{\sd}{{\mathsf{d}}}

\nc{\BK}{{\bar{K}}}

\nc{\tA}{{\widetilde{\mathbf{A}}}}
\nc{\tB}{{\widetilde{\mathcal{B}}}}
\nc{\tg}{{\widetilde{\mathfrak{g}}}}
\nc{\tG}{{\widetilde{G}}}
\nc{\TM}{{\widetilde{\mathbb{M}}}{}}
\nc{\tO}{{\widetilde{\mathsf{O}}}{}}
\nc{\tU}{{\widetilde{\mathfrak{U}}}{}}
\nc{\TZ}{{\tilde{Z}}}
\nc{\tx}{{\tilde{x}}}
\nc{\tbv}{{\tilde{\bv}}}
\nc{\tfP}{{\widetilde{\mathfrak{P}}}{}}
\nc{\tz}{{\tilde{\zeta}}}
\nc{\tmu}{{\tilde{\mu}}}

\nc{\urho}{\underline{\pi}}
\nc{\uB}{\underline{B}}
\nc{\uC}{{\underline{\mathbb{C}}}}
\nc{\ui}{\underline{i}}
\nc{\uj}{\underline{j}}
\nc{\ofP}{{\overline{\mathfrak{P}}}}
\nc{\oB}{{\overline{\mathcal{B}}}}
\nc{\og}{{\overline{\mathfrak{g}}}}
\nc{\oI}{{\overline{I}}}

\nc{\eps}{\varepsilon}
\nc{\hrho}{{\hat{\pi}}}

\nc{\one}{{\mathbf{1}}}
\nc{\two}{{\mathbf{t}}}

\nc{\Rep}{{\mathop{\operatorname{\rm Rep}}}}
\nc{\Tot}{{\mathop{\operatorname{\rm Tot}}}}
\nc{\Ker}{{\mathop{\operatorname{\rm Ker}}}}
\nc{\Hilb}{{\mathop{\operatorname{\rm Hilb}}}}
\nc{\End}{{\mathop{\operatorname{\rm End}}}}
\nc{\Ext}{{\mathop{\operatorname{\rm Ext}}}}
\nc{\CHom}{{\mathop{\operatorname{{\mathcal{H}}\it om}}}}
\nc{\GL}{{\mathop{\operatorname{\rm GL}}}}
\nc{\gr}{{\mathop{\operatorname{\rm gr}}}}
\nc{\Id}{{\mathop{\operatorname{\rm Id}}}}
\nc{\de}{{\mathop{\operatorname{\rm def}}}}
\nc{\length}{{\mathop{\operatorname{\rm length}}}}
\nc{\supp}{{\mathop{\operatorname{\rm supp}}}}

\nc{\Cliff}{{\mathsf{Cliff}}}
\nc{\Fl}{\on{Fl}}
\nc{\Fib}{{\mathsf{Fib}}}
\nc{\Coh}{{\mathsf{Coh}}}
\nc{\QCoh}{{\on{QCoh}}}
\nc{\IndCoh}{{\on{IndCoh}}}
\nc{\FCoh}{{\mathsf{FCoh}}}

\nc{\reg}{{\text{\rm reg}}}

\nc{\cplus}{{\mathbf{C}_+}}
\nc{\cminus}{{\mathbf{C}_-}}
\nc{\cthree}{{\mathbf{C}_*}}
\nc{\Qbar}{{\bar{Q}}}
\nc\Eis{\on{Eis}}
\nc\Eisb{\ol\Eis{}}
\nc\Eisr{\on{Eis}^{rat}{}}
\nc\wh{\widehat}
\nc{\Def}{\on{Def_{\check{\fb}}(E)}}
\nc{\barZ}{\overline{Z}{}}
\nc{\barbarZ}{\overline{\barZ}{}}
\nc{\barpi}{\overline\iota}
\nc{\barbarpi}{\overline\barpi}
\nc{\barpip}{\overline\iota{}^+}
\nc{\barpim}{\overline\iota{}^-}

\nc{\fq}{\mathfrak q}

\nc{\fqb}{\ol{\fq}{}}
\nc{\fpb}{\ol{\fp}{}}
\nc{\fpr}{{\fp^{rat}}{}}
\nc{\fqr}{{\fq^{rat}}{}}

\nc{\hattimes}{\wh\otimes}

\nc{\bh}{{\bar{h}}}
\nc{\bOmega}{{\overline{\Omega(\check \fn)}}}

\nc{\seq}[1]{\stackrel{#1}{\sim}}

\nc{\cT}{{\check{T}}}
\nc{\cG}{{\check{G}}}
\nc{\cM}{{\check{M}}}
\nc{\cB}{{\check{B}}}

\nc{\ct}{{\check{\mathfrak t}}}
\nc{\cg}{{\check{\fg}}}
\nc{\cb}{{\check{\fb}}}
\nc{\cn}{{\check{\fn}}}

\nc{\cLambda}{{\check\Lambda}}

\nc{\cla}{{\check\kappa_x}}
\nc{\cmu}{{\check\mu}}
\nc{\cnu}{{\check\nu}}
\nc{\ceta}{{\check\eta}}

\nc{\DefbE}{{\on{Def}_{\cB}(E_\cT)}}

\nc{\imathb}{{\ol{\imath}}}
\nc{\rlr}{\overset{\longrightarrow}{\underset{\longrightarrow}\longleftarrow}}

\nc{\KG}{K\backslash G}
\nc{\comult}{{co\text{-}mult}}
\nc{\counit}{{co\text{-}unit}}
\nc{\uHom}{{\underline{\Maps}}}
\nc{\dgSch}{\on{Sch}}
\nc{\Sch}{\on{Sch}}
\nc{\affdgSch}{\on{Sch}^{\on{aff}}}
\nc{\affSch}{\on{Sch}^{\on{aff}}}
\nc{\Groupoids}{\on{Grpd}}
\nc{\inftygroup}{\on{Spc}}
\nc{\inftyCat}{\infty\on{-Cat}}
\nc{\StinftyCat}{\inftyCat^{\on{St}}}
\nc{\MoninftyCat}{\infty\on{-Cat}^{\on{Mon}}}
\nc{\SymMoninftyCat}{\infty\on{-Cat}^{\on{SymMon}}}
\nc{\SymMonStinftyCat}{\on{DGCat}^{\on{SymMon}}}
\nc{\MonStinftyCat}{\on{DGCat}^{\on{Mon}}}
\nc{\inftystack}{\on{Stk}}
\nc{\inftystackalg}{Stk^{1\text{-}alg}}
\nc{\inftyprestack}{\on{PreStk}}
\nc{\inftydgnearstack}{\on{NearStk}}
\nc{\inftydgstack}{\on{Stk}}
\nc{\inftydgstackalg}{DGStk^{1\text{-}alg}}
\nc{\inftydgprestack}{\on{PreStk}}
\nc{\dgindSch}{\on{indSch}}
\nc{\indSch}{{}^{\on{cl}}\!\on{indSch}}
\nc{\infSch}{\on{infSch}}
\nc{\dr}{{\on{dR}}}

\nc{\mmod}{{\on{-}\!{\mathbf{mod}}}}

\nc{\starr}{\text{\dh}}
\nc{\Spectra}{\on{Spectra}}
\nc{\Crys}{\on{Crys}}
\nc{\oblv}{{\mathbf{oblv}}}
\nc{\ind}{{\mathbf{ind}}}
\nc{\CMaps}{{\mathcal Maps}}
\nc{\Maps}{\on{Maps}}
\nc{\bMaps}{\mathbf{Maps}}
\nc{\BMaps}{\ul{\on{Maps}}}
\nc{\Grid}{\on{Grid}}
\nc{\hGrid}{\on{Grid}^{\geq\,\on{dgnl}}}
\nc{\Diag}{\on{Diag}}
\nc{\bDelta}{\mathbf{\Delta}}
\nc{\tCateg}{(\infty\on{-2)-Cat}}
\nc{\ul}{\underline}
\nc{\Seg}{\on{Seq}}
\nc{\biSeg}{\on{bi-Seq}}
\nc{\triSeg}{\on{tri-Seq}}
\nc{\quadSeg}{\on{quad-Seq}}
\nc{\nSeg}{\on{n-Seq}}
\nc{\Segm}{\on{Seg}^{\on{mkd}}}
\nc{\fLm}{\fL^{\on{mkd}}}
\nc{\inftyCatm}{\inftyCat^{\on{mkd}}}
\nc{\Blocks}{\mathbf{Blocks}}
\nc{\Snakes}{\mathbf{Snakes}}
\nc{\bifL}{\on{bi-}\!\fL}
\nc{\Sets}{\on{Sets}}
\nc{\Ran}{\on{Ran}}
\nc{\Vect}{\on{Vect}}
\nc{\Shv}{\on{Shv}}
\nc{\unn}{\mathbf{union}}
\nc{\Spc}{\on{Spc}}
\nc{\ppart}{(\!(t)\!)}
\nc{\qqart}{[\![t]\!]}
\nc{\Dmod}{\on{D-mod}}
\nc{\cD}{\mathcal D}
\nc{\ocD}{\overset{\circ}{\cD}}
\nc{\sfp}{\mathsf{p}}
\nc{\sfq}{\mathsf{q}}
\nc{\DGCat}{\on{DGCat}}
\renc{\det}{\on{det}}

\begin{document}

\title[Parameters for metaplectic Langlands theory]
{Parameters and duality for the  \\  metaplectic geometric Langlands theory}

\author{D. Gaitsgory and S.~Lysenko}

\dedicatory{For Sasha Beilinson} 

\date{\today}

\begin{abstract}

This is a corrected version of the paper, and it differs substantially from the original one.

\medskip

We introduce the space of parameters for the metaplectic Langlands theory
as \emph{factorization gerbes} on the affine Grassmannian, and develop metaplectic
Langlands duality in the incarnation of the metaplectic geometric Satake functor. 

\medskip

We formulate
a conjecture in the context of the global metaplectic Langlands theory, which is a metaplectic
version of the ``vanishing theorem" of \cite[Theorem 4.5.2]{Ga5}.
\end{abstract}

\maketitle

\section*{Introduction}

\ssec{What is this paper about?}

The goal of this paper is to provide a summary of the metaplectic Langlands theory. Our main objectives
are: 

\smallskip

\noindent--Description of the set (rather, space) of parameters for the metaplectic Langlands theory;

\smallskip

\noindent--Construction of the \emph{metaplectic Langlands dual} (see \secref{sss:L dual} for what we mean by this);

\smallskip

\noindent--The statement of the \emph{metaplectic geometric Satake}.

\sssec{The metaplectic setting}

Let $\bF$ be a local field and $G$ an algebraic group over $\bF$. The classical representation theory of locally compact
groups studies (smooth) representations of the group $G(\bF)$ on vector spaces over another field $E$.  Suppose now
that we are given a central extension 
\begin{equation} \label{e:metap ext}
1\to E^\times \to \wt{G(\bF)}\to G(\bF)\to 1.
\end{equation}

We can then study representations of $\wt{G(\bF)}$ on which the central $E^\times$ acts by the tautological character. We will refer
to \eqref{e:metap ext} as a \emph{local metaplectic extension} of $G(\bF)$, and to the above category of representations as
\emph{metaplectic representations} of $G(\bF)$ corresponding to the extension \eqref{e:metap ext}.

\medskip

Let now $\bF$ be a global field, and let $\BA_\bF$ be the corresponding ring of ad\`eles.  Let us be given a central extension
\begin{equation} \label{e:metap ext global}
1\to E^\times \to \wt{G(\BA_\bF)}\to G(\BA_\bF)\to 1,
\end{equation}
equipped with a splitting over $G(\bF)\hookrightarrow G(\BA_\bF)$. 

\medskip

We can then study the space of $E$-valued functions on the quotient $\wt{G(\BA_\bF)}/G(\bF)$, on which the 
central $E^\times$ acts by the tautological character. We will refer
to \eqref{e:metap ext global} as a \emph{global metaplectic extension} of $G(\bF)$, and to the above space of functions
as \emph{metaplectic automorphic functions} on $G(\bF)$ corresponding to the extension \eqref{e:metap ext global}.

\medskip

There has been a renewed interest in the study of metaplectic representations and metaplectic automorphic functions,
e.g., by B.Brubaker--D.Bump--S.Friedberg, P.McNamara, W.T.Gan--F.Gao. 

\medskip

M.~Weissman has initiated a program of constructing the L-groups corresponding to metaplectic extensions, to be used in the
formulation of the Langlands program in the metaplectic setting, see \cite{We}.  

\sssec{Parameters for metaplectic extensions}

In order to construct metaplectic extensions, in both the local and global settings, one starts with a datum of
algebro-geometric nature. Namely, one usually takes as an input what we call a \emph{Brylinski-Deligne datum}, by which we mean 
a central extension 
\begin{equation} \label{e:K2 ext}
1\to (K_2)_{\on{Zar}}\to \wt{G} \to G\to 1,
\end{equation}
of sheaves of groups on the big Zariski site of $\bF$, where $(K_2)_{\on{Zar}}$ is the sheafification of the presheaf of abelian groups that assigns 
to an affine scheme $S=\Spec(A)$ the group $K_2(A)$.

\medskip

For a local field $\bF$, let ${\mathbf f}$ denote its residue field and let us choose a homomorphism 
\begin{equation} 
{\mathbf f}^\times \to E^\times. 
\end{equation}
Then taking
the group of $\bF$-points of $\wt{G}$ and pushing out with respect to
$$K_2(\bF)\overset{\on{symbol}}\longrightarrow {\mathbf f}^\times \to E^\times,$$
we obtain a central extension \eqref{e:metap ext}. A similar procedure applies also in the global setting. 

\sssec{The geometric theory}

Let $k$ be a ground field and let $G$ be a reductive group over $k$. 

\medskip

In the local geometric Langlands theory one considers the loop group $G\ppart$ along with its action on various spaces, such as the affine
Grassmannian $\Gr_G=G\ppart/G\qqart$. Specfically one studies the behavior of categories of 
sheaves\footnote{See \secref{ss:sheaves} for what we mean by the category of sheaves.} on such spaces
with respect to this action. 

\medskip

In the global geometric Langlands theory one considers a smooth proper curve $X$, and one studies the stack $\Bun_G$ that classifies 
principal $G$-bundles on $X$. The main object of investigation is the category of sheaves on $\Bun_G$. 

\medskip

There are multiple ways in which the local and global theories interact. For example, given a ($k$-rational) point $x\in X$, 
and identifying the local ring $\CO_x$ of $X$ at $x$ with $k\qqart$, we have the map
\begin{equation} \label{e:Gr to Bun}
\Gr_G\to \Bun_G,
\end{equation} 
where we interpret $\Gr_G$ as the moduli space of principal $G$-bundles on $X$, trivialized over $X-x$.

\sssec{The setting of metaplectic geometric Langlands theory}

Let $E$ denote the field of coefficients of the sheaf theory that we consider. Recall (see \secref{sss:twist category sheaves}) that if $\CY$
is a space\footnote{By a ``space" we mean a scheme, stack, ind-scheme, or more generally a \emph{prestack},
see \secref{ss:prestacks} for what the latter word means.} and $\CG$ is a $E^\times$-gerbe on $\CY$, we can twist
the category of sheaves on $\CY$, and obtain a new category, denoted
$$\Shv_\CG(\CY).$$

\medskip

In the local metaplectic Langlands theory, the input datum (which is an analog of a central extension \eqref{e:metap ext}) is an $E^\times$-gerbe 
over the loop group $G\ppart$ that behaves \emph{multiplicatively}, i.e., one that is compatible with the group-law on $G\ppart$.

\medskip

Similarly, whenever we consider an action of $G\ppart$ on $\CY$, we equip $\CY$ with $E^\times$-gerbe that is compatible
with the given multiplicative gerbe on $G\ppart$. In this case we say that the category $\Shv_\CG(\CY)$ carries a \emph{twisted} 
action of $G\ppart$, where the parameter of the twist is our gerbe on $G\ppart$. 

\medskip

In the global setting we consider a gerbe $\CG$ over $\Bun_G$, and the corresponding category $\Shv_\CG(\Bun_G)$ of twisted 
sheaves.  

\medskip

Now, if we want to consider the local vs. global interaction, we need a compatibility structure on our gerbes. For example,
we need that for every point $x\in X$, the pullback along \eqref{e:Gr to Bun} of the given gerbe on $\Bun_G$ be a gerbe 
compatible with some given multiplicative gerbe on $G\ppart$. 

\medskip

So, it is natural to seek an algebro-geometric datum, akin to \eqref{e:K2 ext}, that would provide such a compatible family of gerbes.

\sssec{Geometric metaplectic datum}

It turns out that such a datum (let us call it ``the geometric metaplectic datum") is not difficult to describe, see \secref{sss:on Gr} below.
It amounts to the datum of a \emph{factorization gerbe} with respect to $E^\times$ on the 
\emph{affine Grassmannian}\footnote{Here the affine Grassmannian
appears in its factorization (a.k.a, Beilinson-Drinfeld) incarnation. I.e., it is a prestack mapping to the Ran space of $X$, rather than $G\ppart/G\qqart$,
which corresponds to a particular point of $X$.} $\Gr_G$ of the group $G$. 

\medskip

In a way, this answer is more elementary than \eqref{e:K2 ext} in that we are
dealing with \'etale cohomology rather than $K$-theory. 

\medskip

Moreover, in the original metaplectic setting, if the global field $\bF$ 
is the function field corresponding to the curve $X$ over a finite ground field $k$, a geometric metaplectic datum gives rise directly
to an extension \eqref{e:metap ext global}. 

\medskip

Finally, a Brylinski-Deligne datum (i.e.,  an extension \eqref{e:K2 ext}) and a choice of a character $k^\times\to E^\times$
gives rise to a geometric metaplectic datum, see \secref{ss:lines}.

\medskip

Thus, we could venture into saying that a geometric metaplectic datum
is a more economical way, sufficient for most purposes, to encode also the datum needed to set up the 
classical metaplectic representation/automorphic theory. 

\sssec{The metaplectic Langlands dual}  \label{sss:L dual}

Given a geometric metaplectic datum, i.e., a factorization gerbe $\CG$ on $\Gr_G$, we attach to it a certain reductive group $H$, 
a gerbe $\CG_{Z_H}$ on $X$ with respect to the center $Z_H$ of $H$,
and a character $\epsilon:\pm 1\to Z_H$. We refer to the triple
$$(H, \CG_{Z_H}, \epsilon)$$
as the \emph{metaplectic Langlands dual} datum corresponding to $\CG$. 

\medskip

The datum of $\CG_{Z_H}$ determines the notion of twisted $H$-local system of $X$. Such twisted local systems
are supposed to play a role vis-\`a-vis metaplectic representations/automorphic functions 
of $G$ parallel to that of usual $\cG$-local systems vis-\`a-vis usual representations/automorphic functions of $G$.

\medskip

For example, in the context of the global geometric theory (in the setting of D-modules), we will propose a conjecture (namely, \conjref{c:vanishing})
that says that the monoidal category $\QCoh\left(\on{LocSys}_{H}^{\CG_{Z_H}}\right)$ of quasi-coherent sheaves on the stack 
$\on{LocSys}_{H}^{\CG_{Z_H}}$ classifying such twisted local systems, \emph{acts} on the category 
$\Shv_\CG(\Bun_G)$. 

\medskip

The geometric input for such an action is provided by the metaplectic geometric Satake functor, see \secref{s:Satake}.

\medskip

Presumably, in the arithmetic context, the 
above notion of twisted $H$-local system coincides with that of homomorphism of
the (arithmetic) fundamental group of $X$ to Weissman's L-group.

\ssec{``Metaplectic" vs "Quantum"}

In the paper \cite{Ga4}, a program was proposed towards the \emph{quantum Langlands theory}.
Let us comment on the terminological difference between ``metaplectic" and ``quantum",
and how the two theories are supposed to be related.

\sssec{}

If $\CY$ is a scheme (resp., or more generally, a prestack) we can talk about $E^\times$-gerbes on it.
As was mentioned above, such gerbes on various spaces associated with the group $G$ and the geometry
of the curve $X$ are parameters for the metaplectic Langlands theory. 

\medskip

Let us now assume that $k$ has characteristic $0$, and let us work in the context of D-modules. Then,
in addition to the notion of $E^\times$-gerbe on $\CY$, there is another one: that of \emph{twisting}
(see \cite[Sect. 6]{GR1}). 

\medskip

There is a forgetful map from twistings to gerbes. Roughly speaking, a gerbe $\CG$ on $\CY$ defines the
corresponding twisted category of sheaves (=D-modules) $\Shv_\CG(\CY)=\Dmod_\CG(\CY)$, while 
if we lift our gerbe to a twsiting, we also have a forgetful functor
$$\Dmod_\CG(\CY)\to \QCoh(\CY).$$

\sssec{}

For the \emph{quantum} Langlands theory, our parameter will be a factorizable \emph{twisting} on the
affine Grassmannian, which one can also interpret as a \emph{Kac-Moody level}; we will denote it
by $\kappa$. 

\medskip

Thus, for example, in the global quantum geometric Langlands theory, we consider the category
$$\Dmod_\kappa(\Bun_G),$$
which is the same as $\Shv_\CG(\Bun_G)$, where $\CG$ is the gerbe corresponding to $\kappa$.

\medskip

As was mentioned above, the additional piece of datum that the twisting ``buys" us is the forgetful
functor
$$\Dmod_\kappa(\Bun_G)\to \QCoh(\Bun_G).$$

In the TQFT interpretation of geometric Langlands, this forgetful functor is called ``the big brane".
It allows us to relate the category $\Dmod_\kappa(\Bun_G)$ to representations of the Kac-Moody
algebra attached to $G$ and the level $\kappa$.

\sssec{}

Consider the usual Langlands dual group $\cG$ of $G$, and if $\kappa$ is non-degenerate, it gives rise to a twisting, denoted
$-\kappa^{-1}$, on the affine Grassmannian $\Gr_{\cG}$ of $\cG$. 

\medskip

In the global quantum geometric theory one expects to have an equivalence of categories
\begin{equation} \label{e:quatum L dual} 
\Dmod_{\kappa}(\Bun_G)\simeq  \Dmod_{-\kappa^{-1}}(\Bun_{\cG}).
\end{equation}

We refer to \eqref{e:quatum L dual} as the \emph{global quantum Langlands equivalence}.

\sssec{How are the two theories related?}

The relationship between the equivalence \eqref{e:quatum L dual} and the metaplectic Langlands dual is the following:

\medskip

Let $\CG$ (resp., $\check\CG$) be the gerbe on $\Gr_G$ (resp., $\Gr_{\cG}$) corresponding to $\kappa$ (resp., $-\kappa^{-1}$). 
We conjecture that the metaplectic Langlands dual data $(H,\CG_{Z_H},\epsilon)$ corresponding to 
$\CG$ and $\check\CG$ \emph{are isomorphic}.  

\medskip

Furthermore, we conjecture that the resulting actions of
$$\QCoh\left(\on{LocSys}_{H}^{\CG_{Z_H}}\right)$$ on 
$\Dmod_{\kappa}(\Bun_G)$ and $\Dmod_{-\kappa^{-1}}(\Bun_{\cG})$, respectively (see \secref{sss:L dual} above)
are intertwined by the equivalence \eqref{e:quatum L dual}.

\ssec{What is actually done in this paper?}

Technically, our focus is on the geometric metaplectic theory, with the goal of constructing the 
\emph{metaplectic geometric Satake} functor. 

\sssec{}  \label{sss:goals}

The mathematical content of this paper is the following: 

\medskip

\noindent--We define a geometric metaplectic datum to be a factorization gerbe on the (factorization version) of 
affine Grassmannian $\Gr_G$. This is done in \secref{s:Gr}. 

\medskip 

\noindent--We formulate the classification result that describes factorization gerbes on $\Gr_G$ in terms of \'etale
cohomology on the classifying stack $BG$ of $G$. This is done in \secref{s:param}. 

\medskip

This classification result is inspired by an analogous one in the topological setting, explained to us by J.~Lurie. 

\medskip 

\noindent--We make an explicit analysis of the space of factorization gerbes in the case when $G=T$ is a torus. 
This is done in \secref{s:torus}. 

\medskip 

\noindent--We study the relationship between factorization gerbes on $\Gr_G$ and those on $\Gr_M$, where $M$ is the
Levi quotient of a parabolic $P\subset G$. This is done in \secref{s:Jacquet}. 

\medskip

The main point is that the naive map from factorization gerbes on $\Gr_G$ to those on $\Gr_M$
needs to be corrected by a gerbe that has to do with signs. It is this correction that is responsible 
for the fact that the usual geometric Satake does not quite produce the category $\Rep(\cG)$, but rather its modification where we
alter the commutativity constraint by the element $2\rho(-1)\in Z(\cG)$. 

\medskip 

\noindent--We define the notion of \emph{metaplectic Langlands dual} datum, denoted $(H, \CG_{Z_H}, \epsilon)$,
attached to a given geometric metaplectic datum $\CG$. We introduce the notion of $\CG_{Z_H}$-twisted 
$H$-local system on $X$; when we work with D-modules, these local systems are $k$-points  of a (derived) algebraic
stack, denoted $\on{LocSys}_{H}^{\CG_{Z_H}}$. This is done in \secref{s:metap dual}. 

\medskip 

\noindent--We show that a factorization gerbe on $\Gr_G$ gives rise to a \emph{multiplicative} gerbe over the loop group
$G\ppart$ for every point $x\in X$. Moreover, these multiplicative gerbes also admit a natural factorization structure when
instead of a single point $x$ we consider the entire Ran space. This is done in \secref{s:loop}.  

\medskip 

\noindent--We introduce the various twisted versions of the category of representations of a reductive group,
and the associated notion of twisted local system. This is done in \secref{s:tw rep}.

\medskip 

\noindent--We define metaplectic geometric Satake as a functor between 
\emph{factorization categories} over the Ran space. This is done in \secref{s:Satake}.

\medskip 

\noindent--We formulate a conjecture about the action of the monoidal category 
$\QCoh\left(\on{LocSys}_{H}^{\CG_{Z_H}}\right)$ on $\Shv_\CG(\Bun_G)$.
This is also done in \secref{s:Satake}.

\sssec{A disclaimer}

Although most of the items listed in \secref{sss:goals} have not appeared in the previously existing literature,
this is mainly due to the fact that these earlier sources, specifically the paper \cite{FL} of M.~Finkelberg
and the second-named author and the paper \cite{Re} of R.~Reich, did not use the language of $\infty$-categories, 
while containing most of the relevant mathematics. 

\medskip

So, one can regard the present paper as a summary of results that are ``almost known", but formulated in the
language that is better adapted to the modern take on the geometric Langlands theory\footnote{This excludes, however, 
the material in \secref{ss:global Hecke} and the statement of \conjref{c:vanishing} (the latter is new, to the
best of our knowledge)}. 

\medskip

We felt that there was a need for such a summary in order to facilitate further research in this area.  

\medskip

Correspondingly, our focus is on statements, rather than proofs. Most of the omitted proofs can be found 
in either \cite{FL} or \cite{Re}, or can be obtained from other sources cited in the paper. 

\medskip

Below we give some details on the relation of contents of this paper and some of previously existing literature. 

\sssec{Relation to other work: geometric theory}

As was just mentioned, a significant part of this paper is devoted to reformulating the results of \cite{FL} and \cite{Re} in a way tailored for  
the needs of the geometric metaplectic theory.

\medskip

The paper \cite{Re} develops the theory of factorization gerbes on $\Gr_G$ (in {\it loc. cit.} they are called ``symmetric factorizable 
gerbes"). One caveat is that in the setting of \cite{Re} one works with schemes over $\BC$ and sheaves in the analytic topology, while in the present paper 
we work over a general ground field and \'etale sheaves. 

\medskip

The main points of the theory developed in \cite{Re} are the description of the \emph{homotopy groups} of the space of 
factorization gerbes (but not of the space itself; the latter is done in \secref{s:param} of the present paper), and the fact that a factorization gerbe
on $\Gr_G$ gives rise to a multiplicative gerbe on (the factorization version of) the loop group (we summarize this construction in \secref{s:loop}
of the present paper). 

\medskip

The proofs of the corresponding results in \cite{Re} are obtained by reducing assertions for a reductive group $G$ to that
for its Cartan subgroup, and an explicit analysis for tori. We do not reproduce these proofs in the present paper. 

\medskip

In both \cite{FL} and \cite{Re}, metaplectic geometric Satake is stated as an equivalence of certain abelian categories. 
In \cite{FL}, this is an equivalence of symmetric monoidal categories (corresponding to a chosen point $x\in X$),
for a particular class of gerbes (namely, ones obtained from the determinant line  bundle). 

\medskip

In \cite{Re} more general gerbes are considered and the factorization structure on both sides of the equivalence 
is taken into account.  Our version of metaplectic geometric Satake 
is a statement at the level of DG categories; it is no longer an equivalence, but rather a functor in one
direction, between \emph{monoidal factorization categories}.  In this form, our formulation is a simple consequence
of that of \cite{Re}. 

\sssec{Relation to other work: arithmetic theory}

As was already mentioned above, our notion of the metaplectic Langlands dual datum is probably equivalent
to the datum constructed by M.~Weissman in \cite{We} for his definition of the L-group. 

\ssec{Conventions}

\sssec{Algebraic geometry}

In the main body of the paper we will be working over a fixed ground field $k$, assumed algebraically closed. 

\medskip

For arithmetic applications one would also be interested in the case of $k$ being a finite field $\BF_q$.
However, since all the constructions in this paper are canonical, the results over $\BF_q$ can be deduced from
those over $\ol\BF_q$ by Galois descent. 

\medskip

We will denote by $X$ a smooth connected algebraic curve over $k$ (we \emph{do not} need $X$ to be complete). 

\medskip

For the purposes of this paper, we \emph{do not need} derived algebraic geometry, with the exception of Sects. \ref{ss:tw loc syst}
and \ref{ss:global Hecke Dmod} (where we discuss the stack of local systems, which is a derived object). 

\medskip

In the last two sections of the paper we will make an extensive use of algebro-geometric objects more general than schemes, namely,
prestacks. We recall the definition of prestacks in \secref{ss:prestacks}, and refer the reader to \cite[Vol. 1, Chapter 2]{GR2} for a 
more detailed discussion. 

\sssec{Coefficients}

Gerbes, which constitute the object of study of this paper, can be used to \emph{twist} categories of sheaves, see \secref{ss:twist category}.

\medskip

We will mostly work with the \emph{sheaf theory} of D-modules. We will denote by $E$ the field of coefficients of our sheaves
(assumed algebraically closed and of characteristic $0$). 

\sssec{Groups}

We will work with a fixed connected algebraic group $G$ over $k$; our main interest is the case when $G$ is reductive. 

\medskip

We will denote by $\Lambda$ the coweight lattice of $G$ and by $\cLambda$ its dual, i.e., the weight lattice. 

\medskip

We will denote by $\alpha_i\in \Lambda$ (resp., $\check\alpha_i\in \cLambda$) the simple coroots (resp., roots), where $i$ runs over the set of 
vertices of the Dynkin diagram of $G$.

\medskip

If $G$ is reductive, we denote by $\cG$ its Langlands dual, viewed as a reductive group over $E$. 

\sssec{The usage of higher category theory}

Although, as we have said above, we do not need derived algebraic geometry, we do need higher category theory. 
However, we only really need
$\infty$-categories for one type of manipulation: in order to define the notion of the \emph{category of sheaves} on a given prestack
(and a related notion of a \emph{sheaf of categories} over a prestack); 
we will recall the corresponding definitions in Sects. \ref{ss:prestacks} and \ref{ss:shv-of-cat}), respectively.  
These definitions involve the procedure of taking the limit, and the language of higher categories is the adequate framework for doing so. 

\medskip

In their turn, sheaves of categories on prestacks appear for us as follows: the metaplectic spherical Hecke category, which is the recipient of
the metaplectic geometric Satake functor (and hence is of primary interest for us), is a sheaf of categories over the Ran space. 

\medskip

Thus, the reader who is only interested in the notion of geometric metaplectic datum (and does not wish to proceed to 
metaplectic geometric Satake) \emph{does not} need higher category theory either. 



\sssec{Glossary of $\infty$-categories}

We will now recall several most common pieces of notation, pertaining to $\infty$-categories,
used in this paper. We refer the reader to \cite{Lu1,Lu2} for the foundations of the theory,
or \cite[Vol. 1, Chapter 1]{GR2} for a concise summary. 

\medskip

We denote by $\Spc$ the $\infty$-category of spaces. We denote by $*$ the point-space. 
For a space $\CS$, we denote by $\pi_0(\CS)$ its \emph{set of connected components}. If $\CS$ is a space
we can view it as an $\infty$-category; its objects are also called the \emph{points} of $\CS$. 

\medskip

For an $\infty$-category $\bC$ and two objects $\bc_0,\bc_1\in \bC$, we let $\Maps_\bC(\bc_0,\bc_1)\in \Spc$ denote the mapping
space between them. 

\medskip

For an object $\bc\in \bC$ we let $\bC_{\bc/}$ (resp., $\bC_{/\bc}$) denote the corresponding under-category (resp., over-category). 

\medskip

In several places in the paper we will need the notion of left (resp., right) Kan extension. 
Let $F:\bC\to \bD$ be a functor, and let $\bE$ is an $\infty$-category with colimits.
Then the functor
\begin{equation} \label{e:restr}
\on{Funct}(\bD,\bE) \overset{\circ F}\longrightarrow \on{Funct}(\bC,\bE)
\end{equation}
admits a left adjoint, called the functor of \emph{left Kan extension} along $F$. 

\medskip

For $\Phi\in \on{Funct}(\bC,\bE)$, the value of its left Kan extension on $\bd\in \bD$ is calculated by the formula
$$\underset{(\bc,F(\bc)\to \bd)\in \bC\underset{\bD}\times \bD_{/\bd}}{\on{colim}}\, \Phi(\bc).$$

\medskip

The notion of \emph{right Kan extension} is obtained similarly: it is the right adjoint of \eqref{e:restr};
the formula for it is given by
$$\underset{(\bc,\bd\to F(\bc))\in \bC\underset{\bD}\times \bD_{\bd/}}{\on{lim}}\, \Phi(\bc).$$

\sssec{DG categories}

We let $\on{DGCat}$ denote the $\infty$-category of DG categories over $E$, see \cite[Vol. 1, Chapter 1, Sect. 10.3.3]{GR2}
(in {\it loc.cit.} it is denoted $\on{DGCat}_{\on{cont}}$). I.e., we will assume all our DG categories to be \emph{cocomplete} and 
we allow only colimit-preserving functors as 1-morphisms. 

\medskip

For example, let $R$ be a DG associative algebra over $k$. Then we let $R\mod$ denote the corresponding DG category
of $R$-modules (i.e., its homotopy category is the usual derived category of the abelian category of $R$-modules, without
any boundedness conditions). 

\medskip

For an algebraic group $H$ over $E$, we let $\Rep(H)$ denote the DG category of representations of $H$, see, e.g., \cite[Sects. 6.4.3-6.4.4]{DrGa}.

\medskip

The piece of structure on $\on{DGCat}$ that we will exploit extensively is the operation of tensor product, which makes 
$\on{DGCat}$ into a symmetric monoidal category.  

\medskip

For a pair of DG associative algebras $R_1$ and $R_2$, we have:
$$(R_1\mod)\otimes (R_2\mod)\simeq (R_1\otimes R_2)\mod.$$

\ssec{Acknowledgements}

The first author like to thank J.~Lurie for numerous helpful discussions related to factorization gerbes. 

\medskip

We are grateful to A.~Beilinson, N.~Rozenblyum and Y.~Zhao for many helpful discussions and comments. 
We are grateful to Ofer Gabber for pointing out a mistake in the previous version in \secref{sss:construct gerbe} and
for his help with the construction in \secref{sss:construct gerbe local}. 

\medskip

We would also like to thank the referee for some very helpful remarks. 

\section{Preliminaries}   \label{s:prel}

This section is included for the reader's convenience: we review some constructions in algebraic
geometry that involve higher category theory. The reader having a basic familiarity with this
material should feel free to skip it. 

\ssec{Some higher algebra}

To facilitate the reader's task, in this subsection we will review some notions from higher algebra that will be used
in this paper. The main reference for this material is \cite{Lu2}.

\medskip

We should emphasize that for the purposes of studying geometric metaplectic data, we only need higher
algebra in $\infty$-categories that are $(n,1)$-categories for small values of $n$. The corresponding objects
can be studied in a hands-on way (i.e., we do not need the full extent of higher category theory).

\medskip

The only place where we really need higher categories is for working with categories of sheaves on prestacks. 

\sssec{Monoids and groups}  

In any $\infty$-category $\bC$ that contains finite products (including the empty finite product, i.e., a final object), it makes
sense to consider the category $\on{Monoid}(\bC)$ of monoid-objects in $\bC$. This
is a full subcategory in the category of \emph{simplicial objects} of $\bC$ (i.e., $\on{Funct}(\bDelta^{\on{op}},\bC)$) that
consists of objects, satisfying the Segal condition. 

\medskip

One defines the category commutative monoids $\on{ComMonoid}(\bC)$ in $\bC$ similarly, but using the category $\on{Fin}_s$
of pointed finite sets instead of $\bDelta^{\on{op}}$.

\medskip 

For example, take $\bC=\infty\on{-Cat}$. In this way we obtain the notion of monoidal (resp., symmetric monoidal) category. 

\sssec{}

The $\infty$-category $\on{Monoid}(\bC)$ (resp., $\on{ComMonoid}(\bC)$) contains the full subcategory of group-like
objects, denoted  $\on{Grp}(\bC)$ (resp., $\on{ComGrp}(\bC)$). 

\medskip

Let $\on{Ptd}(\bC)$ be the category of pointed objects in $\bC$, i.e., $\bC_{*/}$, where $*$ denotes the final object in $\bC$.
We have the loop functor 
$$\Omega:\on{Ptd}(\bC)\to \on{Grp}(\bC), \quad (*\to \bc)\mapsto *\underset{\bc}\times *.$$

The left adjoint of this functor (if it exists) is called the functor of the \emph{classifying space} and is denoted
$$H\mapsto B(H).$$  

\sssec{}

For $\bC=\Spc$ (or $\bC=\on{Funct}(\bD,\Spc)$ for some other category $\bD$), the functor $B$ does exist and
is fully faithful. The essential image of $B:\on{Grp}(\Spc)\to \on{Ptd}(\Spc)$ consists of \emph{connected} spaces. 

\medskip

For an object $\CS\in \on{Ptd}(\Spc)$, its $i$-th homotopy group $\pi_i(\CS)$ is defined to be
$$\pi_0(\Omega^i(\CS)),$$
where $\Omega^i(\CS)$ is viewed as a plain object of $\Spc$. 

\sssec{}  \label{sss:E2}

For $k\geq 0$, we introduce the category $\BE_k(\bC)$ of $\BE_k$-objects in $\bC$ inductively, by setting 
$$\BE_0(\bC)=\on{Ptd}(\bC)$$
and
$$\BE_k(\bC)=\on{Monoid}(\BE_{k-1}(\bC)).$$
Let $\BE_k^{\on{grp-like}}(\bC)\subset \BE_k(\bC)$ the full subcategory of group-like objects, defined to be
the preimage of $$\on{Grp}(\bC)\subset \on{Monoid}(\bC)=\BE_1(\bC)$$ under any of the $k$ possible forgetful functors
$\BE_k(\bC)\to \BE_1(\bC)$. 

\medskip

The functor $B:\on{Grp}(\bC)\to \on{Ptd}(\bC)$ (if it exists) induces a functor 
$$B:\BE_k^{\on{grp-like}}(\bC)\rightleftarrows \BE_{k-1}^{\on{grp-like}}(\bC):\Omega$$
for $k\geq 2$, which is the left adjoint of 
$$\Omega: \BE_{k-1}^{\on{grp-like}}(\bC)\to \BE_k^{\on{grp-like}}(\bC).$$

\medskip

For $i\leq k$ we let $B^i$ denote the resulting functor
$$\BE_k^{\on{grp-like}}(\bC)\to \BE_{k-i}^{\on{grp-like}}(\bC).$$






\sssec{}  \label{sss:conn spectra}

One shows that the forgetful functor
$$\on{Monoid}(\on{ComMonoid}(\bC))\to \on{ComMonoid}(\bC)$$
is an equivalence.

\medskip

This implies that for every $k$ we have a canonically defined functor 
$$\on{ComMonoid}(\bC)\to \BE_k(\bC),$$
and these functors are compatible with the forgetful functors $\BE_k(\bC)\to \BE_{k-1}(\bC)$. Thus, we obtain a canonically defined
functor
\begin{equation} \label{e:Eoo}
\on{ComMonoid}(\bC)\to \BE_\infty(\bC):=\underset{\longleftarrow}{\on{lim}}\, \BE_k(\bC).
\end{equation} 

It is known (see \cite[Remark 5.2.6.26]{Lu2}) that the functor \eqref{e:Eoo} is an equivalence. 

\sssec{}

The category 
$$\on{ComGrp}(\Spc)\simeq \BE^{\on{grp-like}}_\infty(\Spc)$$
identifies with that of connective spectra. 

\medskip

For any $i\geq 0$, we have the mutually adjoint endo-functors
$$B^i:\on{ComGrp}(\Spc)\rightleftarrows \on{ComGrp}(\Spc):\Omega^i$$
with $B^i$ being fully faithful.

\sssec{} \label{sss:E2 on categ} 

Let $A$ be an object of $\BE_2^{\on{grp-like}}(\Spc)$, so that $B(A)$ is an object of $\on{Grp}(\Spc)$. 

\medskip

By an action of $A$ on an $\infty$-category $\bC$ we shall mean an action of $B(A)$ on $\bC$ as an object of 
$\inftyCat$.

\medskip

For example, taking $A=E^\times\in \on{ComGrp}(\Spc)$, we obtain an action of $E^\times$ on any DG category.
Explicitly, we identify $B(E^\times)$ with the space of $E^\times$-torsors, i.e., lines, and the action in question
sends a line $\ell$ to the endofunctor 
$$\bc\mapsto \ell\otimes \bc.$$

\ssec{Prestacks}  \label{ss:prestacks}

\sssec{}

Let $\Sch^{\on{aff}}$ be the category of \emph{classical} affine schemes over $k$. 

\medskip

We let $\on{PreStk}$ denote the category of all 
(accessible) functors
$$(\Sch^{\on{aff}})^{\on{op}}\to \Spc.$$

\medskip

We shall say that an object of $\on{PreStk}$ is $n$-truncated if it takes values in the full subcategory of $\Spc$
that consists of $n$-truncated spaces\footnote{An object of $\Spc$ is said to be truncated,
if for any choice of a base point, its homotopy groups $\pi_{n'}$ vanish for $n'>n$.}.

\medskip

The $\infty$-category of $n$-truncated prestacks is in fact an $(n+1,1)$-category. For small values of $n$, 
one can work with it avoiding the full machinery of higher category theory. 

\begin{rem}

There will be two types of prestacks in this paper: the ``source" type and the ``target" type.  The source type will be various geometric
objects associated to the group $G$ and the curve $X$, such as the Ran space, affine Grassmannian $\Gr_G$, the loop
group $\fL(G)$, etc. These prestacks are $0$-truncated, i.e., they take values in the full 
subcategory
$$\on{Sets}\subset \Spc.$$

\medskip

There will be a few other source prestacks (such as $\Bun_G$ or quotients of $\Gr_G$ by groups acting on it) and they will 
be $1$-truncated  (i.e., they take values in the full subcategory of $\Spc$ spanned by ordinary groupoids). 

\medskip

When we talk about the category of sheaves on a prestack, the prestack in question will be typically of the source type. 

\medskip

The target prestacks will be of the form $B^n(\CA)$ (see \secref{sss:conn spectra}), where $\CA$ is a prestack that takes a 
constant value $A$, where $A$ is a \emph{discrete} abelian group (or its sheafification in, say, the \'etale topology, denoted $B^n_{\on{et}}(\CA)$,
see below). Such a prestack is $n$-truncated. When $n$ is small, they can be described in a 
hands-on way by specifying objects, 1-morphisms, 2-morphisms, etc; in this paper 
$n$ will be $\leq 4$, and in most cases $\leq 2$. 

\medskip

For example, we will often use the notion of a \emph{multiplicative} $A$-gerbe on a group-prestack $\CH$. Such an object is the
same as a map of group-prestacks 
$$\CH\to B^2_{\on{et}}(A).$$

\end{rem} 

\sssec{}

Let $\Sch^{\on{aff}}_{\on{ft}}\subset \Sch^{\on{aff}}$ denote the full subcategory of affine schemes of finite type. Functorially,
thus subcategory can be characterized as consisting of \emph{co-compact} objects, i.e., $S\in \Sch^{\on{aff}}$ if
and only if the functor
$$S'\mapsto \Hom(S',S)$$
commutes with filtered limits. 

\medskip

Moreover, every object of $\Sch^{\on{aff}}$ can be written as a filtered limit of objects of $\Sch^{\on{aff}}_{\on{ft}}$.

\medskip

The two facts mentioned above combine to the statement that we can identify $\Sch^{\on{aff}}$ with the pro-completion
of $\Sch^{\on{aff}}_{\on{ft}}$.

\sssec{}  \label{sss:lft}

We let
$$\on{PreStk}_{\on{lft}}\subset \on{PreStk}$$
denote the full subcategory consisting of functors that preserve filtered colimits. I.e., $\CY\in \on{PreStk}$ is
locally of finite type if for
$$S=\underset{\alpha}{\on{lim}}\, S_\alpha,$$
the map
$$\underset{\alpha}{\on{colim}}\, \Maps(S_\alpha,\CY)\to \Maps(S,\CY)$$
is an isomorphism in $\Spc$.

\medskip

The functors of restriction and left Kan extension along 
\begin{equation} \label{e:aff sch ft}
(\Sch^{\on{aff}}_{\on{ft}})^{\on{op}}\hookrightarrow (\Sch^{\on{aff}})^{\on{op}}
\end{equation}
define an equivalence between $\on{PreStk}_{\on{lft}}$ and the category of all
functors
$$(\Sch^{\on{aff}}_{\on{ft}})^{\on{op}}\to \Spc.$$

\medskip

If $\CF\in \on{PreStk}_{\on{lft}}$ is such that its restriction to $\Sch^{\on{aff}}_{\on{ft}}$ takes 
values in $n$-truncated spaces, then $\CY$ itself is $n$-truncated.

\sssec{}  

In this paper we will work with the \'etale topology on $\Sch^{\on{aff}}$. Let
$$\on{Stk}\subset \on{PreStk}$$
be the full subcategory consisting of objects that satisfy descent for \v{C}ech nerves 
of \'etale morphisms, see \cite[Vol. 1, Chapter 2, Sect. 2.3.1]{GR2}.

\medskip

The inclusion $\on{Stk}\hookrightarrow \on{PreStk}$ admits a left adjoint, called 
the functor of \'etale sheafification, denoted $L_{\on{et}}$. 

\medskip

This functor sends $n$-truncated objects to $n$-truncated objects. 

\sssec{}

Denote 
$$\on{Stk}_{\on{lft}}:=\on{Stk}\cap \on{PreStk}_{\on{lft}}\subset \on{PreStk}.$$

\medskip

However, we can consider a different subcategory of $\on{PreStk}_{\on{lft}}$, denoted
$\on{NearStk}_{\on{lft}}$. Namely,
identifying  $\on{PreStk}_{\on{lft}}$ with $\on{Funct}((\Sch^{\on{aff}}_{\on{ft}})^{\on{op}},\Spc)$,
we can consider the full subcategory consisting of functors that satisfy descent for \v{C}ech
covers of \'etale morphisms (within $\Sch^{\on{aff}}_{\on{ft}}$).

\medskip

Restriction along \eqref{e:aff sch ft} sends $\on{Stk}_{\on{lft}}$ to $\on{NearStk}_{\on{lft}}$.
However, it is not true that the functor of left Kan extension along \eqref{e:aff sch ft} 
sends $\on{NearStk}_{\on{lft}}$ to $\on{Stk}_{\on{lft}}$. However, the following weaker
statement holds (see \cite[Vol. 1, Chapter 2, Proposition 2.7.7]{GR2}): 

\begin{lem}  \label{l:descent lft}
Assume that $\CY\in \on{Funct}((\Sch^{\on{aff}}_{\on{ft}})^{\on{op}},\Spc)$ is $n$-truncated 
for some $n$ and belongs to $\on{NearStk}_{\on{lft}}$. Then the left Kan extension of $\CY$ along \eqref{e:aff sch ft} belongs to
$\on{Stk}_{\on{lft}}$.
\end{lem}

This formally implies:

\begin{cor} \label{c:descent lft}
If $\CY\in \on{PreStk}_{\on{lft}}$ is $n$-truncated for some $n$, then $L_{\on{et}}(\CY)$ 
belongs to $\on{Stk}_{\on{lft}}$.
\end{cor}

\ssec{Gerbes}

\sssec{}

Let $\CY$ be a prestack, and let $\CA$ be a group-like $\BE_n$-object in the category $\on{PreStk}_{/\CY}$, for $n\geq 1$.
In other words, for a given
$(S\overset{y}\to \CY)\in (\Sch^{\on{aff}})_{/\CY}$, the space 
\begin{equation} \label{e:lift to A}
\Maps(S,\CA)\underset{\Maps(S,\CY)}\times \{y\}
\end{equation} 
is a group-like $\BE_n$-object of $\Spc$, in a way functorial in 
$(S,y)$. 

\medskip

We include the case of $n=\infty$, when we stipulate that $\CA$ is a commutative group-object of $\on{PreStk}_{/\CY}$. I.e., 
\eqref{e:lift to A} should be a commutative group-object of $\Spc$, i.e., a connective spectrum. 

\medskip

For any $0\leq i\leq n$, we let $B^i(\CA)$ denote the $i$-fold classifying space of $\CA$. This is a group-like $\BE_{n-i}$-object 
in $\on{PreStk}_{/\CY}$. For $i=1$ we simply write $B(\CA)$ instead of $B^1(\CA)$. 

\sssec{}

We let $B^i_{\on{et},/\CY}(\CA)$ (resp., $B^i_{\on{Zar},/\CY}(\CA)$) denote the \'etale (resp., Zariski) sheafification of $B^i(\CA)$ in 
the category $(\Sch^{\on{aff}})_{/\CY}$ (see \cite[Vol. 1, Chapter 2, Sect. 2.3]{GR2}). 
We will be interested in spaces of the form 
\begin{equation} \label{e:general sect}
\Maps_{/\CY}(\CY,B^i_{\on{et},/\CY}(\CA)),
\end{equation} 
where $\Maps_{/\CY}(-,-)$ is short-hand for $\Maps_{\on{PreStk}_{/\CY}}(-,-)$.

\medskip 

Note that \eqref{e:general sect} is naturally a group-like $\BE_{n-i}$-space (resp., a commutative group object in $\Spc$ if $n=\infty$). 

\sssec{}  \label{sss:constant grp-sch}

In most examples, we will take $\CA$ to be of the form $A\times \CY$, where $A$ is a torsion abelian group, considered
as a constant prestack. In this case
$$\Maps_{/\CY}(\CY,B^i_{\on{et},/\CY}(\CA))\simeq \Maps(\CY,B^i_{\on{et}}(A)).$$

\medskip  

Note that
$$\pi_j\left(\Maps(\CY,B^i_{\on{et}}(A))\right)=
\begin{cases} 
&H^{i-j}_{\on{et}}(\CY,A), \quad j\leq i;\\
&0, \quad j>i.
\end{cases}
$$

Here $H^\bullet_{\on{et}}(\CY,A)$ refers to the \'etale cohomology of $\CY$ with coefficients in $A$. In other words,
it is the cohomology of the object 
$$\on{C}^\bullet_{\on{et}}(\CY,A):=\underset{(S,y)\in \Sch^{\on{aff}}/\CY}{\underset{\longleftarrow}{\on{lim}}}\, \on{C}^\bullet_{\on{et}}(S,A),$$
see \cite[Construction 3.2.1.1]{GL2}. 

\sssec{}

Note also that in this case the functor 
$$S\mapsto \Maps(S,B^i_{\on{et}}(A)), \quad (\Sch^{\on{aff}})^{\on{op}}\to \Spc$$
identifies with the \emph{left Kan extension} of its restriction to $(\Sch^{\on{aff}}_{\on{ft}})^{\on{op}}$. I.e., if an affine scheme $S$ is written as
a filtered limit
$$S=\underset{\alpha}{\underset{\longleftarrow}{\on{lim}}}\, S_\alpha,\quad S_\alpha\in \Sch^{\on{aff}}_{\on{ft}},$$ then the map
$$\underset{\alpha}{\underset{\longrightarrow}{\on{colim}}}\, \Maps(S_\alpha,B^i_{\on{et}}(A))\to \Maps(S,B^i_{\on{et}}(A))$$
is an isomorphism (this latter assertion means that $B^i_{\on{et}}(A)$ is locally of finite type as a prestack), see \corref{c:descent lft}.

\sssec{} 

For $k=1$, the points of the space
\begin{equation} \label{e:general sect 1}
\on{Tors}_\CA(\CY):=\Maps_{/\CY}(\CY,B_{\on{et},/\CY}(\CA))
\end{equation} 
are by definition $\CA$-torsors on $\CY$.

\sssec{}

Our primary interest is the cases of $k=2$. We will call objects of the space
\begin{equation} \label{e:general sect 2}
\on{Ge}_{\CA}(\CY):=\Maps_{/\CY}(\CY,B^2_{\on{et},/\CY}(\CA)).
\end{equation} 
$\CA$-gerbes on $\CY$. 

\medskip

When $\CA$ is of the form $A\times \CY$ (see \secref{sss:constant grp-sch} above), we will simply write $\on{Ge}_A(\CY)$. 









\ssec{Gerbes coming from line bundles}

In this subsection we will be studying gerbes for a constant commutative group-prestack, corresponding
to a torsion abelian group $A$. In what follows, we will be assuming that the orders of elements of $A$ are co-prime to $\on{char}(k)$. 

\sssec{}  \label{sss:ex line bundle}

Let $A(-1)$ denote the group
$$\underset{n\in \BN}{\on{colim}}\, \Hom(\mu_n,A).$$
In the above formula we regard $\BN$ as a poset via
$$n'\geq n \, \Leftrightarrow\, n\mid n',$$
and in forming the above colimit the transition maps are given by 
\begin{equation} \label{e:mu n}
\mu_{n'}\overset{x\mapsto x^{\frac{n'}{n}}}\twoheadrightarrow \mu_n, \quad \text{for }n\mid n'.
\end{equation} 

\medskip

For future reference, denote also 
$$A(1)=\underset{n\in \BN}{\on{colim}}\, \left(\mu_{n'}\underset{\BZ/n'\BZ}\otimes A_{n\on{-tors}}\right),$$
where $A_{n\on{-tors}}\subset A$ is the subgroup of $n$-torsion elements, and in the above formula
$n'$ is any integer divisible by $n$.

\sssec{}  \label{sss:line bundles to gerbes}

We claim that to any line bundle $\CL$ on a prestack $\CY$ and an element $a\in A(-1)$ 
one can canonically associate an $A$-gerbe, denoted $\CL^{a}$, over $\CY$.

\medskip

It suffices to perform this construction for $A=\mu_n$ and $a$ coming from the identity map $\mu_n\to \mu_n$. 
In this case, the corresponding $\mu_n$-gerbe will be denoted $\CL^{\frac{1}{n}}$.

\medskip

By definition, for an affine test scheme $S$ over $\CY$, the value of $\CL^{\frac{1}{n}}$ on $S$ is the groupoid of pairs
$$(\CL',(\CL')^{\otimes n}\simeq \CL|_{S}),$$
where $\CL'$ is a line bundle on $S$. 

\medskip

Note that if $\CL$ admits an $n$-th root $\CL'$, then this $\CL'$ determines a trivialization of $\CL^{\frac{1}{n}}$. 

\begin{rem}  \label{r:roots}
We emphasize the notational difference between the $\mu_n$-gerbe $\CL^{\frac{1}{n}}$, and the line bundle 
$\CL^{\otimes \frac{1}{n}}$, when the latter happens to exist. Namely, a choice of $\CL^{\otimes \frac{1}{n}}$
defines a trivialization of the gerbe $\CL^{\frac{1}{n}}$. 
\end{rem} 

\sssec{} \label{sss:trivialized gerbes}

Let $Y$ be a smooth scheme, and let $Z\subset Y$ be a subvariety of codimension one. 
Let $Z_i,i\in I$ denote the irreducible components of $Z$. 
For every $i$, let $\CO(Z_i)$ denote the corresponding line bundle on $Y$, trivialized away from $Z$.

\medskip
  
We obtain a homomorphism
\begin{equation} \label{e:all triv gerbes}
\Maps(I,A(-1))\to \on{Ge}_{A}(Y)\underset{\on{Ge}_{A}(Y-Z)}\times *, \quad (I\mapsto a_i)\rightsquigarrow \underset{i}\bigotimes\, \CO(Z_i)^{a_i}
\end{equation}

\begin{lem} \label{l:trivialized gerbes}
Assume that the orders of elements in $A$ 
are \emph{prime} to $\on{char}(k)$, i.e., that $A$ has no $p$-torsion, where $p=\on{char}(k)$. 
Then the map \eqref{e:all triv gerbes} is an isomorphism in $\Spc$. 
\end{lem}

\begin{proof}
The assertion follows from the fact that the \'etale cohomology group $H^i_{\on{et},Z}(Y,A)$ identifies with $\Maps(I,A(-1))$ 
for $i=2$ and vanishes for $i=1,0$.
\end{proof} 

\ssec{The sheaf-theoretic context}  \label{ss:sheaves}

Most of this paper is devoted to the discussion of gerbes. However, in the last two sections, we will apply this
discussion in order to formulate metaplectic geometric Satake. The latter involves \emph{sheaves} and
more generally \emph{sheaves of categories} on various geometric objects. 

\medskip

When discussing sheaves (and sheaves of categories) we will only need to consider algebro-geometric objects that are 
\emph{locally of finite type}, i.e., prestacks that belong to $\on{PreStk}_{\on{lft}}$, see \secref{sss:lft}. 
In what follows, in order to simplify the notation, we will omit the subscripts $\on{ft}$ and $\on{lft}$. 

\sssec{}

There are several possible sheaf-theoretic contexts (for schemes of finite type): 

\medskip

\noindent(a) For any ground field $k$ we can consider the ind-completion of the constructible 
derived category of $\ell$-adic sheaves (viewed as a DG category);

\medskip

\noindent(b) When the ground field is $\BC$, then for an arbitrary algebraically closed field $E$ of characteristic $0$, we 
can consider the ind-completion of the constructible derived category of $E$-vector spaces (viewed as a DG category);

\medskip

\noindent(c) When the ground field $k$ has characteristic $0$, we can consider the ind-completion of the derived category of
holonomic D-modules (viewed as a DG category); 

\medskip

\noindent(c') In the setting of (c), we consider can consider the derived category of \emph{all} D-modules (viewed as a DG category). 

\medskip

We will refer to contexts (a), (b) and (c) as \emph{constructible}. 

\medskip

We will denote by $E$ the field of coefficients of our sheaves. So, in the above cases, this is $\ol\BQ_\ell$, $E$ and $k$,
respectively. For $Y\in \Sch$ we will denote by 
$$\Shv(Y)$$
the resulting category of sheaves. This is an $E$-linear DG category. 

\medskip
 
Note that in cases (c) and (c'), $E=k$. 
Yet, we will keep the notational distinction between $k$ (the ground field) and $E$ (the field of coefficients) even in this case. 

\medskip

When we need to emphasize the distinction between case (c') and the other cases, we will use the notation
$$\Dmod(Y).$$

\sssec{}

We will denote by 
\begin{equation} \label{e:sch ft}
\Shv:(\Sch^{\on{aff}})^{\on{op}}\to \DGCat
\end{equation} 
the functor constructed  that associates to an affine scheme $S$ of finite type the DG category $\Shv(S)$
and to a morphism $f:S_1\to S_2$
the functor 
$$f^!:\Shv(S_2)\to \Shv(S_1).$$

\medskip

A basic feature of this functor is that the functor
\eqref{e:sch ft} carries a natural \emph{lax symmetric monoidal structure}. In particular, for $S_1,S_2\in \Sch^{\on{aff}}$
we have a (fully faithful) functor 
\begin{equation} \label{e:extrnl ten prod}
\Shv(S_1)\otimes \Shv(S_2)\to \Shv(S_1\times S_2).
\end{equation} 

When $\Shv(-)=\Dmod(-)$, i.e., in context (c'), this lax structure is \emph{strict}, i.e., the functor \eqref{e:extrnl ten prod} is an
equivalence. This is \emph{emphatically not the case} in the constructible contexts. 

\sssec{}  

Yoneda embedding is a fully faithful functor
$$\Sch^{\on{aff}}\hookrightarrow \on{PreStk}.$$

The right Kan extension of $\Shv$ along the (opposite of the) Yoneda embedding $(\Sch^{\on{aff}})^{\on{op}}\to (\on{PreStk})^{\on{op}}$ 
defines a functor
$$\Shv:(\on{PreStk})^{\on{op}}\to \DGCat.$$

\medskip

Thus, if $\CY\in \on{PreStk}$ is written as
$$\CY=\underset{i}{\underset{\longrightarrow}{\on{colim}}}\, S_i, \quad S_i\in \Sch^{\on{aff}},$$ 
we have by definition
$$\Shv(\CY)=\underset{i}{\underset{\longleftarrow}{\on{lim}}}\, \Shv(S_i).$$













\ssec{Sheaves of categories} \label{ss:shv-of-cat}

Sheaves of categories appear in this paper as a language in which we formulate the metaplectic geometric Satake functor.
The reader can skip this subsection on the first pass, and return to it when necessary. 

\medskip

In this subsection we take our sheaf theory to be that of all D-modules, i.e., $\Shv(-)=\Dmod(-)$. 

\medskip

The discussion in this section is essentially borrowed from \cite{Ga1}. 

\sssec{}

Note that the diagonal morphism for affine schemes defines on every object of $\Sch^{\on{aff}}$ a canonical structure 
of co-commutative co-algebra. 

\medskip

Hence, the symmetric monoidal structure on $\Shv$ (see \cite[Vol. 2, Chapter 3, Corollary 6.1.2]{GR2}) naturally gives rise to a functor
$$(\Sch^{\on{aff}})^{\on{op}}\to \on{ComAlg}(\DGCat)=:\DGCat^{\on{SymMon}}.$$

\medskip

In particular, for every $S\in \Sch^{\on{aff}}$, the category $\Shv(S)$ has a natural symmetric monoidal structure, and for every $f:S_1\to S_2$,
the functor $f^!:\Shv(S_2)\to \Shv(S_1)$ is symmetric monoidal. 

\sssec{}

By a sheaf of DG categories $\CC$ over $\CY\in \on{PreStk}$ we will mean a functorial assignment
\begin{equation} \label{e:shv of cat data}
(S\overset{y}\to \CY)\in ((\Sch^{\on{aff}})_{/\CY})^{\on{op}} \rightsquigarrow \CC(S,y)\in \Shv(S)\mmod,
\end{equation} 
where $\Shv(S)\mmod$ denotes the category of modules in the (symmetric) monoidal category $\DGCat$ for the
(commutative) algebra object $\Shv(S)$. We impose the following quasi-coherence condition:

\medskip

For a morphism of affine schemes $f:S_1\to S_2$, 
$y_2:S_2\to \CY$ and $y_1=y_2\circ f$, consider the corresponding functor
\begin{equation} \label{e:pullback shv of cat}
\CC(S_2,y_2)\to \CC(S_1,y_1).
\end{equation} 

Part of the data of \eqref{e:shv of cat data} is that the functor \eqref{e:pullback shv of cat} should be
$\Shv(S_2)$-linear. Hence, it gives rise to a functor of $\Shv(S_1)$-module categories
\begin{equation} \label{e:tensor shv of cat}
\Shv(S_1)\underset{\Shv(S_2)}\otimes \CC(S_2,y_2)\to \CC(S_1,y_1),
\end{equation}
where $\otimes$ is the operation of tensor product of DG categories (see, e.g., \cite[Vol. 1, Chapter 1, Sect. 10.4]{GR2}).

\medskip

We require that \eqref{e:tensor shv of cat} should be an isomorphism.

\begin{rem}
What we defined as a sheaf of categories over $\CY$ would in the 
language of \cite{Ga1} be rather called a \emph{crystals of categories}. More precisely, \cite[Theorem 2.6.3]{Ga1} 
guarantees that our notion of a sheaf of categories over $\CY$ coincides with the notion of a sheaf of categories 
over $\CY_{\on{dR}}$ in the terminology of \cite{Ga1}.
\end{rem} 

\sssec{} \label{sss:ex prsheaf of categ}

A basic example of a sheaf of categories is denoted $\Shv_{/\CY}$; it is defined by setting
$$\Shv_{/\CY}(S,y):=\Shv(S).$$

\medskip

Let $Z$ be a prestack locally of finite type over $\CY$. We define a sheaf of categories $\Shv(Z)_{/\CY}$ over $\CY$ by setting for
$S\overset{y}\to \CY$,
$$\Shv(Z)_{/\CY}(S,y)=\Shv(S\underset{\CY}\times Z).$$

\medskip

The fact that for $f:S_1\to S_2$, the functor
$$\Shv(S_1)\underset{\Shv(S_2)}\otimes \Shv(S_2\underset{\CY}\times Z)\to \Shv(S_1\underset{\CY}\times Z)$$
is an equivalence follows from \cite[Theorem 2.6.3]{Ga1}. 

\sssec{Descent}

Forgetting the module structure, a sheaf of DG categories $\CC$ over $\CY$ defines a functor
\begin{equation} \label{e:preshv of cat as funct}
((\Sch^{\on{aff}})_{/\CY})^{\on{op}}\to \DGCat.
\end{equation} 

It follows from \cite[Corollary 1.5.2]{Ga1} that the assignment \eqref{e:preshv of cat as funct}
satisfies \'etale descent (in fact, it satisfies h-descent). 

\sssec{}  \label{sss:category prestack}  \label{sss:global sections}

Applying to the functor \eqref{e:preshv of cat as funct} the procedure of right Kan extension along
$$((\Sch^{\on{aff}})_{/\CY})^{\on{op}}\to ((\on{PreStk})_{/\CY})^{\on{op}},$$
we obtain that for every prestack $Z$ over $\CY$
there is a well-defined DG category $\CC(Z)$. 

\medskip

Namely, if 
$$Z\simeq \underset{i}{\underset{\longrightarrow}{\on{colim}}}\, S_i, \quad (S_i,y_i)\in (\Sch^{\on{aff}})_{/\CY},$$
then 
$$\CC(Z)=\underset{i}{\underset{\longleftarrow}{\on{lim}}}\, \CC(S_i,y_i).$$

We will refer to $\CC(Z)$ as the ``category of sections of $\CC$ over $Z$".
By construction the DG category $\CC(Z)$ is naturally an object of $\Shv(Z)\mmod$.

\medskip

When $Z$ is $\CY$ itself, we will refer to $\CC(\CY)$ as the ``category of global sections of $\CC$". 

\sssec{Example}
  
For $\CC=\Shv(Z)_{/Y}$ as in \secref{sss:ex prsheaf of categ}, we have
$$\CC(\CY)\simeq\Shv(Z).$$

\sssec{}

The construction in \secref{sss:global sections} defines a functor
\begin{equation} \label{e:bGamma}
\{\text{Sheaves of categories over }\CY\}\to \Shv(\CY)\mmod.
\end{equation}

The functor \eqref{e:bGamma} admits a left adjoint given by sending
\begin{equation} \label{e:bLoc}
\CC\rightsquigarrow \left((S\to \CY)\mapsto \Shv(S)\underset{\Shv(\CY)}\otimes \CC\right).
\end{equation}

\medskip

We have the following assertion from \cite[Theorem 2.6.3]{Ga1} states:

\begin{thm}   \label{t:1-aff}
For $\CY$ that is an ind-scheme of ind-finite type, the mutually adjoint functors \eqref{e:bGamma} and \eqref{e:bLoc}
are equivalences.
\end{thm} 

\begin{rem}
For the purposes of the present paper one can make do avoiding (the somewhat non-trivial) \thmref{t:1-aff}.
However, allowing ourselves to use it simplifies a lot of discussions related to sheaves of categories. 
\end{rem} 

\ssec{Some twisting constructions}  \label{ss:twist category}

The material in this subsection may not have proper references in the literature, so we provide
some details. The reader is advised to skip it and return to it when necessary. 

\medskip

When discussing sheaves of categories, we will be assuming that $\Shv(-)=\Dmod(-)$. 

\sssec{Twisting by a torsor}  \label{sss:twist sheaf by torsor}

Let $\CY$ be a prestack, and let $\CH$ (resp., $\CF$) a group-like object in $\on{PreStk}_{/\CY}$ (resp., an object in 
$\on{PreStk}_{/\CY}$, equipped with an action of $\CH$). In other words, these are functorial assignments
$$(S,y)\in (\Sch^{\on{aff}})_{/\CY}\rightsquigarrow \CH(S,y)\in \on{Grp}(\Spc), \quad 
(S,y)\in (\Sch^{\on{aff}})_{/\CY}\rightsquigarrow \CF(S,y)\in \Spc,$$
and an action of $\CH(S,y)$ on $\CF(S,y)$. 

\medskip

Let $\CT$ be an $\CH$-torsor on $\CY$. In this case, we can form a $\CT$-twist of $\CF$, denoted
$\CF_{\CT}$, and which is an \'etale \emph{sheaf}. Here is the construction\footnote{Note that when $\CT$ is the trivial torsor, 
the output of this construction is the \'etale sheafification of $\CF$.}:

\medskip

Consider the subcategory $\on{Split}(\CT)\subset (\Sch^{\on{aff}})_{/\CY}$ formed by $(S,y)\in (\Sch^{\on{aff}})_{/\CY}$ for which the torsor $\CT|_S$ 
\emph{admits a splitting}. This subcategory forms a basis of the \'etale topology, so it is sufficient to specify the restriction of $\CF_{\CT}$ to $\on{Split}(\CT)$. 

\medskip

The sought-for functor $\CF_{\CT}|_{\on{Split}(\CT)}$ is given by sending $(S,y)$ to
\begin{equation} \label{e:twist torsor}
\left(*\underset{\Maps_{/\CY}(S,B_{\on{et}}(\CH))}\times *\right)\overset{\Maps_{/\CY}(S,\CH)}\times \CF,
\end{equation}
where the two maps
$$*\to \Maps_{/\CY}(S,B_{\on{et}}(\CH))\leftarrow *$$
are (i) the trivial map, and (ii) the one given by the composition
$$S\to \CY \overset{\CT}\to  B_{\on{et}}(\CH).$$ 

Note that $$*\underset{\Maps_{/\CY}(S,B_{\on{et}}(\CH))}\times *$$
is a groupoid equipped with a simply-transitive action of the group $\Maps_{/\CY}(S,\CH)$. In formula
\eqref{e:twist torsor}, the notation $\overset{\Maps_{/\CY}(S,\CH)}\times$ means ``divide by the diagonal
action of $\Maps_{/\CY}(S,\CH)$".

\sssec{A twist of a sheaf of categories by a gerbe} \label{sss:twist sheaves of categ}

Let now $\CC$ be a sheaf of DG categories over $\CY$, and let $\CA$ be a group-like $\BE_2$-object 
in $(\on{PreStk}_{\on{lft}})_{/\CY}$. 

\medskip

Let us be given an action of $\CA$ on $\CC$. In other words, we are given a functorial
assignment for every $(S,y)\in (\Sch^{\on{aff}}_{\on{ft}})_{/\CY}$ of an action of $\CA(S,y)$ on $\CC(S,y)$,
see \secref{sss:E2 on categ}. 

\medskip

Let $\CG$ be an \'etale $\CA$-gerbe on $\CY$. 
Repeating the construction of \secref{sss:twist sheaf by torsor}, we obtain
that we can form the twist $\CC_\CG$ of $\CC$ by $\CG$, which is a new sheaf of DG categories over $\CY$. 

\medskip

In more detail, for $(S,y)\in (\Sch^{\on{aff}}_{\on{ft}})_{/\CY}$ such that $\CG|_S$ \emph{admits a splitting}, we define the
value of $\CC_\CG$ on $(S,y)$ to be
$$\left(*\underset{\Maps_{/\CY}(S,B^2_{\on{et}}(\CA))}\times *\right) 
\overset{\Maps_{/\CY}(S,B_{\on{et}}(\CA))}\times \CC(S,y).$$

\medskip

Concretely, for every $(S\overset{y}\to \CY)\in (\Sch^{\on{aff}}_{\on{ft}})_{/\CY}$ and a trivialization of $\CG|_S$ we have an identification
$$\CC_\CG(S,y)\simeq \CC(S,y).$$
The effect of change of trivialization by a point $a\in B_{\on{et}}(\CA)(S,y)$ has the effect of action of 
$$a\in \on{Funct}(\CC(S,y),\CC(S,y)).$$ 

\sssec{}  \label{sss:twist sheaves of categ 1}

Let $A$ be a torsion subgroup of $E^\times$.

\medskip

Let us take $\CA$ to be the constant group-prestack $\CY\times A$. In this case, the 
embedding $A\to E^{\times}$ gives rise to an action of $\CA$ on \emph{any} sheaf of DG
categories. 

\medskip

Thus, for every
$\CG\in \on{Ge}_A(\CY)$ and any sheaf of categories  $\CC$ over $\CY$, we can form its twisted
version $\CC_\CG$. 

\sssec{The category of sheaves twisted by a gerbe}  \label{sss:twist category sheaves}

Let $A$ and $\CG$ be as in \secref{sss:twist sheaves of categ 1}. 

\medskip

We apply the above construction to $\CC:=\Shv_{/\CY}$. Thus, for any $(S,y)\in (\Sch_{\on{ft}}^{\on{aff}})_{/\CY}$
we have the twisted version of the category $\Shv(S)$, denoted $\Shv_\CG(S)$. 

\medskip

As in \secref{sss:category prestack}, the procedure of Kan extension defines the category
$$\Shv_\CG(Z)$$
for any $Z\in \on{PreStk}_{/\CY}$. 

\section{Factorization gerbes on the affine Grassmannian}   \label{s:Gr}

In this section we introduce our main object of study: factorization gerbes on the affine Grassmannian,
which we stipulate to be the parameters for the metaplectic Langlands theory. 

\ssec{The Ran space}

The Ran space of a curve $X$ is an algebro-geometric device (first suggested in \cite{BD1}) that allows us to talk
about \emph{factorization structures} relative to our curve. 

\sssec{}

Let $X$ be a fixed smooth algebraic curve. We let $\Ran\in \on{PreStk}$ be the Ran space of $X$. By definition,
for an affine test scheme $S$, the space $\Maps(S,\Ran)$ is discrete (i.e., is a set), and equals the set of finite non-empty subsets
of the (set) $\Maps(S,X)$. 

\medskip

For a finite set $J$ we have a map
\begin{equation} \label{e:union map}
\Ran^J\to \Ran
\end{equation} 
given by the union of the corresponding finite subsets.

\medskip

This operation makes $\Ran$ into a (non-unital) semi-group object in $\on{PreStk}_{\on{lft}}$ (see \cite[Definition 5.4.1.1]{Lu2} for what this means). 

\sssec{}  \label{sss:Ran expl}

The Ran space admits the following explicit description as a colimit (as an object of $\on{PreStk}$):
$$\Ran=\underset{I}{\underset{\longrightarrow}{\on{colim}}}\, X^I,$$
where $I$ runs through the category opposite to that of non-empty finite
sets and surjective maps\footnote{We note that this category is \emph{not filtered}, and hence $\Ran$ is
\emph{not} an ind-scheme.}. For a surjection $\phi:I_1\to I_2$, the
corresponding map
$X^{I_2}\to X^{I_1}$ is the corresponding diagonal morphism, denoted
$\Delta_\phi$. 

\medskip

This presentation makes it manifest that $\Ran\in \on{PreStk}_{\on{lft}}$.

\sssec{}

We denote by
$$(\Ran\times \Ran)_{\on{disj}}\subset \Ran\times \Ran$$
the open substack corresponding to the following condition:

\medskip

For an affine test scheme $S$, and two points
$$I_1,I_2\in \Maps(S,\Ran),$$
the point $I_1\times I_2\in \Maps(S,\Ran\times \Ran)$ belongs to $(\Ran\times \Ran)_{\on{disj}}$
if the corresponding subsets
$$I_1,I_2\subset \Maps(S,X)$$
satisfy the following condition: for every $i_1\in I_1,i_2\in I_2$, the corresponding two maps $S\rightrightarrows X$ have non-intersecting images. 

\sssec{}

We give a similar definition for any power: for a finite set $J$ we let
$$\Ran^J_{\on{disj}}\subset \Ran^J$$
be the open substack corresponding to the following condition: 

\medskip

An $S$-point of $\Ran^J$, given by 
$$I_j\subset \Maps(S,X), \quad j\in J$$
belongs to $\Ran^J_{\on{disj}}$ if for every $j_1\neq j_2$ and $i_1\in I_{j_1}$, $i_2\in I_{j_2}$,
the corresponding two maps $S\rightrightarrows X$ have non-intersecting images. 

\ssec{Factorization patterns over the Ran space}

Let $Z$ be a prestack over $\Ran$. At the level of $k$-points, a factorization structure on $Z$ is the following system of isomorphisms:

\medskip

For a $k$-point $\ul{x}$ of $\Ran$ corresponding to a finite set $x_1,...,x_n$ of $k$-points of $X$, 
the fiber $Z_{\ul{x}}$ of $Z$ over the above point is supposed to be identified with
$$\underset{i}\prod\, Z_{\{x_i\}},$$
where $\{x_i\}$ are the corresponding singleton points of $\Ran$.

\medskip

We will now spell this idea, and some related notions, more precisely. 

\sssec{}

By a factorization structure on $Z$ we shall mean an assignment 
for any finite set $J$ of an 
isomorphism
\begin{equation} \label{e:factorization abs}
Z^J\underset{\Ran^J}\times \Ran^J_{\on{disj}} \overset{\gamma_J}\simeq Z\underset{\Ran}\times \Ran^J_{\on{disj}},
\end{equation}
where the morphism $\Ran^J\to \Ran$ is given by \eqref{e:union map}.

\medskip

We require the isomorphisms \eqref{e:factorization abs} to be compatible with surjections of finite sets
in the sense that for $I\overset{\phi}\twoheadrightarrow J$ the diagram

\begin{equation} \label{e:refinement}
\CD
Z^{I}\underset{\Ran^{I}}\times \Ran^{I}_{\on{disj}}  @>{\gamma_{I}}>>  Z\underset{\Ran}\times \Ran^{I}_{\on{disj}}  \\
@V{\sim}VV      @AA{\sim}A  \\
\left(\underset{j\in J}\prod\, Z^{I_j}\underset{\Ran^{I_j}}\times \Ran^{I_j}_{\on{disj}}\right)
\underset{\underset{j\in J}\prod\, \Ran^{I_j}_{\on{disj}}}\times \Ran^{I}_{\on{disj}}   &  & 
(Z\underset{\Ran}\times \Ran^{J}_{\on{disj}}) \underset{\Ran^{J}_{\on{disj}}} \times \Ran^{I}_{\on{disj}}   \\
@V{\underset{j\in J}\prod\, \gamma_{I_j}}VV   @AA{\gamma_{J}}A   \\
\left(\underset{j\in J}\prod\, Z\underset{\Ran}\times \Ran^{I_j}_{\on{disj}}\right)
\underset{\underset{j\in J}\prod\, \Ran^{I_j}_{\on{disj}}}\times \Ran^{I}_{\on{disj}}   & &  
(Z^{J}\underset{\Ran^{J}}\times \Ran^{J}_{\on{disj}}) \underset{\Ran^{J}_{\on{disj}}}\times \Ran^{I}_{\on{disj}}  \\
@V{\sim}VV   @A{\sim}AA  \\
\left(Z^{J} \underset{\Ran^{J}}\times \underset{j\in J}\prod\, \Ran^{I_j}_{\on{disj}}\right)
\underset{\underset{j\in J}\prod\, \Ran^{I_j}_{\on{disj}}}\times \Ran^{I}_{\on{disj}}  
@>{\sim}>>  Z^{J}\underset{\Ran^{J}}\times \Ran^{I}_{\on{disj}}, 
\endCD
\end{equation}
where $I_j:=\phi^{-1}(j)$, is required to commute. Furthermore, if $Z$ takes values in $\infty$-groupoids
(rather than sets), we require a homotopy-coherent system of compatibilities for higher order compositions,
see \cite[Sect. 6]{Ras1}.


\sssec{}  \label{sss:factorization}

Let $\CC$ be a sheaf of DG categories over $\Ran$ (recall that this means that we are working over a ground field
of characteristic $0$ and in the context of D-modules). 

\medskip

By a \emph{factorization structure} on $\CC$ we shall mean a functorial assignment 
for any finite set $J$ and an $S$-point of $\Ran^J_{\on{disj}}$, given by 
$$I_j\subset \Maps(S,X), \quad j\in J$$
of an identification 
\begin{equation} \label{e:factorization abs categ}
\underset{j,\Shv(S)}\bigotimes\, \CC(S,I_j) \simeq \CC(S,I),
\end{equation} 
where $I=\underset{j\in J}\sqcup\, I_j$. 

We require the functors \eqref{e:factorization abs categ} to be compatible with surjections $J_1\twoheadrightarrow J_2$ 
via the commutative diagrams analogous to \eqref{e:refinement}. A precise formulation of these compatibilities is given in
\cite[Sect. 6]{Ras1}.

\sssec{}    \label{sss:fact from Z}

Let $Z$ be a factorization prestack over $\Ran$.  Assume that for every finite set $I$, the category
$\Shv(X^I\underset{\Ran}\times Z)$ is dualizable. We claim that in this case the sheaf of categories $\Shv(Z)_{/\Ran}$, i.e., 
$$(S,I\subset \Maps(S,X))\rightsquigarrow  \Shv(S\underset{\Ran}\times Z),$$
has a natural factorization structure. 

\medskip

Indeed, for any $Z$ we have a canonically defined system of functors
$$\underset{j,\Shv(S)}\bigotimes\, \Shv(S\underset{I_j,\Ran}\times Z) \to \Shv\left(\underset{j,S}\Pi\, (S\underset{I_j,\Ran}\times Z)\right)=
\Shv(S\underset{\Ran^J}\times Z^J) \overset{\text{\eqref{e:factorization abs}}}\simeq \Shv(S \underset{I,\Ran}\times Z)$$
for a map $S\to \Ran^J_{\on{disj}}$. We claim that the first arrow is an equivalence if each  
$\Shv(X^I\underset{\Ran}\times Z)$ is dualizable. 

\medskip

To prove this, it suffices to can consider the universal case when 
$$S:=X^I_{\on{disj}}:=X^I\underset{\Ran^J}\times \Ran^J_{\on{disj}}$$
for a finite set $I$ and a surjection $I\twoheadrightarrow J$. 

\medskip

We have
\begin{multline*}
\underset{j,\Shv(S)}\bigotimes\, \Shv(S\underset{I_j,\Ran}\times Z)\simeq
\left(\underset{j\in J}\bigotimes\, \Shv(X^{I_j}\underset{\Ran}\times Z)\right)\underset{\Shv(X^I)}\otimes \Shv(X^I_{\on{disj}}) \to \\
\to \Shv\left(\underset{j\in J}\Pi\, (X^{I_j}\underset{\Ran}\times Z)\right) \underset{\Shv(X^I)}\otimes \Shv(X^I_{\on{disj}})\simeq
 \Shv\left(\left(\underset{j\in J}\Pi\, (X^{I_j}\underset{\Ran}\times Z)\right)\underset{X^I}\times X^I_{\on{disj}}\right)=
\Shv(S\underset{\Ran^J}\times Z^J),
\end{multline*}
where the second arrow is an isomorphism due to the assumption that the categories $\Shv(X^{I_j}\underset{\Ran}\times Z)$
are dualizable.

\sssec{}

Let $Z$ be a factorization prestack over $\Ran$, and let $A$ be a torsion abelian group. 
Let $\CG$ be an $A$-gerbe on $Z$. By a factorization structure 
on $\CG$ we shall mean a system of identifications 
\begin{equation} \label{e:factorization abs gerbe}
\CG^{\boxtimes J}|_{Z^J\underset{\Ran^J}\times \Ran^J_{\on{disj}}}\simeq \CG|_{Z\underset{\Ran}\times \Ran^J_{\on{disj}}},
\end{equation} 
where the underlying spaces are identified via \eqref{e:factorization abs}. 

\medskip

The identifications \eqref{e:factorization abs gerbe} are required to be compatible with surjections $J_1\twoheadrightarrow J_2$ 
via the commutative diagrams \eqref{e:refinement}.
Note that since gerbes form a 2-groupoid, we only need to specify the datum of \eqref{e:factorization abs gerbe} up to $|J|=3$,
and check the relations up to $|J|=4$.  

\medskip

Factorization gerbes over $Z$ naturally form a space (in fact, a 2-groupoid), equipped with a structure of commutative group in $\Spc$
(i.e., connective spectrum), to be denoted $\on{FactGe}_{A}(Z)$. 

\begin{rem}
Note that the diagrams \eqref{e:refinement} include those corresponding to automorphisms of finite sets. I.e., the datum of
factorization gerbe includes equivariance with respect to the action of the symmetric group. For this reason what we call
``factorization gerbe" in \cite{Re} was called ``symmetric factorizable gerbe".
\end{rem} 

\sssec{Variant} \label{sss:variant factor}

Let $Z$ be a factorization prestack over $\Ran$, and let $\CG$ be a factorization $A$-gerbe over it for $A\subset E^\times$. 
Assume that for every finite set $I$, the category $\Shv_\CG(X^I\underset{\Ran}\times Z)$ is dualizable. Then 
the sheaf of categories $\Shv_\CG(Z)_{/\Ran}$ defined by 
$$(S,I\subset \Maps(S,X))\rightsquigarrow \Shv_\CG(S\underset{\Ran}\times Z)$$
has a natural factorization structure. 

\sssec{}

By a similar token, we can consider factorization line bundles over factorization prestacks, and also $\BZ$- or $\BZ/2\BZ$-graded 
line bundles\footnote{Note that in the latter case, the compatibility involved in the factorization structure (arising from the diagrams 
\eqref{e:refinement} for automorphisms
of finite sets $J$) involves \emph{sign rules}. I.e., a factorization $\BZ/2\BZ$-graded line bundle \emph{does not} give rise to a factorization line bundle
by forgetting the grading.}. 

\medskip

If $\CL$ is a (usual, i.e., not graded) factorization line bundle and $a\in A(-1)$, we obtain a factorization gerbe $\CL^{a}$.

\ssec{The Ran version of the affine Grassmannian}

In this subsection we introduce the Ran version of the affine Grassmannian, which 
plays a crucial role in the geometric Langlands theory. 

\sssec{}  \label{sss:U}

For an algebraic group $G$, we define the Ran version of the affine Grassmannian of $G$,
denoted $\Gr_G$, to be the following prestack. 

\medskip

For an affine test scheme $S$, the groupoid (in fact, set) $\Maps(S,\Gr_G)$ consists of triples
$$(I,\CP_G,\alpha),$$
where $I$ is an $S$-point of $\Ran$, $\CP_G$ is a $G$-bundle on $S\times X$, and $\alpha$ is a trivialization
of $\CP_G$ over the open subset $U_I\subset S\times X$ equal to the complement of the union $\Gamma_I$ of the graphs of the maps $S\to X$
corresponding to the elements of $I\subset  \Maps(S,X)$. 

\sssec{}

It is known that for every finite set $I$, the prestack $X^I\underset{\Ran}\times \Gr_G$ is an ind-scheme of ind-finite type.
This implies, in partircular, that the dualizability assumptions in Sects. \ref{sss:fact from Z} and \ref{sss:variant factor} are satisfied. 

\sssec{}  

The basic feature of the prestack $\Gr_G$ is that it admits a natural factorization structure over $\Ran$, obtained
by gluing bundles. 

\medskip

Hence, for a torsion abelian group $A$, it makes sense to talk about factorization $A$-gerbes over $\Gr_G$. We denote the 
the resulting space (i.e., in fact, a connective $2$-truncated spectrum) by $$\on{FactGe}_{A}(\Gr_G).$$

\sssec{An example}  \label{sss:line bundles Grass}

Let $\CL$ be a factorization line bundle on $\Gr_G$, and let $a$ be an element of $A(-1)$. Then the $A$-gerbe 
$$\CL^{a}$$
of \secref{sss:ex line bundle}
is naturally a factorization gerbe on $\Gr_G$. 

\medskip

This example is important because there is a canonical factorization line bundle on $\Gr_G$, denoted $\det_\fg$; we will
encounter it in \secref{sss:det}. 

\sssec{}  \label{sss:global}

Recall that we are not assuming that $X$ be complete. Let $\ol{X}$ be its compactification. Let $D\subset \ol{X}$ be the
complementary divisor. Let $\Bun_G(\ol{X};D)$ be the moduli stack of $G$-bundles on $\ol{X}$ equipped with a trivialization
along $D$. (When $X$ is complete, we have $\Bun_G(\ol{X};D)$ is the usual stack $\Bun_G(X)$ classifying $G$-bundles on $X$.

\medskip

We have a naturally defined map 
\begin{equation} \label{e:to BunG}
\Gr_G\to \Bun_G(\ol{X};D).
\end{equation} 

\medskip

Recall now that \cite[Theorem 3.2.13]{GL2} says\footnote{This assertion was proved in {\it loc.cit.} under the additional
assumption that $G$ be semi-simple and simply connected. However, in the case of constant groups-schemes, the 
statement is known to hold in general: see \cite[Theorem 4.1.6]{Ga3}.} that the map \eqref{e:to BunG} is a \emph{universal homological equivalence}. 
This implies that any gerbe on $\Gr_G$ uniquely descends to a gerbe on $\Bun_G(\ol{X};D)$. 

\medskip

In particular, this is the case for factorization gerbes. 

\ssec{The space of geometric metaplectic data}

\sssec{}  \label{sss:on Gr}

Let $E^{\times,\on{tors}}$ denote the  group of roots of unity in $E$ of orders prime to  $\on{char}(k)$. 

\medskip

We stipulate that the space 
$$\on{FactGe}_{E^{\times,\on{tors}}}(\Gr_G)$$
is the space of parameters for the metaplectic Langlands theory. We also refer to it as the space of \emph{geometric metaplectic data}. 

\medskip

This includes both the global case (when $X$ is complete), and the local case when we take $X$ to be a Zariski neighborhood
of some point $x$. 

\sssec{}

Given an $E^{\times,\on{tors}}$-factorization gerbe $\CG$ on $\Gr_G$, we can thus talk about the factorization sheaf of categories,
denoted $$\Shv_\CG(\Gr_G)_{/\Ran},$$
whose value on $S,I\subset \Maps(S,X)$ is
$$\Shv_\CG(S\underset{\Ran}\times \Gr_G).$$

\section{Parameterization of factorization gerbes}  \label{s:param}

From now on we let $A$ be a torsion abelian group whose elements have orders prime to $\on{char}(k)$.
The main example is $A=E^{\times,\on{tors}}$. 

\medskip

The goal of this section is to describe the set of isomorphism classes (and, more ambitiously, the \emph{space})
of $A$-factorization gerbes on $\Gr_G$ in terms of more concise algebro-geometric objects. 

\ssec{Parameterization via \'etale cohomology}  \label{ss:param 1}

In this subsection we will create a space, provided by the theory of \'etale cohomology, that maps to the space   
$\on{FactGe}_A(\Gr_G)$, thereby giving a parameterization of geometric metaplectic data. 

\sssec{}




\medskip

Let $B_{\on{et}}(G)=:\on{pt}/G$ be the stack of $G$-torsors. I.e., this is the sheafification in the \'etale
topology of the prestack $B(G) $ that attaches to an affine test scheme $S$ the groupoid 
$$*/\Maps(S,G).$$

\sssec{}

Consider the space of maps
$$\Maps_{\on{Ptd}(\on{PreStk})}(B_{\on{et}}(G) \times X,B^4_{\on{et}}(A(1))),$$
where the subscript $\on{Ptd}$ stands for the the space of maps 
\begin{equation} \label{e:BG B4}
B_{\on{et}}(G)\times X\to B^4_{\on{et}}(A(1)),
\end{equation} 
equipped with an identification of the composite map
\begin{equation} \label{e:BG B4 triv}
X=\on{pt}\times X\to B_{\on{et}}(G)\times X\to B^4_{\on{et}}(A(1))
\end{equation} 
with 
$$X\to \on{pt}\to B^4_{\on{et}}(A(1)).$$

\medskip

We claim that there is a naturally defined map
\begin{equation} \label{e:from BG infty}
\Maps_{\on{Ptd}(\on{PreStk})}(B_{\on{et}}(G) \times X,B^4_{\on{et}}(A(1)))
\to \on{FactGe}_{A}(\Gr_G).
\end{equation} 

\sssec{} \label{sss:construct gerbe prel}



The construction of the map \eqref{e:from BG infty} proceeds as follows. Let us be given a map \eqref{e:BG B4}
equipped with a trivialization of the composition \eqref{e:BG B4 triv}. 

\medskip

For an affine test scheme $S$ and an $S$-point $(I,\CP_G,\alpha)$ of $\Gr_G$, we need to construct 
a $A$-gerbe $\CG_I$ on $S$. 

\medskip

Moreover, for $\phi:I\twoheadrightarrow J$, such that the point
$$\{\phi^{-1}(j) \subset \Maps(S,\Ran^J), \quad j\in J\}$$ hits $\Ran^J_{\on{disj}}$, we need to be
given an identification
\begin{equation} \label{e:gerbe on S factor}
\CG_{I}\simeq \underset{j\in J}\bigotimes\, \CG_{I_j}.
\end{equation} 

\sssec{} \label{sss:construct gerbe}

Let us interpret the datum of $\CP_G$ as a map
$$S\times X\to B_{\on{et}}(G)\times X.$$

Composing with \eqref{e:BG B4}, we obtain a map 
\begin{equation} \label{e:S to B4}
S\times X\to B^4_{\on{et}}(A(1)),
\end{equation} 
and a trivialization of the resulting map
\begin{equation} \label{e:U to B4}
U_I\to B^4_{\on{et}}(A(1)),
\end{equation} 
where $U_I$ is as in \secref{sss:U}.

\medskip

We claim that such a datum indeed gives rise to a $A$-gerbe $\CG_{I}$ on $S$, equipped 
with identifications \eqref{e:gerbe on S factor}. 

\sssec{} \label{sss:construct gerbe inter}

Recall that $\Gamma_I\subset S\times X$ denotes the union of the graphs of the maps that comprise $I\subset \Hom(S,X)$. 
Consider the maps
$$
\CD
\Gamma_I  @>{\iota}>> S\times X  \\
@V{\pi}VV  \\
S.
\endCD
$$

We regard the datum of \eqref{e:S to B4} together with a trivialization \eqref{e:U to B4} as a 4-cocycle
in
$$\on{C}^\bullet_{\on{et}}(\Gamma_I,\iota^!(A_{S\times X}(1))),$$
where for a scheme $Y$ we denote by $A_Y$ the constant \'etale sheaf with value $A$. 

\medskip 

Thus, in order to construct the gerbe $\CG_I$ on $S$ from \secref{sss:construct gerbe}, it suffices 
construct a map
\begin{equation} \label{e:trace map cohomology}
\on{C}^\bullet_{\on{et}}(\Gamma_I,\iota^!(A_{S\times X}(1)[2]))\to \on{C}^\bullet_{\on{et}}(S,A).
\end{equation}

\sssec{} \label{sss:construct gerbe inter next}

Let $p_X$ denote the projection $X\to \on{pt}$. We have a canonical identification
$$(p_X)^!(A)\simeq A_X(1)[2],$$
and hence
\begin{equation} \label{e:upper !}
(\on{id}_S\times p_X)^!(A_S)\simeq A_{S\times X}(1)[2].
\end{equation} 

\medskip

From here we obtain an isomorphism 
$$\iota^!(A_{S\times X})(1)[2]\simeq \pi^!(A_S),$$
and by the $(\pi_*,\pi^!)$-adjunction, a morphism 
\begin{equation}  \label{e:trace}
\pi_*\circ \iota^!(A_{S\times X})(1)[2]\to A_{S},
\end{equation}

\medskip

The sought-for morphism \eqref{e:trace map cohomology} is obtained from \eqref{e:trace} by applying 
$\on{C}^\bullet_{\on{et}}(S,-)$.  

\sssec{} \label{sss:construct gerbe end}

The factorization structure on the assignment
$$I\rightsquigarrow \CG_I$$
follows from the construction:

\medskip

For $\phi:I\twoheadrightarrow J$ as in \secref{sss:construct gerbe prel}, we have
$$\Gamma_I=\underset{j}\sqcup\, \Gamma_{I_j},$$
and under this identification, the map \eqref{e:trace map cohomology} is the sum of the corresponding maps
\eqref{e:trace map cohomology} with $\Gamma_I$ replaced by $\Gamma_{I_j}$. 

\sssec{}

We have the following assertion:

\begin{prop}  \label{p:parameterization}
The map \eqref{e:from BG infty} is an isomorphism.
\end{prop}

\begin{rem}
As was explained to us by J.~Lurie, the assertion of \propref{p:parameterization} is nearly tautological if
one works over the field of complex numbers and in the context of sheaves in the analytic topology.
\end{rem} 

\sssec{}

From \propref{p:parameterization} we will obtain that  
$$\pi_i(\on{FactGe}_{A}(\Gr_G))=H^{4-i}_{\on{et}}(B_{\on{et}}(G) \times X;\on{pt}\times X,A(1)).$$

In \secref{ss:analysis} below we will analyze what these cohomology groups look like. 

\sssec{Status of \propref{p:parameterization}}

The assertion of \propref{p:parameterization} was initially claimed in \cite[Theorem II.7.3]{Re}, but it was stated incorrectly, 
and the proof was erroneous.

\medskip

Currently, a proof is available thanks to the paper \cite{Zhao}. Namely, in Theorem 5.4 of {\it loc.cit.} an explicit description 
of $\on{FactGe}_{A}(\Gr_G)$ is given in terms of what the author calls ``enhanced $\Theta$-data". In particular, this theorem 
gives an explicit description of the homotopy groups $\pi_i(\on{FactGe}_{A}(\Gr_G))$, and those turn out to be isomorphic to
the homotopy groups $\pi_i(\Maps_{\on{Ptd}(\on{PreStk})}(B_{\on{et}}(G) \times X,B^4_{\on{et}}(A(1)))$ (the \secref{sss:calc homotopy}
for the explicit description of the latter). 

\medskip

Further, the compatibility of the map \eqref{e:from BG infty} for $G$ with that for its Cartan subgroup $T$, and an explicit
description of the map \eqref{e:from BG infty} for $T$ (see Sects. \ref{ss:tori expl} and \ref{ss:matching}), show that the map \eqref{e:from BG infty} induces
an isomorphism on homotopy groups. This implies that \eqref{e:from BG infty} is an equivalence, see also Remark
\ref{r:theta G bis}. 

\medskip

We will supply a yet different argument in \secref{s:proof of class}. 

\ssec{Digression: \'etale cohomology of $B(G)$}

\sssec{}  \label{sss:alg fund}

Let $\pi_{1,\on{alg}}(G)$ denote the algebraic fundamental group of $G$. Explicitly, $\pi_{1,\on{alg}}(G)$ can be described as follows:

\medskip

Choose a short exact sequence
\begin{equation} \label{e:cover group}
1\to T_2\to \wt{G}_1\to G\to 1,
\end{equation}
where $T_2$ is a torus and $[\wt{G}_1,\wt{G}_1]$ is simply connected. Set $T_1=\wt{G}_1/[\wt{G}_1,\wt{G}_1]$. Let $\Lambda_1$ and $\Lambda_2$
be the coweight lattices of $T_1$ and $T_2$, respectively. Then $\pi_{1,\on{alg}}(G)\simeq \Lambda_1/\Lambda_2$. 

\medskip

Equivalently, $\pi_{1,\on{alg}}(G)$ is the quotient of $\Lambda$ by the coroot lattice.

\sssec{}  \label{sss:quad restr}

For an abelian group $A$, let $\on{Quad}(\Lambda,A)^W$ denote the set of $W$-invariant quadratic forms on $\Lambda$ with values
in $A$. For any such form, denoted $q$, let $b$ denote the associated symmetric bilinear form:
\begin{equation} \label{e:bilin from quad}
b(\lambda_1,\lambda_2)=q(\lambda_1+\lambda_2)-q(\lambda_1)-q(\lambda_2).
\end{equation} 

\medskip

Let $\on{Quad}(\Lambda,A)^W_{\on{restr}}\subset \on{Quad}(\Lambda,A)^W$ be the subset consisting of forms $q$ that satisfy
the following additional condition: for every coroot $\alpha\in \Lambda$ and any $\lambda\in \Lambda$
\begin{equation} \label{e:restr cond}
b(\alpha,\lambda)=\langle \check\alpha,\lambda\rangle\cdot q(\alpha),
\end{equation} 
where $\check\alpha$ is the root corresponding to $\alpha$. 

\begin{rem}
Note that the identity
$$2b(\alpha,\lambda)=2\langle \check\alpha,\lambda\rangle\cdot q(\alpha)$$
holds automatically. 

\medskip

Moreover, \eqref{e:restr cond} itself holds automatically if 
$\frac{\alpha}{2}\in \Lambda$. 

\end{rem}

\sssec{}  \label{sss:quad restr bis}

Note that we have injective maps
$$\on{Quad}(\Lambda,\BZ)^W \underset{\BZ}\otimes A \hookrightarrow \on{Quad}(\Lambda,A)^W \hookleftarrow \on{Quad}(\pi_{1,\on{alg}}(G),A).$$
whose images belongs to $\on{Quad}(\Lambda,A)^W_{\on{restr}}$. 

\medskip

Assume for a moment that $A$ is divisible. Then we have a surjection 
$$\left(\on{Quad}(\Lambda,\BZ)^W \underset{\BZ}\otimes A \right) \oplus \on{Quad}(\pi_{1,\on{alg}}(G),A)\twoheadrightarrow
\on{Quad}(\Lambda,A)^W_{\on{restr}}.$$

\medskip

In particular, the inclusion
$$\on{Quad}(\Lambda,\BZ)^W \underset{\BZ}\otimes A\hookrightarrow \on{Quad}(\Lambda,A)^W_{\on{restr}}$$
is an equality when the derived group $[G,G]$ of $G$ is simply connected. 

\sssec{}

We claim:

\begin{thm} \label{t:cohomology BG}
Let $A$ be a torsion abelian group $A$ whose elements have orders prime to  $\on{char}(k)$. Assume also that $A$
is divisible. Then: 
$$H^{i}_{\on{et}}(B_{\on{et}}(G) ,A(1))=0 \text{ for } i=1,3;$$
$$H^{2}_{\on{et}}(B_{\on{et}}(G) ,A(1))\simeq \Hom(\pi_{1,\on{alg}}(G),A);$$
$$H^{4}_{\on{et}}(B_{\on{et}}(G) ,A(1))\simeq \on{Quad}(\Lambda,A(-1))^W_{\on{restr}}.$$
\end{thm}

\begin{rem} \label{r:H3}
When $A$ is not divisible, the only difference will be that $H^3_{\on{et}}(B_{\on{et}}(G) ,A(1))\simeq \on{Ext}^1(\pi_{1,\on{alg}}(G),A)$;
in particular it will vanish if the derived group of $G$ is simply-connected.
\end{rem} 

\medskip

As we could not find a reference for this statement in the literature, we will supply the proof in \secref{s:BG}. 

\begin{rem}
In fact, this is the same computation as in the context of algebraic topology, where we calculate singular
cohomology with coefficients in $\BQ/\BZ$ of the classifying space of a compact connected Lie group,
for which we could not find a reference either (the cohomology with $\BZ$ coefficients is well-known of course). 
\end{rem} 

\ssec{Analysis of homotopy groups of the space of factorization gerbes}   \label{ss:analysis}

In this subsection we will assume that $A$ is divisible (this assumption is only necessary 
when the derived group of $G$ is not simply-connected, see Remark \ref{r:H3} above). 

\sssec{}

Consider the object
\begin{equation} \label{e:K}
\CQ:=\tau^{\geq 1,\leq 4}\left(\on{C}_{\on{et}}(B_{\on{et}}(G) ,A(1))\right).
\end{equation}

In \secref{sss:expl K} we will give a ``hands-on" description of $\CQ$. 

\medskip

By the Leray spectral sequence associated with the projection $B_{\on{et}}(G) \times X\to X\to \on{pt}$
and smooth base change along
$$
\CD
B_{\on{et}}(G) \times X @>>> B_{\on{et}}(G)  \\
@VVV @VVV \\
X @>{p_X}>> \on{pt},
\endCD
$$
we have
$$H^{4-i}_{\on{et}}(B_{\on{et}}(G) \times X;\on{pt}\times X,A(1))\simeq 
H^{4-i}_{\on{et}}(X,p_X^*(\CQ)).$$

\sssec{}  \label{sss:calc homotopy} 

By \thmref{t:cohomology BG}, we have a distinguished triangle
\begin{equation} \label{e:loc triv gerbes}
\Hom(\pi_{1,\on{alg}}(G),A)[-2] \to \CQ \to \on{Quad}(\Lambda,A(-1))^W_{\on{restr}}[-4].
\end{equation} 

In particular, we obtain that $H^4_{\on{et}}(B_{\on{et}}(G) \times X;\on{pt}\times X,A(1))$ 
fits into \emph{canonically split} short exact sequence 
$$0\to H^2_{\on{et}}(X,\Hom(\pi_{1,\on{alg}}(G),A))\to H^4_{\on{et}}(B_{\on{et}}(G) \times X;\on{pt}\times X,A(1))
\to \on{Quad}(\Lambda,A(-1))^W_{\on{restr}}\to 0,$$
while
$$
H^3_{\on{et}}(B_{\on{et}}(G) \times X;\on{pt}\times X,A(1))
\simeq H^1_{\on{et}}(X,\Hom(\pi_{1,\on{alg}}(G),A))$$
and 
$$
H^2_{\on{et}}(B_{\on{et}}(G) \times X;\on{pt}\times X,A(1))
\simeq H^0_{\on{et}}(X,\Hom(\pi_{1,\on{alg}}(G),A)),$$
and
$$H^1_{\on{et}}(B_{\on{et}}(G) \times X;\on{pt}\times X,A(1))=H^0_{\on{et}}(B_{\on{et}}(G) \times X;\on{pt}\times X,A(1))=0.$$

\sssec{} \label{sss:calc homotopy Ge} 

Combining with \propref{p:parameterization}, we obtain that $\pi_0(\on{FactGe}_{A}(\Gr_G))$ fits into \emph{canonically split} short exact sequence 
$$0\to H^2_{\on{et}}(X,\Hom(\pi_{1,\on{alg}}(G),A)) \to \pi_0(\on{FactGe}_{A}(\Gr_G))
\to \on{Quad}(\Lambda,A(-1))^W_{\on{restr}}\to 0,$$
while 
$$\pi_1(\on{FactGe}_{A}(\Gr_G))\simeq H^1_{\on{et}}(X,\Hom(\pi_{1,\on{alg}}(G),A))$$
and 
$$\pi_2(\on{FactGe}_{A}(\Gr_G))\simeq \Hom(\pi_{1,\on{alg}}(G),A).$$

\sssec{}

In particular, we obtain a map (of spectra)
$$\on{FactGe}_{A}(\Gr_G)\to \on{Quad}(\Lambda,A(-1))^W_{\on{restr}}.$$

Let $\on{FactGe}^0_{A}(\Gr_G)$ denote its fiber. 

\sssec{}

From \propref{p:parameterization} and the distinguished triangle \eqref{e:loc triv gerbes}, we obtain:  

\begin{cor} \label{c:parameterization neutral} 
There is a canonical isomorphism
\begin{equation} \label{e:from BG infty neutral}
\Maps(X,B^2_{\on{et}}(\Hom(\pi_{1,\on{alg}}(G),A)))\simeq \on{FactGe}^0_{A}(\Gr_G).
\end{equation}
The subspace 
$$\on{FactGe}^0_{A}(\Gr_G)\subset \on{FactGe}_{A}(\Gr_G)$$ consists of objects that are trivial 
\'etale-locally on $X$. 
\end{cor}

\ssec{Parametrization of factorization line bundles}  \label{ss:lines}

This subsection is included for the sake of completeness, in order to make contact with the theory of metaplectic extensions
developed in \cite{We}. 

\medskip

Recall from \secref{sss:line bundles Grass} that given a factorization line bundle $\CL$ on $\Gr_G$ and an element $a\in A(-1)$
we can produce a factorization gerbe $\CL^{a}$.  In this subsection we will describe a geometric data that gives rise
to factorization line bundles\footnote{We emphasize that this construction produces just factorization line bundles, 
and \emph{not} $\BZ/2\BZ$-graded ones.} on $\Gr_G$. 

\sssec{}

Let $K_2$ denote the prestack over $X$ that associates to an affine scheme $S=\Spec(A)$ mapping to $X$ the abelian group $K_2(A)$. Let
$(K_2)_{\on{Zar}}$ be the sheafification of $K_2$ in the Zariski topology.

\medskip

On the one hand, we consider the space $\on{CExt}(G,(K_2)_{\on{Zar}})$ (in fact, an ordinary groupoid) of \emph{Brylinski-Deligne data}, 
which are by definition \emph{central} extensions 
$$1\to (K_2)_{\on{Zar}}\to \wt{G}\to G\times X\to 1$$
of the constant group-scheme $G\times X$ by $(K_2)_{\on{Zar}}$. 

\medskip

The operation of Baer sum makes $\on{CExt}(G,(K_2)_{\on{Zar}})$ into a commutative group in spaces, i.e., into a Picard category. 

\medskip

On the other hand, consider the Picard category
$$\on{FactPic}(\Gr_G)$$
of factorizable line bundles on $\Gr_G$. 

\medskip

In the paper \cite{Ga6} a map of Picard groupoids is constructed:
\begin{equation} \label{e:from K_2}
\on{CExt}(G,(K_2)_{\on{Zar}})\to \on{FactPic}(\Gr_G),
\end{equation} 
and the following conjecture is stated (this is Conjecture 6.1.2 in {\it loc.cit.})\footnote{Since the previous version of this paper, this 
conjecture has been proved in \cite{TZ}.}: 

\medskip

\begin{conj} \label{c:line bundle class}
The map \eqref{e:from K_2} is an isomorphism.
\end{conj} 

\begin{rem}
One can show that 
it follows from \cite[Theorem 3.16]{BrDe} combined with \secref{sss:theta} that \conjref{c:line bundle class} holds when $G=T$ is a torus. 
\end{rem}

\sssec{}

Let us fix an integer $\ell$ of order prime to $\on{char}(k)$. In \cite[Sect. 6.3.6]{Ga6} the following map was constructed
\begin{equation} \label{e:from K to etale}
\on{CExt}(G,(K_2)_{\on{Zar}})\to  \Maps_{\on{Ptd}(\on{PreStk}_{/X})}\left(B_{\on{et}}(G)\times X,B^4_{\on{et}}(\mu_\ell^{\otimes 2})\times X\right).
\end{equation} 

Let us take $A=\mu_\ell$, and note that $A(1)\simeq \mu_\ell^{\otimes 2}$. 
Note that the construction in \secref{sss:line bundles to gerbes} gives rise to a canonical map
\begin{equation} \label{e:line bundles to gerbes}
\on{FactPic}(\Gr_G)\to \on{FactGe}_{\mu_\ell}(\Gr_G).
\end{equation} 

\medskip

The following is equivalent to Conjecture 6.3.8 of {\it loc.cit.}:

\begin{conj}  \label{c:compat with K2}
The following diagram commutes: 
$$
\CD
\on{CExt}(G,(K_2)_{\on{Zar}})  @>{\text{\eqref{e:from K to etale}}}>>  
 \Maps_{\on{Ptd}(\on{PreStk}_{/X})}\left(B_{\on{et}}(G)\times X,B^4_{\on{et}}(\mu_\ell^{\otimes 2})\times X\right) \\
@V{\text{\eqref{e:from K_2}}}VV     @VV{\text{\eqref{e:from BG infty}}}V   \\
\on{FactPic}(\Gr_G)  @>{\text{\eqref{e:line bundles to gerbes}}}>>    \on{FactGe}_{\mu_\ell}(\Gr_G).
\endCD
$$
\end{conj}

\section{The case of tori}   \label{s:torus}

In this section we let $G=T$ be a torus. We will perform an explicit analysis of factorization gerbes on the affine Grassmannian
$\Gr_T$, and introduce related objects (multiplicative factorization gerbes) that will play an important
role in the sequel. 

\ssec{Factorization Grassmannian for a torus}  \label{ss:factor gerbes for tori}

In this section we will show that the affine Grassmannian of a torus can be approximated by a prestack
assembled from (=written as a colimit of) powers of $X$.

\sssec{}  \label{sss:Gr for tori}

Recall that $\Lambda$ denotes the coweight lattice of $G=T$. 
Consider the index category whose objects are pairs $(I,\lambda^I)$, where $I$ is a finite non-empty set and $\lambda^I$ is a map $I\to \Lambda$;
in what follows we will denote by $\lambda_i\in \Lambda$ is the value of $\lambda^I$ on $i\in I$.

\medskip

A morphism $(J,\lambda^J)\to (I,\lambda^I)$ is a surjection $\phi:I\twoheadrightarrow J$ such that
\begin{equation} \label{e:int}
\lambda_j=\underset{i\in \phi^{-1}(j)}\Sigma\, \lambda_i.
\end{equation}

\medskip 

Consider the prestack
$$\Gr_{T,\on{comb}}:=\underset{(I,\lambda^I)}{\on{colim}}\, X^I.$$

The prestack $\Gr_{T,\on{comb}}$ endowed with its natural forgetful map to $\Ran$, also has a natural factorization structure. 

\medskip

There is a canonical map 
\begin{equation} \label{e:combinatorial map}
\Gr_{T,\on{comb}}\to \Gr_T,
\end{equation} 
compatible with the factorization structures.

\medskip

Namely, for each $(I,\lambda^I)$ the corresponding $T$-bundle on $X^I\times X$ is
$$\underset{i\in I}\bigotimes\, \lambda_i\cdot \CO(\Delta_i),$$
where $\Delta_i$ is the divisor on $X^I\times X$ corresponding to the
$i$-th coordinate being equal to the last one. 

\sssec{}

As in \cite[Sect. 8.1]{Ga2} one shows that the map \eqref{e:combinatorial map} induces an isomorphism of the sheafifications
with respect to the topology on the category of affine schemes of finite type, in which coverings are finite surjective maps. 

\medskip

In particular, for any $S\to \Ran$, the map
$$\on{Ge}_A(S\underset{\Ran}\times \Gr_T)\to \on{Ge}_A(S\underset{\Ran}\times \Gr_{T,\on{comb}})$$
is an isomorphism, and hence, so is the map 
$$\on{FactGe}_A(\Gr_T)\to \on{FactGe}_A(\Gr_{T,\on{comb}}).$$

Furthermore, for a given $\CG\in \on{FactGe}_{E^{\times,\on{tors}}}(\Gr_T)$, the corresponding map
of sheaves of categories 
$$\Shv_\CG(\Gr_T)_{/\Ran}\to \Shv_\CG(\Gr_{T,\on{comb}})_{/\Ran}$$
is also an isomorphism. 

\sssec{}  \label{sss:factor gerbes tori}

The datum of a factorization gerbe on $\Gr_{T,\on{comb}}$ can be explicitly described as follows: 

\medskip

For a finite set
$I$ and a map  
$$\lambda^I:I\to \Lambda$$ 
we specify a gerbe $\CG^{\lambda^I}$ on $X^I$.  

\medskip

For a surjection of finite sets $I\overset{\phi}\twoheadrightarrow J$ such that \eqref{e:int} holds, 
we specify an identification
\begin{equation} \label{e:diag compat of gerbes}
(\Delta_\phi)^*(\CG^{\lambda^I})\simeq \CG^{\lambda^J}.
\end{equation} 


The identifications \eqref{e:diag compat of gerbes} must be compatible with compositions of maps of finite sets in the natural sense. 

\medskip

Let now $I\overset{\phi}\twoheadrightarrow J$ be a surjection of finite sets, and let 
$$X^I_{\phi,\on{disj}}\subset X^I, \quad x_{i_1}\neq x_{i_2} \text{ whenever } \phi(i_1)\neq \phi(i_2)$$ be the corresponding
open subset. For $j\in J$, let
$\lambda^{I_j}$ be the restriction of $\lambda^I$ to $I_j$. 

\medskip

We impose the structure of factorization that consists of isomorphisms
\begin{equation} \label{e:div factor of gerbes}
(\CG^{\lambda^I})|_{X^I_{\phi,\on{disj}}}\simeq \left(\underset{j\in J}\bigotimes\, \CG^{\lambda^{I_j}}\right)|_{X^I_{\phi,\on{disj}}}.
\end{equation}

The isomorphisms \eqref{e:div factor of gerbes} must be compatible with compositions of maps of finite sets in the natural sense. 

\medskip

In addition, the isomorphisms \eqref{e:div factor of gerbes} and \eqref{e:diag compat of gerbes} must be compatible in the natural
sense. 

\sssec{}

For a factorization gerbe $\CG$ on $\Gr_{T,\on{comb}}$, the value of the category $ \Shv_\CG(\Gr_{T,\on{comb}})_{/\Ran}$
on $X^I\to \Ran$ can be explicitly described as follows:

\medskip

It is the limit over the index category 
$$(J,\lambda^J,I\twoheadrightarrow J)$$
of the categories $\Shv_{\CG^{\lambda^J}}(X^J)$. 

\sssec{The case of factorization line bundles}  \label{sss:theta}

The datum of a factorization $\BZ/2\BZ$-graded line bundle on $\Gr_{T,\on{comb}}$ can be described in a way similar to that of factorization gerbes.
This description recovers the notion of what in \cite[Sect. 3.10.3]{BD1}
is called a $\theta$-datum. 

\medskip

We note that a factorization $\BZ/2\BZ$-graded line bundle is evenly (i.e., trivially) graded if and only if the corresponding
$\theta$-datum is even, i.e., if the corresponding symmetric bilinear $\BZ$-valued form on $\Lambda$
comes from a $\BZ$-valued quadratic form.

\medskip

We also note that \cite[Proposition 3.10.7]{BD1} says that restriction along 
$$\Gr_{T,\on{comb}}\to \Gr_T$$ defines an equivalence between the Picard categories of factorization ($\BZ/2\BZ$-graded) line bundles.

\ssec{Making the parameterization explicit for tori}  \label{ss:tori expl}

In this subsection we will show explicitly how a factorization $A$-gerbe on $\Gr_T$ gives rise to an $A$-valued quadratic form
$$q:\Lambda\to A(-1).$$

\medskip

For the duration of this section we assume that elements of $A$ have orders prime to $\on{char}(k)$. 


\sssec{}  \label{sss:bilin expl}

We first describe the bilinear form 
$$b:\Lambda\times \Lambda \to A(-1).$$

\medskip

For an element $\lambda\in \Lambda$, take $I$ to be the one-element set $\{*\}$, and consider the
corresponding map 
$$\lambda^I:I\to \Lambda, \quad *\mapsto \lambda.$$

Let $\CG^\lambda$ denote the resulting $A$-gerbe on $X$.

\medskip

Given two elements $\lambda_1,\lambda_2\in \Lambda$, consider $I=\{1,2\}$ and the map
$$\lambda^I:I\to \Lambda;\quad 1\mapsto \lambda_1,2\mapsto \lambda_2.$$

Consider the corresponding gerbe 
$$\CG^{\lambda_1,\lambda_2}:=\CG^{\lambda^I}$$ over $X^2$. 

\medskip 

By \eqref{e:div factor of gerbes}, $\CG^{\lambda_1,\lambda_2}$ is
identified with $\CG^{\lambda_1}\boxtimes \CG^{\lambda_2}$ over $X^2-\Delta$. By \lemref{l:trivialized gerbes},
there exists a well-defined element $a\in A(-1)$ such that
\begin{equation} \label{e:form as pole}
\CG^{\lambda_1,\lambda_2}\simeq (\CG^{\lambda_1}\boxtimes \CG^{\lambda_2})\otimes \CO(\Delta)^a.
\end{equation} 

We let
$$a=:b(\lambda_1,\lambda_2).$$

\sssec{}   \label{sss:bilin}

The fact that $b(-,-)$ is bilinear can be seen as follows. For a triple of elements $\lambda_1,\lambda_2,\lambda_3$
consider the corresponding gerbes
$$\CG^{\lambda_1,\lambda_2,\lambda_3} \text{ and } (\CG^{\lambda_1,\lambda_2}\boxtimes \CG^{\lambda_3})\otimes 
\CO(\Delta_{1,3})^{\otimes b(\lambda_1,\lambda_3)}\otimes \CO(\Delta_{2,3})^{\otimes b(\lambda_2,\lambda_3)}$$
over $X^3$.

\medskip

They are identified away from the main diagonal $\Delta_{1,2,3}$, and hence this identification extends to all of $X^3$,
since $\Delta_{1,2,3}$ has codimension 2. Restricting to $\Delta_{1,2}$, 
we obtain an identification 
$$\CG^{\lambda_1+\lambda_2,\lambda_3}\simeq (\CG^{\lambda_1+\lambda_2}\boxtimes \CG^{\lambda_3})\otimes 
\CO(\Delta)^{\otimes b(\lambda_1,\lambda_3)}\otimes \CO(\Delta)^{\otimes b(\lambda_2,\lambda_3)}$$
as gerbes over $X^2$. Comparing with the identification
$$ \CG^{\lambda_1+\lambda_2,\lambda_3}\simeq (\CG^{\lambda_1+\lambda_2}\boxtimes \CG^{\lambda_3})\otimes 
\CO(\Delta)^{\otimes b(\lambda_1+\lambda_2,\lambda_3)},$$
we obtain the desired
$$b(\lambda_1,\lambda_3)+b(\lambda_2,\lambda_3)=b(\lambda_1+\lambda_2,\lambda_3).$$

\sssec{}  

It is easy to see that the resulting map 
$$b:\Lambda\times \Lambda\to A(-1)$$ 
is symmetric. In fact, we have a canonical datum of commutativity for the diagram
\begin{equation} \label{e:transposition unequal}
\CD
\sigma^*(\CG^{\lambda_1,\lambda_2}) @>>>  \sigma^*((\CG^{\lambda_1}\boxtimes \CG^{\lambda_2})\otimes \CO(\Delta)^{b(\lambda_1,\lambda_2)})  \\
@VVV  @VVV  \\
\CG^{\lambda_2,\lambda_1} @>>>  (\CG^{\lambda_2}\boxtimes \CG^{\lambda_1})\otimes \CO(\Delta)^{b(\lambda_2,\lambda_1)}
\endCD
\end{equation}
that extends the given one over $X\times X-\Delta$  (in the above formula, $\sigma$ denotes the transposition acting on $X\times X$): 

\medskip

Indeed, the measure of \emph{non-commutativity} of the above diagram is an \'etale $A$-torsor over $X\times X$, which is trivialized over $X\times X-\Delta$,
and hence this trivialization uniquely extends to all of $X\times X$. 

\medskip

For the sequel we will need to understand in more detail the behavior of the restriction of diagram \eqref{e:transposition unequal} 
to the diagonal. 

\sssec{}  \label{sss:Kummer}

We start with the following observation. We claim that to an element $a\in A(-1)$ one can canonically attach an $A$-torsor $(-1)^a$:

\medskip

The Kummer cover
$$\BG_m\overset{x\mapsto x^n}\longrightarrow \BG_m$$
defines a group homomorphism
\begin{equation} \label{e:Kummer}
\BG_m\to B_{\on{et}}(\mu_n).
\end{equation}

From here we obtain a bilinear map
\begin{equation} \label{e:Kummer again}
A(-1)\times \BG_m\to B_{\on{et}}(A),
\end{equation}
i.e., an element $a\in A(-1)$ defines an \'etale $A$-torsor $\chi_a$ over $\BG_m$, which behaves multiplicatively. 

\sssec{} \label{sss:-1}

We let $(-1)^a$
denote the fiber of $\chi_a$ at $-1\in \BG_m$.

\medskip

The multiplicativity of \eqref{e:Kummer again} along $\BG_m$ implies that we have a canonical trivialization
\begin{equation} \label{e:-1 triv}
((-1)^a)^{\otimes 2}\simeq \on{triv}. 
\end{equation}

\medskip

The multiplicativity of \eqref{e:Kummer again} along $A(-1)$ implies that a choice of $a'\in A(-1)$ such that $2a'=a$ defines
a trivialization of $(-1)^a$. Moreover, this trivialization is compatible with \eqref{e:-1 triv}. 

\medskip

This construction is a morphism (and 
hence an \emph{isomorphism}) of $A_{2\on{-tors}}$-torsors:  
$$\{a'\in A(-1)\,,\, 2a'=a\}\to \{\text{trivializations of }(-1)^a\text{ compatible with \eqref{e:-1 triv}}\}.$$
(By enlarging $A$ if needed, one shows that the LHS is empty if and only if the RHS is.) 

\sssec{} \label{sss:-1 diag}

Consider now the $A$-gerbe $\CO(\Delta)^a$ on $X\times X$, equipped with the natural identification
\begin{equation} \label{e:sigma Delta}
\sigma^*(\CO(\Delta)^a)\simeq \CO(\Delta)^a,
\end{equation}
which uniquely extends the tautological one over $X\times X-\Delta$. 

\medskip

Restricting \eqref{e:sigma Delta} to the diagonal, and using the fact that $\sigma|_{\Delta}$ is trivial, 
we obtain an identification of $A$-gerbes
\begin{equation} \label{e:former phi}
\CO(\Delta)^a|_{\Delta}\simeq \CO(\Delta)^a|_{\Delta},
\end{equation}
whose square is the identity map.

\medskip

The map \eqref{e:former phi} is given by tensoring by an $A$-torsor that squares to the trivial one. It is easy to see that
this torsor is constant along $X$ and identifies canonically with $(-1)^a$ in a way compatible with 
\eqref{e:-1 triv}. This follows from the fact that the composite
$$\CO(\Delta)|_{\Delta} \simeq \sigma^*(\CO(\Delta))|_{\Delta}\simeq  \sigma^*(\CO(\Delta)|_{\Delta})\simeq \CO(\Delta)|_{\Delta}$$
acts as $-1$. 

\sssec{}

The identification \eqref{e:diag compat of gerbes} for the map $\{1,2\}\to \{*\}$ yields an identification
\begin{equation} \label{e:diag compat of gerbes 2}
\CG^{\lambda_1,\lambda_2}|_\Delta\simeq \CG^{\lambda_1+\lambda_2},
\end{equation} 
compatible with the transposition of factors, i.e., the diagram  
\begin{equation} \label{e:transposition on diag}
\CD
\CG^{\lambda_1+\lambda_2}   @>>>  \CG^{\lambda_1,\lambda_2}|_{\Delta} \\
@V{\on{id}}VV  @VV{\sim}V  \\
\CG^{\lambda_2+\lambda_1}   @>>>  \CG^{\lambda_2,\lambda_1}|_{\Delta} 
\endCD
\end{equation}
is endowed with a datum of commutativity that squares to one. In the above diagram the right vertical arrow is the map
$$\CG^{\lambda_1,\lambda_2}|_{\Delta} \simeq \sigma^*(\CG^{\lambda_1,\lambda_2})|_{\Delta} \simeq 
\CG^{\lambda_2,\lambda_1}|_{\Delta}.$$ 

\medskip

Restricting  diagram \eqref{e:transposition unequal} to the diagonal and concatenating with 
diagrams \eqref{e:transposition on diag}, we obtain that we have a datum of commutativity for the diagram
\begin{equation} \label{e:transposition unequal diag}
\CD
\CG^{\lambda_1+\lambda_2}   @>>>  (\CG^{\lambda_1}\otimes \CG^{\lambda_2})\otimes \CO(\Delta)^{b(\lambda_1,\lambda_2)}|_{\Delta} \\
@V{\on{id}}VV  @VV{\text{tautological}\otimes (-1)^{b(\lambda_1,\lambda_2)}}V  \\
\CG^{\lambda_2+\lambda_1}   @>>>  (\CG^{\lambda_2}\otimes \CG^{\lambda_1})\otimes \CO(\Delta)^{b(\lambda_2,\lambda_1)}|_{\Delta} 
\endCD
\end{equation}
that squares to the tautological one.

\sssec{}  \label{sss:quad}

We are finally ready to recover the quadratic form 
$$q:\Lambda\to A(-1).$$

Namely, in \eqref{e:transposition unequal diag}, let us set $\lambda_1=\lambda=\lambda_2$. We obtain a datum of commutativity
of the diagram
\begin{equation} \label{e:recover quad}
\CD
\CG^{2\lambda}   @>>>  (\CG^{\lambda}\otimes \CG^{\lambda})\otimes \CO(\Delta)^{b(\lambda,\lambda)}|_{\Delta} \\
@V{\on{id}}VV  @VV{\on{id}\otimes (-1)^{b(\lambda,\lambda)}}V  \\
\CG^{2\lambda}   @>>>  (\CG^{\lambda}\otimes \CG^{\lambda})\otimes \CO(\Delta)^{b(\lambda,\lambda)}|_{\Delta},
\endCD
\end{equation}
where the upper and lower horizontal arrows are canonically identified, and which squares to the tautological one.

\medskip

This datum is equivalent to that of trivialization of the $A$-torsor $(-1)^{b(\lambda,\lambda)}$,
that squares to the identity. By \secref{sss:-1}, this datum is equivalent to that of an element $q(\lambda)\in A(-1)$ 
such that $2q(\lambda)=b(\lambda,\lambda)$. This is the value of our quadratic form on $\lambda$.

\sssec{}

The relation 
$$q(\lambda_1+\lambda_2)=q(\lambda_1)+q(\lambda_2)+b(\lambda_1,\lambda_2)$$
is verified in a way similar to \secref{sss:bilin}. 

\ssec{Matching the parameters} \label{ss:matching}

\sssec{}

Let us start with a datum of a based map
\begin{equation} \label{e:based map tori}
B_{\on{et}}(T)\times X\to B^4_{\on{et}}(A(1)),
\end{equation} 
and produce an object $\CG$ of $\on{FactGe}_{A}(\Gr_T)$ by the map \eqref{e:from BG infty}.

\medskip

In \secref{ss:tori expl}, to $\CG$ we have attached a quadratic form
$$q:\Lambda\to A(-1).$$

In this section we will show that $q$ equals to the form attached to \eqref{e:based map tori} via the map
\begin{multline}   \label{e:form again}
\pi_0\left(\Maps_{\on{Ptd}(\on{PreStk})}(B_{\on{et}}(T)\times X,B^4_{\on{et}}(A(1)))\right)=
H^4_{\on{et}}(B_{\on{et}}(T) \times X;\on{pt}\times X,A(1))\to \\
\to H^4_{\on{et}}(B_{\on{et}}(T) ;\on{pt},A(1))\simeq \on{Quad}(\Lambda,A(-1)).
\end{multline} 

I.e., we want to establish the commutativity of the diagram
\begin{equation} \label{e:compare forms}
\CD
\pi_0\left(\Maps_{\on{Ptd}(\on{PreStk})}(B_{\on{et}}(T)\times X,B^4_{\on{et}}(A(1)))\right) 
@>{\text{\eqref{e:from BG infty}}}>> \pi_0(\on{FactGe}_{A}(\Gr_T)) \\
@V{\sim}VV  @VV{\text{\secref{ss:tori expl}}}V   \\
H^4_{\on{et}}(B_{\on{et}}(T) \times X;\on{pt}\times X,A(1)) @>>> \on{Quad}(\Lambda,A(-1)).
\endCD
\end{equation} 

\sssec{}

From the fiber sequence \eqref{e:loc triv gerbes}, we obtain a fiber sequence 

\begin{equation} \label{e:loc triv gerbes again}
\Maps(X,B^2_{\on{et}}(\Hom(\Lambda,A)))\to \Maps_{\on{Ptd}(\on{PreStk})}(B_{\on{et}}(T)\times X,B^4_{\on{et}}(A(1)))\to 
\on{Quad}(\Lambda,A(-1)).
\end{equation} 

\medskip

First, we claim that if $\CG$ comes from an object in $\Maps(X,B^2_{\on{et}}(\Hom(\Lambda,A)))$, then the form $q$,
attached to $\CG$ by \secref{ss:tori expl}, equals $0$.

\medskip

Indeed, in this case $\CG$ is trivial \'etale-locally on $X$. And the form $q$ is zero by construction. 

\medskip

Thus, it remains to exhibit a collection of objects $\CG$ of $\Maps_{\on{Ptd}(\on{PreStk})}(B_{\on{et}}(T)\times X,B^4_{\on{et}}(A(1)))$,
whose images under \eqref{e:form again} span $\on{Quad}(\Lambda,A(-1))$, on which the two circuits of the diagram \eqref{e:compare forms}
produce the same result.

\sssec{}

With a future application in mind, we will choose our collection of objects $\CG$ to be obtained as compositions
$$B_{\on{et}}(T)\times X \to B_{\on{et}}(T)\to B^4_{\on{et}}(A(1))$$
for some particular collection of based maps $B_{\on{et}}(T)\to B^4_{\on{et}}(A(1))$.

\medskip

Let $A_1$ and $A_2$ be a pair of finite abelian groups (of orders prime to $\on{char}(k)$),
equipped with homomorphisms
$$\chi_i:\Lambda\to A_i(-1)$$
and a bilinear map
$$b':A_1\times A_2\to A(1),$$
which we can also think of as a bilinear map
$$A_1(-1)\times A_2(-1)\to A(-1).$$

Let $q$ be the quadratic form on $\Lambda$ with values in $A(-1)$ given by
\begin{equation} \label{e:q from b'}
q(\lambda)=b'(\chi_1(\lambda),\chi_2(\lambda)).
\end{equation}

Clearly, forms $q$ obtained in this way span $\on{Quad}(\Lambda,A(-1))$. 

\medskip

Let $b$ denote the associated symmetric bilinear form. Explicitly,
\begin{equation} \label{e:b from b'}
b(\lambda_1,\lambda_2)=b'(\chi_1(\lambda_1),\chi_2(\lambda_2))+b'(\chi_1(\lambda_2),\chi_2(\lambda_1)).
\end{equation} 

\sssec{}

We can regard $\chi_i$ as a map of pointed prestacks
\begin{equation} \label{e:chi}
B_{\on{et}}(T)  \to  B^2_{\on{et}}(A_i).
\end{equation} 

The map $b'$ and cup-product give rise to a (based) map 
$$B^2(A_1)\times B^2(A_2)\to B^4(A(1)),$$
which in turn gives rise to a (based) map
$$B^2_{\on{et}}(A_1)\times B^2_{\on{et}}(A_2)\to B^4_{\on{et}}(A(1)).$$

Precomposing with the $\chi_i$'s we obtain a (based) map
\begin{equation} \label{e:cup product}
B_{\on{et}}(T)\overset{\chi_1,\chi_2}\longrightarrow B^2_{\on{et}}(A_1)\times B^2_{\on{et}}(A_2) \to B^4_{\on{et}}(A(1)).
\end{equation}

\medskip

The class of the map \eqref{e:cup product} in $H^4_{\on{et}}(B_{\on{et}}(T) ;\on{pt},A(1))$
is the cup product of the classes of $\chi_i\in H^2_{\on{et}}(B_{\on{et}}(T);\on{pt},A_i)$,
$i=1,2$. 

\medskip

The corresponding element in 
$$H^4_{\on{et}}(B_{\on{et}}(T);\on{pt},A(1))\simeq \on{Quad}(\Lambda,A(-1))$$
is the form $q$ from \eqref{e:q from b'}. 

\sssec{}

Let $\CG$ be the object of $\on{FactGe}_{A}(\Gr_T)$ corresponding via \eqref{e:from BG infty} to the composition
$$B_{\on{et}}(T)\times X \to B_{\on{et}}(T)\overset{\text{\eqref{e:cup product}}}\longrightarrow B^4_{\on{et}}(A(1)).$$

We are going to show that the quadratic form attached to $\CG$ by the procedure of \secref{ss:tori expl} equals $q$. 
To show this, we will have to unwind the construction in Sects. \ref{sss:construct gerbe prel}-\ref{sss:construct gerbe end}. 

\sssec{} \label{sss:identify 2}

First, we identify explicitly the corresponding gerbes $\CG^\lambda$. We claim that we have
\begin{equation} \label{e:id G lambda}
\CG^\lambda \simeq (\omega^{\otimes -1}_X)^{q(\lambda)},
\end{equation} 
where $\omega_X$ is the sheaf of $1$-forms on $X$. 

\medskip

We take $S=X$ with $I$ being the one-element set corresponding to the identity map $X\to X$. Consider the map
\begin{equation} \label{e:diag lambda}
X\times X \to B_{\on{et}}(T),
\end{equation} 
corresponding to the $T$-bundle $\lambda\cdot \CO(\Delta)$, equipped with its natural trivialization over
$$U_I=X\times X-\Delta.$$

We identify $\Gamma_I=X$ with the map $\iota$ being the diagonal map. The composition of \eqref{e:diag lambda} with
\eqref{e:chi} (say, for $i=1$) 
defines a 2-cocycle in 
\begin{equation} \label{e:2-cocyle}
\on{C}^\bullet_{\on{et}}(X,\Delta^!((A_1)_{X\times X})).
\end{equation} 

We identify 
$$\Delta^!((A_1)_{X\times X})\simeq A_1(-1)[-2].$$

Thus, the above 2-cocycle corresponds to an element of $A_1(-1)$. 
We claim that the resulting element of $A_1(-1)$ equals $\chi_1(\lambda)$. 

\medskip

Indeed, the object 
$$\chi_1\in \Maps_{\on{Ptd}(\on{PreStk})}(B_{\on{et}}(T) ,B^2_{\on{et}}(A_1))$$
is the $A_1$-gerbe over $B_{\on{et}}(T)$ attached by 
the procedure in \secref{sss:line bundles to gerbes} to the tautological $T$-bundle
and the map $\chi_1:\Lambda\to A_1(-1)$. Hence, the 2-cocyle in \eqref{e:2-cocyle}
is obtained by the procedure \secref{sss:line bundles to gerbes} applied to the
$T$-bundle $\lambda\cdot \CO(\Delta)$ equipped with its trivialization over $X\times X-\Delta$.
The required assertion follows now from \secref{sss:trivialized gerbes}. 

\sssec{} \label{sss:cup prod}

Now, the composition of \eqref{e:diag lambda} with \eqref{e:cup product}, viewed as a 4-cocycle in
$$\on{C}^\bullet_{\on{et}}(X,\Delta^!(A_{X\times X})(1))$$
is obtained from the above 2-cocycle in \eqref{e:2-cocyle} by multiplication with the 2-cocycle in
\begin{equation} \label{e:pullback to diag}
\on{C}^\bullet_{\on{et}}(X,\Delta^*((A_2)_{X\times X})))
\end{equation}
under the cup-product map
\begin{equation} \label{e:cup prod on curve}
\on{C}^\bullet_{\on{et}}(X,\Delta^!((A_1)_{X\times X}))) \otimes 
\on{C}^\bullet_{\on{et}}(X,\Delta^*((A_2)_{X\times X}))\overset{b'}\to
\on{C}^\bullet_{\on{et}}(X,\Delta^!(A_{X\times X}(1)))
\end{equation}
where the 2-cocycle in \eqref{e:pullback to diag} is 
$$X \overset{\Delta}\to X\times X \overset{\text{\eqref{e:diag lambda}}}\longrightarrow B_{\on{et}}(T) 
\overset{\text{\eqref{e:chi} for $i=2$}}\longrightarrow 
B^2_{\on{et}}(A_2).$$

\sssec{}

The map \eqref{e:cup prod on curve} fits into a commutative diagram
$$
\CD 
\on{C}^\bullet_{\on{et}}(X,\Delta^!((A_1)_{X\times X}))) \otimes  \on{C}^\bullet_{\on{et}}(X,\Delta^*((A_2)_{X\times X})) @>{b'}>> 
\on{C}^\bullet_{\on{et}}(X,\Delta^!(A_{X\times X}(1))) \\
@V{\sim}VV @VV{\sim}V \\
\on{C}^\bullet_{\on{et}}(X,(A_1)_{X}(-1)[-2])  \otimes 
\on{C}^\bullet_{\on{et}}(X,(A_2)_X) @>{b'}>> 
\on{C}^\bullet_{\on{et}}(X,A_X[-2]).
\endCD
$$

Hence, in order to prove \eqref{e:id G lambda}, it suffices to show that the 2-cocycle in \eqref{e:pullback to diag},
thought of as a 2-cocycle in 
\begin{equation} \label{e:Chern class}
\on{C}^\bullet_{\on{et}}(X,(A_2)_X),
\end{equation} 
interpreted as a $A_2$-gerbe on $X$, identifies with
$$(\omega^{\otimes -1}_X)^{\chi_2(\lambda)}.$$

\sssec{}

Note that the map \eqref{e:chi} (say, for $i=2$) is obtained by the procedure of 
\secref{sss:line bundles to gerbes} applied to the tautological $T$-bundle on $B_{\on{et}}(T)$ and 
$\chi_2:\Lambda\to A_2(-1)$. 

\medskip

Hence, the above class in \eqref{e:Chern class} is attached by the procedure of 
\secref{sss:line bundles to gerbes} applied to the $T$-bundle $\lambda\cdot \CO(\Delta)|_{\Delta}$ and 
$\chi_2:\Lambda\to A_2(-1)$.

\medskip

The required assertion follows now from the fact that
$$\CO(\Delta)|_{\Delta}\simeq \omega^{\otimes -1}_X.$$

\begin{rem} \label{r:swap lambda}

The above description of the resulting $A$-gerbe $\CG^\lambda$ relied on breaking the symmetry in the roles of $A_1$ and $A_2$.
We claim that if we swap the roles, the resulting automorphism of the gerbe $(\omega^{\otimes -1}_X)^{q(\lambda)}$
will be given by the $A_{2\on{-tors}}$-torsor $(-1)^{q(\lambda)}$. 

\medskip

Indeed, the 4-cocycle we obtained 
$$\on{C}_{\on{et}}^\bullet(X,\Delta^!(A_{X\times X}(1)))$$ 
equals the image of the 4-cocycle in 
$$\chi_1(\lambda)\otimes \chi_2(\lambda) \in 
\on{C}^\bullet_{\on{et}}(X,(A_1)_X(-1)[-2] \otimes (A_2)_X(-1)[-2])\simeq  
\on{C}^\bullet_{\on{et}}(X,\Delta^!((A_1)_{X\times X})\otimes \Delta^!((A_2)_{X\times X})),$$ 
along the map
\begin{multline*} 
\Delta^!((A_1)_{X\times X})\otimes \Delta^!((A_2)_{X\times X}) \to
\Delta^!((A_1)_{X\times X})\otimes (A_2)_{X\times X} \to \\
\to \Delta^!((A_1)_{X\times X})\otimes \Delta^*((A_2)_{X\times X})\to 
\Delta^!((A_1\otimes A_2)_{X\times X})\overset{b'}\to \Delta^!(A_{X\times X}(1)).
\end{multline*} 

We note that the latter map is canonically identified with
\begin{multline*} 
\Delta^!((A_1)_{X\times X})\otimes \Delta^!((A_2)_{X\times X}) \to
(A_1)_{X\times X}\otimes \Delta^!((A_2)_{X\times X}) \to \\
\to \Delta^*((A_1)_{X\times X})\otimes \Delta^!((A_2)_{X\times X}) \to 
\Delta^!((A_1\otimes A_2)_{X\times X})\overset{b'}\to \Delta^!(A_{X\times X}(1)).
\end{multline*} 

\medskip

The exchange of roles of $A_1$ and $A_2$ thus acts as the swap of the two factors in $X\times X$, resulting in the automorphism 
$(-1)$ on $\omega_X^{-1}\simeq \CO(\Delta)|_\Delta$. 

\end{rem} 

\sssec{}

Fix now two elements $\lambda_1,\lambda_2\in \Lambda$, and let us describe explicitly the resulting
gerbe $\CG^{\lambda_1,\lambda_2}$. We claim that we will obtain a canonical identification
\begin{equation} \label{e:id G lambda 2}
\CG^{\lambda_1,\lambda_2}\simeq \left((\omega_X^{\otimes -1})^{q(\lambda_1)}\boxtimes 
(\omega_X^{\otimes -1})^{q(\lambda_2)}\right)\otimes 
\CO(\Delta)^{b'(\chi_1(\lambda_1),\chi_2(\lambda_2))}\otimes \CO(\Delta)^{b'(\chi_1(\lambda_2),\chi_2(\lambda_1))}.
\end{equation} 

\medskip

This would imply that the symmetric bilinear form attached to $\CG$ by the procedure of \secref{ss:tori expl} equals $b$
of \eqref{e:b from b'}.

\sssec{}

We take $S=X\times X$ and $I=\{1,2\}$, where the two maps $S\to X$ are the two projections. 
The subset $\Gamma_I \subset S\times X$ identifies with $X\times X\underset{\Delta}\sqcup\, X\times X$.
It admits a normalization 
$$\wt\Gamma_I\simeq (X\times X) \sqcup (X\times X) \overset{s}\to \Gamma_I.$$

Denote the composite map
$$\wt\Gamma_I\to \Gamma_I\to S\times X$$
by $\wt\iota$.

\medskip

Denote the two maps
$$X\times X \to \Gamma_I$$
by $s_i$. We have
$$\iota\circ s_1=\Delta_{1,3} \text{ and } \iota\circ s_2=\Delta_{2,3}.$$ 
Let $s_{1,2}$ denote the diagonal map
$$X\to \Gamma_I,$$
so that $\iota\circ s_{1,2}$ is the main diagonal $\Delta_{1,2,3}$. 

\sssec{}

We consider the map
\begin{equation} \label{e:diag lambda 2}
(X\times X)\times X \to B_{\on{et}}(T),
\end{equation} 
corresponding to the $T$-bundle $(\lambda_1\cdot \CO(\Delta_{1,3}))\otimes (\lambda_2\cdot \CO(\Delta_{2,3}))$,
equipped with its natural trivialization on $U_I=(X\times X)\times X-(\Delta_{1,3}\cup \Delta_{2,3})$. 

\medskip

The composition of \eqref{e:diag lambda 2} with \eqref{e:chi} (for $i=1$) defines a 2-cocycle 
\begin{equation} \label{e:2-cocyle 2}
\on{C}^\bullet_{\on{et}}(\Gamma_I,\iota^!((A_1)_{(X\times X)\times X})).
\end{equation} 

We have a distinguished triangle 
$$(s_{1,2})_!((A_1)_X(-2)[-4])\to (s_1)_!((A_1)_{X\times X}(-1)[-2]) \oplus (s_2)_!((A_1)_{X\times X}(-1)[-2])  \to \iota^!((A_1)_{(X\times X)\times X}).$$ 

Hence, the above 2-cocycle canonically comes from a 2-cocycle in
$$\on{C}^\bullet_{\on{et}}(\wt\Gamma_I,\wt\iota^!((A_1)_{(X\times X)\times X}))\simeq 
\on{C}^\bullet_{\on{et}}(X\times X,(A_1)_{X\times X}(-1)[-2])\oplus \on{C}^\bullet_{\on{et}}(X\times X,(A_1)_{X\times X}(-1)[-2]).$$

I.e., this cocycle corresponds to a pair of elements in $A_1(-1)$. The computation in \secref{sss:identify 2} implies that this pair of 
elements is given by
\begin{equation} \label{e:components}
(\chi_1(\lambda_1),\chi_1(\lambda_2))\in A_1(-1)\oplus A_1(-1).
\end{equation} 

\sssec{}

From here, using the cup-product manipulation as in \secref{sss:cup prod}, we obtain that the 4-cocycle in 
\begin{equation} \label{e:4-cocyle 2}
\on{C}^\bullet_{\on{et}}(\Gamma_I,\iota^!(A_{(X\times X)\times X}(1))),
\end{equation} 
corresponding to the composition of \eqref{e:diag lambda 2} with \eqref{e:cup product}, also comes from a 4-cocycle in 
\begin{equation} \label{e:4-cocyle 2 tilde}
\on{C}^\bullet_{\on{et}}(\wt\Gamma_I,\wt\iota^!(A_{(X\times X)\times X})(1))\simeq 
\on{C}^\bullet_{\on{et}}(X\times X,A_{X\times X}[-2])\oplus \on{C}^\bullet_{\on{et}}(X\times X,A_{X\times X}[-2]),
\end{equation} 
with components obtained via  
$$\on{C}^\bullet_{\on{et}}(X\times X,(A_2)_{X\times X}[-2]) \overset{b'(\chi_1(\lambda_1),-)} \longrightarrow 
\on{C}^\bullet_{\on{et}}(X\times X,A_{X\times X}[-2])$$
and  
$$\on{C}^\bullet_{\on{et}}(X\times X,(A_2)_{X\times X}[-2]) \overset{b'(\chi_1(\lambda_2),-)} \longrightarrow 
\on{C}^\bullet_{\on{et}}(X\times X,A_{X\times X}[-2]),$$
respectively, from the 2-cocycles in 
$$\on{C}^\bullet_{\on{et}}(X\times X,(A_2)_{X\times X}[-2])$$
attached by the procedure of \secref{sss:line bundles to gerbes} to the $T$-bundles
\begin{equation} \label{e:T-bundles 2}
X\times X \overset{\Delta_{i,3}}\longrightarrow (X\times X)\times X \overset{\text{\eqref{e:diag lambda 2}}}\longrightarrow B_{\on{et}}(T), \quad i=1,2
\end{equation} 
and $\chi_2:\Lambda\to A_2(-1)$. 

\sssec{}

The $T$-bundles in \eqref{e:T-bundles 2} are
$$\lambda_1\cdot (\omega_X^{\otimes -1}\boxtimes \CO_X) \otimes \lambda_2\cdot \CO(\Delta) \text{ and }
\lambda_2\cdot (\CO_X\boxtimes \omega_X^{\otimes -1}) \otimes \lambda_1\cdot \CO(\Delta),$$
respectively. 

\medskip

Hence, we obtain that the 4-cocycle in \eqref{e:4-cocyle 2 tilde} has components in 
$$\on{C}^\bullet_{\on{et}}(X\times X,A_{X\times X}[-2])\oplus \on{C}^\bullet_{\on{et}}(X\times X,A_{X\times X}[-2]),$$
thought of as a pair of $A$-gerbes on $X\times X$, equal to
$$(\omega_X^{\otimes -1}\boxtimes \CO_X)^{b'(\chi_1(\lambda_1),\chi_2(\lambda_1))}\otimes
\CO(\Delta)^{b'(\chi_1(\lambda_1),\chi_2(\lambda_2))}$$ 
and 
$$(\CO_X\boxtimes \omega_X^{\otimes -1})^{b'(\chi_1(\lambda_2),\chi_2(\lambda_2))}
\otimes
\CO(\Delta)^{b'(\chi_1(\lambda_2),\chi_2(\lambda_1))},$$
respectively. 

\medskip

Finally, the trace map
$$\on{C}^\bullet_{\on{et}}(\wt\Gamma_I,\wt\iota^!(A_{S\times X}(1))) \simeq
\on{C}^\bullet_{\on{et}}(\wt\Gamma_I,s^!\circ \pi^!(A_S)[-2]) \to \on{C}^\bullet_{\on{et}}(S,A_S[-2]),$$
thought of as a map
$$\on{C}^\bullet_{\on{et}}(X\times X,A_{X\times X}[-2])\oplus \on{C}^\bullet_{\on{et}}(X\times X,A_{X\times X}[-2])\to
\on{C}^\bullet_{\on{et}}(X\times X,A_{X\times X}[-2]),$$
is the sum of two identity maps.

\medskip

This implies the specified description of \eqref{e:id G lambda 2}. 

\sssec{}

We have showed that the symmetric bilinear form corresponding to $\CG$ via the procedure of 
\secref{ss:tori expl} equals the form $b$ from \eqref{e:b from b'}. We will now proceed to showing
that the quadratic form equals $q$.  For that we will have to describe explicitly the identifications
$$\CG^{\lambda_1,\lambda_2}|_{\Delta}\simeq \CG^{\lambda_1+\lambda_2}$$
of \eqref{e:diag compat of gerbes 2} and the datum of commutativity of the corresponding diagram \eqref{e:recover quad}

\medskip

Unwinding the definitions, we obtain that the resulting map 
\begin{multline*}
(\omega_X^{\otimes -1})^{q(\lambda_1+\lambda_2)}\simeq 
(\omega_X^{\otimes -1})^{q(\lambda_1)+q(\lambda_2)+b(\lambda_1,\lambda_2)}\simeq \\
\simeq (\omega_X^{\otimes -1})^{b'(\chi_1(\lambda_1),\chi_2(\lambda_1))+b'(\chi_1(\lambda_1),\chi_2(\lambda_2))+
b'(\chi_1(\lambda_2),\chi_2(\lambda_2))+b'(\chi_1(\lambda_2),\chi_2(\lambda_1))}\simeq \\
\simeq
\biggl((\omega_X^{\otimes -1}\boxtimes \CO_X)^{b'(\chi_1(\lambda_1),\chi_2(\lambda_1))}\otimes
\CO(\Delta)^{b'(\chi_1(\lambda_1),\chi_2(\lambda_2))}\otimes \\
\otimes (\CO_X\boxtimes \omega_X^{\otimes -1})^{b'(\chi_1(\lambda_2),\chi_2(\lambda_2))}
\otimes \CO(\Delta)^{b'(\chi_1(\lambda_2),\chi_2(\lambda_1))}\biggr)|_\Delta \simeq \\
\simeq \CG^{\lambda_1,\lambda_2}|_{\Delta}\simeq \CG^{\lambda_1+\lambda_2} \simeq (\omega_X^{\otimes -1})^{q(\lambda_1+\lambda_2)}
\end{multline*}
equals the tautological map tensored with the $A_{2\on{-tors}}$-torsor 
$$(-1)^{b'(\chi_1(\lambda_1),\chi_2(\lambda_2))}.$$

This torsor comes from the fact that the map 
$$\iota\circ s_1:X\times X\to X\times X\times X$$
acts as
$$(x_1,x_2)\mapsto (x_1,x_2,x_1),$$
which involves the transposition of the last two factors.

\medskip

Furthermore, the datum of commutativity of the diagram \eqref{e:transposition unequal diag} is given by the identification 
of the $A_{2\on{-tors}}$-torsors 
\begin{multline*}
(-1)^{b'(\chi_1(\lambda_1),\chi_2(\lambda_2))}\otimes (-1)^{b(\lambda_1,\lambda_2)}\simeq
(-1)^{-b'(\chi_1(\lambda_1),\chi_2(\lambda_2))}\otimes (-1)^{b(\lambda_1,\lambda_2)}\simeq \\
\simeq (-1)^{b'(\chi_1(\lambda_2),\chi_2(\lambda_1))},
\end{multline*}
where the last isomorphism comes from
$$b(\lambda_1,\lambda_2)=b'(\chi_1(\lambda_1),\chi_2(\lambda_2))+b'(\chi_1(\lambda_2),\chi_2(\lambda_1)).$$

From here we obtain that the datum of commutativity of the diagram \eqref{e:recover quad} is given by the identification
$$(-1)^{q(\lambda)} \otimes (-1)^{b(\lambda,\lambda)} \simeq (-1)^{-q(\lambda)} \otimes (-1)^{b(\lambda)}\simeq (-1)^{q(\lambda)},$$
where the last isomorphism comes from
$$b(\lambda,\lambda)=2q(\lambda),$$
as required. 

\ssec{Proof of \propref{p:parameterization} for tori} \label{ss:proof for tori}

We will now give an alternative proof of \propref{p:parameterization} in the special case of tori.

\sssec{}

From the commutative diagram \eqref{e:compare forms}, we obtain a commutative diagram 
$$
\CD
\Maps_{\on{Ptd}(\on{PreStk})}(B_{\on{et}}(T) \times X,B^4_{\on{et}}(A(1)))  @>{\text{\eqref{e:from BG infty}}}>> \on{FactGe}_{A}(\Gr_T)  \\
@VVV   @VV{\text{\secref{ss:tori expl}}}V   \\
\on{Quad}(\Lambda,A(-1))  @>{\on{Id}}>>  \on{Quad}(\Lambda,A(-1)),
\endCD 
$$
where the left vertical arrow corresponds to the projection 
$$H^4_{\on{et}}(B_{\on{et}}(T) \times X;\on{pt}\times X,A(1))\to H^4_{\on{et}}(B_{\on{et}}(T);\on{pt},A(1))\simeq \on{Quad}(\Lambda,A(-1)).$$

\medskip

Let $\on{FactGe}^0_{A}(\Gr_T)$ denote the fiber of the right vertical arrow. From the distinguished triangle 
\eqref{e:loc triv gerbes again} we obtain that the fiber of the left vertical arrow identifies canonically with
$\Maps\left(X,B^2_{\on{et}}(\Hom(\Lambda,A))\right)$. 

\medskip

Hence, it remains to show that the induced map 
\begin{equation} \label{e:0 gerbes}
\Maps\left(X,B^2_{\on{et}}(\Hom(\Lambda,A))\right)\to \on{FactGe}^0_{A}(\Gr_T)
\end{equation} 
is an isomorphism. 

\medskip

We will deduce this from the description of $\on{FactGe}_{A}(\Gr_T)$
in \secref{sss:factor gerbes tori}.

\sssec{} \label{sss:0 gerbes}

Namely, the groupoid $\on{FactGe}^0_{A}(\Gr_T)$ is isomorphic to that of
assignments 
$$\lambda \mapsto \CG^\lambda\in \on{Ge}_A(X),$$
equipped with the following pieces of data:

\begin{itemize} 

\item One is \emph{multiplicativity}, i.e., we must be given 
isomorphisms of gerbes 
$$\CG^{\lambda_1+\lambda_2}\simeq \CG^{\lambda_1}\otimes \CG^{\lambda_2}$$
that are associative in the natural sense.

\item The other one is that of \emph{commutativity}, i.e., 
we must be given the data of commutativity for the squares
\begin{equation} \label{e:com lambda12}
\CD 
\CG^{\lambda_1+\lambda_2} @>>>  \CG^{\lambda_1}\otimes \CG^{\lambda_2} \\
@VVV  @VVV  \\
\CG^{\lambda_2+\lambda_1} @>>>  \CG^{\lambda_2}\otimes \CG^{\lambda_1} 
\endCD
\end{equation}
that satisfy the hexagon axiom.

\end{itemize}

In addition, the following conditions must be satisfied:

\begin{enumerate} 

\item
The datum of commutativity for the outer square in 
\begin{equation} \label{e:com lambda sq}
\CD 
\CG^{\lambda_1+\lambda_2} @>>>  \CG^{\lambda_1}\otimes \CG^{\lambda_2} \\
@VVV  @VVV  \\
\CG^{\lambda_2+\lambda_1} @>>>  \CG^{\lambda_2}\otimes \CG^{\lambda_1}  \\
@VVV  @VVV \\
\CG^{\lambda_1+\lambda_2} @>>>  \CG^{\lambda_1}\otimes \CG^{\lambda_2} 
\endCD
\end{equation}
is the identity one. 

\item 
The datum of commutativity in \eqref{e:com lambda12} for $\lambda_1=\lambda=\lambda_2$
\begin{equation} \label{e:com lambda}
\CD 
\CG^{2\lambda} @>>>  \CG^{\lambda}\otimes \CG^{\lambda} \\
@V{\on{id}}VV  @VV{\on{id}}V  \\
\CG^{2\lambda} @>>>  \CG^{\lambda}\otimes \CG^{\lambda} 
\endCD
\end{equation}
is the identity one. 

\end{enumerate}

\sssec{}

The above description implies that for $T\simeq T_1\times T_2$, the natural map
$$\on{FactGe}^0_{A}(\Gr_{T_1})\times \on{FactGe}^0_{A}(\Gr_{T_2})\to \on{FactGe}^0_{A}(\Gr_{T})$$
is an isomorphism.

\medskip

Hence, it is sufficient to show that the map \eqref{e:0 gerbes} is an equivalence for $T=\BG_m$.

\sssec{}

Now, the description in \secref{sss:0 gerbes} implies that for $T=\BG_m$ (so $\Lambda=\BZ$)
we have the obvious equivalence
$$\on{FactGe}^0_{A}(\Gr_{\BG_m})\simeq \on{Ge}_A(X),$$
given by
$$\CG \mapsto \CG^1, \quad 1\in \BZ,$$
and the composition
$$\on{Ge}_A(X)\simeq \Maps\left(X,B^2_{\on{et}}(\Hom(\BZ,A))\right)\to \on{FactGe}^0_{A}(\Gr_{\BG_m})\to \on{Ge}_A(X)$$
is the identity map.

\medskip

Hence, \eqref{e:0 gerbes} is an isomorphism for $\BG_m$. 

\ssec{Relation to $\Theta$-data}

In this section we will make contact with the paper \cite{Zhao}, and describe the category $\on{FactGe}_{A}(\Gr_T)$
in terms of what the author of {\it loc. cit.} calls $\Theta$ data for the lattice $\Lambda$. 

\sssec{} \label{sss:Theta}

Following \cite[Sect. 5.3.5]{Zhao}, we let $\Theta(\Lambda)$ be the space of the following data:

\begin{itemize}

\item A quadratic form $q:\Lambda \to A(-1)$; we denote the associated symmetric bilinear form by $b$; 

\item An assignment $\lambda\in \Lambda \rightsquigarrow \CG^\lambda\in \on{Ge}_A(X)$;

\item A system of isomorphisms 
$$\CG^{\lambda_1+\lambda_2}\overset{c_{\lambda_1,\lambda_2}}\simeq \CG^{\lambda_1}\otimes \CG^{\lambda_2}\otimes
(\omega_X^{-1})^{b(\lambda_1,\lambda_2)},$$
endowed with an associativity constraint;

\item A datum $h_{\lambda_1,\lambda_2}$ of commutativity for the squares
$$
\CD
\CG^{\lambda_1+\lambda_2} @>{c_{\lambda_1,\lambda_2}}>>  \left(\CG^{\lambda_1}\otimes \CG^{\lambda_2}\right) \otimes (\omega_X^{-1})^{b(\lambda_1,\lambda_2)} \\
@V{\on{id}}VV  @VV{\on{id}\otimes (-1)^{b(\lambda_1,\lambda_2)}}V \\
\CG^{\lambda_2+\lambda_1} @>{c_{\lambda_2,\lambda_1}}>>  \left(\CG^{\lambda_2}\otimes \CG^{\lambda_1}\right)\otimes (\omega_X^{-1})^{b(\lambda_2,\lambda_1)},
\endCD
$$
where $(-1)^{b(\lambda_1,\lambda_2)}$ is as in \secref{sss:-1}, equipped with a datum of compatibility with the associativity constraint, and that squares to the
identity; 

\item For $\lambda_1=\lambda=\lambda_2$, the datum of identification of $h_{\lambda,\lambda}$ in the diagram 
$$
\CD
\CG^{\lambda+\lambda} @>{c_{\lambda,\lambda}}>>  \left(\CG^{\lambda}\otimes \CG^{\lambda}\right) \otimes (\omega_X^{-1})^{b(\lambda,\lambda)} \\
@V{\on{id}}VV  @VV{\on{id}\otimes (-1)^{b(\lambda,\lambda)}}V \\
\CG^{\lambda+\lambda} @>{c_{\lambda,\lambda}}>>  \left(\CG^{\lambda}\otimes \CG^{\lambda}\right)\otimes (\omega_X^{-1})^{b(\lambda,\lambda)},
\endCD
$$
with the trivialization of $(-1)^{b(\lambda,\lambda)}$ resulting from the identity $b(\lambda,\lambda)=2q(\lambda)$. 

\end{itemize} 

\bigskip

Let $\Theta^0(\Lambda)$ denote the fiber of the natural projection
$$\Theta(\Lambda)\to \on{Quad}(\Lambda,A(-1)).$$

We have 
$$\Theta^0(\Lambda)\simeq \Maps(X,B^2_{\on{et}}(\Hom(\Lambda,A))).$$

\sssec{} \label{sss:theta basis}

Note that if $\Lambda$ is equipped with a basis $e_1,...,e_n$, then evaluation on basis elements defines a map
$$\Theta(\Lambda) \to (\on{Ge}_A(X))^{\times n},$$
and the resulting map
$$\Theta(\Lambda) \to \on{Quad}(\Lambda,A(-1))\times (\on{Ge}_A(X))^{\times n}$$
is an equivalence. 

\medskip

Indeed, it fits into a map of fiber sequences
$$
\CD
\Theta^0(\Lambda) @>>> \Theta(\Lambda) @>>> \on{Quad}(\Lambda,A(-1)) \\
@VVV  @VVV @VV{\on{id}}V  \\
(\on{Ge}_A(X))^{\times n} @>>> \on{Quad}(\Lambda,A(-1))\times (\on{Ge}_A(X))^{\times n} @>>> \on{Quad}(\Lambda,A(-1)),
\endCD
$$
where the left vertical arrow is the isomorphism 
$$\Theta^0(\Lambda)\simeq \Maps(X,B^2_{\on{et}}(\Hom(\Lambda,A))) \simeq (\on{Ge}_A(X))^{\times n}.$$

\sssec{} \label{sss:theta T}

Following \cite[Lemma 5.6]{Zhao}, we claim that there is a canonical equivalence
\begin{equation} \label{e:theta}
\on{FactGe}_{A}(\Gr_T)\simeq \Theta(\Lambda).
\end{equation} 

Indeed, the description of $\on{FactGe}_{A}(\Gr_T)$ in \secref{ss:tori expl} provides a functor
$$\on{FactGe}_{A}(\Gr_T)\to \Theta(\Lambda).$$

Now, we have a morphism of fiber sequences
$$
\CD
\on{FactGe}^0_{A}(\Gr_T) @>>> \on{FactGe}_{A}(\Gr_T) @>>>  \on{Quad}(\Lambda,A(-1)) \\
@VVV @VVV @VV{\on{id}}V \\
\Theta^0(\Lambda) @>>> \Theta(\Lambda) @>>> \on{Quad}(\Lambda,A(-1)),
\endCD
$$
where the left vertical arrow is an equivalence by \secref{sss:0 gerbes}. 

\sssec{} \label{sss:theta G new 1}

Let now $G$ be a reductive group with $T$ as its Cartan subgroup. Assume that $A$ is divisible.
We define the category $\Theta(\Lambda)_G$
as follows\footnote{Our definition is different, yet equivalent to that in \cite[Sect.5.3.6]{Zhao}.}:

\medskip

An object of $\Theta(\Lambda)_G$ is an object of $\Theta(\Lambda)$, whose bilinear form $q$ belongs to
$$\on{Quad}(\Lambda,A(-1))^W_{\on{restr}}\subset \on{Quad}(\Lambda,A(-1)),$$
and which is endowed with isomorphisms
\begin{equation} \label{e:rigidification i}
\CG^{\alpha_i}\simeq (\omega_X^{\otimes -1})^{q(\alpha_i)}
\end{equation}
for every simple coroot $\alpha_i$. 

\sssec{} \label{sss:theta G new 2}

Restriction along the embedding $T\hookrightarrow G$ defined a map
$$\on{FactGe}_{A}(\Gr_G)\to \on{FactGe}_{A}(\Gr_T).$$

We claim that we have a naturally defined map
\begin{equation} \label{e:theta enh}
\on{FactGe}_{A}(\Gr_G)\to \Theta(\Lambda)_G
\end{equation} 
that makes the diagram
$$
\CD
\on{FactGe}_{A}(\Gr_G) @>>> \on{FactGe}_{A}(\Gr_T) \\
@VVV  @VVV \\
\Theta(\Lambda)_G @>>> \Theta(\Lambda)
\endCD
$$
commute. 

\medskip

Indeed, by \secref{sss:calc homotopy} we know that the composition
$$\on{FactGe}_{A}(\Gr_G) \to \on{FactGe}_{A}(\Gr_T) \to \on{Quad}(\Lambda,A(-1))$$
takes values in $\on{Quad}(\Lambda,A(-1))^W_{\on{restr}}$. 

\medskip

Hence, it remains to construct the data of \eqref{e:rigidification i}. The latter reduces to the
case of $G=SL_2$.

\sssec{} \label{sss:theta sl2}

Note that by \secref{sss:calc homotopy Ge}, any factorizable $A$-gerbe on $\Gr_{SL_2}$ is \emph{canonically} of the form 
$(\det_{SL_2,\on{St}})^{a}$ for some element $a\in A(-1)$, where $\det_{SL_2,\on{St}}$ is the determinant line bundle on $\Gr_{SL_2}$
corresponding to the action on the \emph{standard} representation. 

\medskip

For an integer $k$ let $\det_{\BG_m,k}$ denote the determinant line bundle on $\Gr_{\BG_m}$ associated with the action of $\BG_m$
on the one-dimensional vector space given by the $k$-th power of the tautological character. This is a $\BZ$-graded factorization line bundle,
and we note that the grading is even if $k$ is even.

\medskip

The restriction of $\det_{SL_2,\on{St}}$ to $\Gr_{\BG_m}$ identifies with $\det_{\BG_m,1}\otimes \det_{\BG_m,-1}$, and hence
the restriction of $(\det_{SL_2,\on{St}})^{a}$ to $\Gr_{\BG_m}$ identifies with $(\det_{\BG_m,1}\otimes \det_{\BG_m,-1})^a$.
The associated quadratic form 
$$q:\BZ\to A(-1)$$
takes value $a$ on the generator $1\in \BZ$. 

\medskip

In order to construct \eqref{e:rigidification i}, we have to show that the value of $(\det_{\BG_m,1}\otimes \det_{\BG_m,-1})^a$
on the generator $1\in \BZ=\Lambda$ identifies canonically with $(\omega_X^{\otimes -1})^a$. For that, it suffices to construct 
an isomorphism between the restriction of the line bundle $\det_{SL_2,\on{St}}$ to the section
$$X\to \Gr_{\BG_m,X}$$ 
corresponding to $1\in \BZ$ identifies canonically with $\omega_X^{\otimes -1}$.

\medskip

However, we indeed have a canonical isomorphism
$$\on{det.rel.}(\CO(x)\oplus \CO(-x),\CO\oplus \CO)\simeq \omega_X^{\otimes -1}|_x, \quad x\in X.$$

\sssec{}

Thus, the map \eqref{e:theta enh} has been constructed. It fits into a map of fiber sequences
\begin{equation} \label{e:theta enh diag}
\CD
\on{FactGe}^0_{A}(\Gr_{G}) @>>> \on{FactGe}_{A}(\Gr_{G}) @>>> \on{Quad}(\Lambda,A(-1))^W_{\on{restr}} \\
@VVV @VVV @VV{\on{id}}V \\
\Theta^0(\Lambda)_G @>>> \Theta(\Lambda)_G @>>> \on{Quad}(\Lambda,A(-1))^W_{\on{restr}}, 
\endCD
\end{equation}
where
$$\Theta^0(\Lambda)_G\simeq 
\on{Fib}\left(\Theta^0(\Lambda)\to \underset{i\in I}\Pi\, \on{Ge}_A(X)\right)\simeq \Maps(X,B^2_{\on{et}}(\Hom(\pi_{1,\on{alg}}(G),A))).$$

By unwinding the construction, one obtains that in terms of the above identification, the 
left vertical arrow in \eqref{e:theta enh diag} is the identity map on $\Maps(X,B^2_{\on{et}}(\Hom(\pi_{1,\on{alg}}(G),A)))$. 

\medskip

From here we obtain that the map \eqref{e:theta enh} is also an isomorphism.

\begin{rem}  \label{r:theta G bis}
In \cite[Theorem 5.4]{Zhao}, it is shown directly that the map \eqref{e:theta enh} are isomorphism. This, in turn, can be used to deduce 
\propref{p:parameterization} by reversing the steps. 
\end{rem}

\ssec{The notion of \emph{multiplicative} factorization gerbe} \label{ss:mult com}

In order to be able to state the metaplectic version of geometric Satake, we will need to discuss 
the notion of \emph{multiplicative} factorization gerbe, first on $\Gr_T$,
and then when the lattice $\Lambda=\Hom(\BG_m,T)$ is replaced by a general finitely generated
abelian group. 

\sssec{}

Note that since $T$ is commutative, $\Gr_T$ is naturally a factorization \emph{group}-prestack over $\Ran$.
Hence, along with $\on{FactGe}_{A}(\Gr_T)$, 
we can consider the corresponding space (in fact, commutative group in spaces)
\begin{equation} \label{e:mult and com T}
\on{FactGe}^{\on{mult}}_{A}(\Gr_T) 
\end{equation} 
that corresponds to gerbes that respect the group structure on $\Gr_T$ over $\Ran$.

\medskip

We have the evident forgetful map
\begin{equation} \label{e:forget mult T}
\on{FactGe}^{\on{mult}}_{A}(\Gr_T)\to \on{FactGe}_{A}(\Gr_T). 
\end{equation}

\medskip

Explicitly, a multiplicative structure on a gerbe $\CG$ is an identification 
$$\on{mult}^*(\CG)\simeq \CG\boxtimes \CG$$
as \emph{factorization gerbes} on $\Gr_T\underset{\Ran}\times \Gr_T$ (in the above formula $\on{mult}$ denotes the
multiplication map $\Gr_T\underset{\Ran}\times \Gr_T\to \Gr_T$), equipped with a compatibility datum over
triple product $\Gr_T\underset{\Ran}\times \Gr_T\underset{\Ran}\times \Gr_T$, and an identity satisfied over the quadruple 
product. 

\medskip

We will prove:

\begin{prop}  \label{p:classify mult T} 
The forgetful map
$$\on{FactGe}^{\on{mult}}_{A}(\Gr_T)\to \on{FactGe}_{A}(\Gr_T)$$
is fully faithful. Its essential image is the preimage under
$$\on{FactGe}_{A}(\Gr_T)\to \on{Quad}(\Lambda,A(-1))$$
of the subset consisting of those quadratic forms, whose
associated bilinear form is zero.
\end{prop}

\begin{proof}

We will use the description of factorization on gerbes on $\Gr_T$ given in \secref{sss:factor gerbes tori}.  
In these terms, the multiplicative structure on $\CG$
amounts to specifying isomorphisms
\begin{equation} \label{e:mult lambda}
\CG^{\lambda_1,\lambda_2}\simeq \CG^{\lambda_1}\boxtimes \CG^{\lambda_2}
\end{equation} 
equipped with an associativity constraint, and equipped with the datum of the identification of \eqref{e:mult lambda}
with the factorization isomorphism over $X\times X-\Delta$.

\medskip

In other words, we need that the factorization isomorphisms
$$\CG^{\lambda_1,\lambda_2}|_{X\times X-\Delta}\simeq \CG^{\lambda_1}\otimes \CG^{\lambda_2}|_{X\times X-\Delta}$$
extend to all of $X\times X$. If they extend, they do so uniquely, and the extended isomorphisms are automatically equipped with
an associativity constraint. 

\medskip

Thus, by \secref{sss:bilin expl}, we obtain that the category $\on{FactGe}^{\on{mult}}_{A}(\Gr_T)$ identifies with the full subcategory of 
$\on{FactGe}_{A}(\Gr_T)$, consisting of objects for which the bilinear form $b(-,-)$ vanishes.

\end{proof}

\begin{rem}
Note that the set of quadratic forms $q:\Lambda\to A$ whose associated bilinear form vanishes, is in bijection with the 
set of \emph{linear} maps $\Lambda\to A_{2\on{-tors}}$.

\medskip

Note also that we have a tautological identification 
\begin{equation} \label{e:2-tors tate}
A(-1)_{2\on{-tors}}\simeq A_{2\on{-tors}}, 
\end{equation}
since $\mu_2\simeq\pm 1\simeq  \BZ/2\BZ$
canonically.
\end{rem}

\sssec{} \label{sss:expl mult}

Note that it follows from \secref{sss:theta T} that we can describe the space 
$$\on{FactGe}^{\on{mult}}_{A}(\Gr_T)\simeq 
\on{Fib}\Bigl(\on{FactGe}_{A}(\Gr_T)\to \on{SymBilin}(\Lambda,A(-1))\Bigr)$$
as follows:

\medskip

It consists of 

\begin{itemize}

\item A linear map
$$q:\Lambda\to A_{2\on{-tors}}$$

\item An assignment 
$$\lambda\in \Lambda \rightsquigarrow \CG^\lambda\in \on{Ge}_A(X),$$

\item A system of isomorphisms 
$$\CG^{\lambda_1+\lambda_2}\overset{c_{\lambda_1,\lambda_2}}\simeq \CG^{\lambda_1}\otimes \CG^{\lambda_2},$$
equipped with an associativity constraint;

\item A datum $h_{\lambda_1,\lambda_2}$ of commutativity for the squares 
$$
\CD
\CG^{\lambda_1+\lambda_2} @>{c_{\lambda_1,\lambda_2}}>>  \CG^{\lambda_1}\otimes \CG^{\lambda_2} \\
@V{\on{id}}VV   @VV{\on{id}}V \\
\CG^{\lambda_2+\lambda_1} @>{c_{\lambda_2,\lambda_1}}>>  \CG^{\lambda_2}\otimes \CG^{\lambda_1},
\endCD
$$
compatible with the associativity constraint and that square to the identity;

\smallskip

\item For $\lambda_1=\lambda=\lambda_2$, we require that the datum of $h_{\lambda,\lambda}$, which in this case 
is the trivialization of the trivial $A_{2\on{-tors}}$-torsor, i.e., an element of $2\on{-tors}$, equals $q(\lambda).$

\end{itemize} 

\sssec{}  \label{sss:higher Ek}
Note that $\Gr_T$ is not just a group-prestack over $\Ran$, but a \emph{commutative} group-prestack. Hence, along with 
$$\on{FactGe}^{\on{mult}}_{A}(\Gr_T)=:\on{FactGe}^{\BE_1}_{A}(\Gr_T),$$
we can consider the spaces $\on{FactGe}^{\BE_k}_{A}(\Gr_T)$ for any $k\geq 1$ and also
$$\on{FactGe}^{\on{com}}_{A}(\Gr_T):= \on{FactGe}^{\BE_\infty}_{A}(\Gr_T):=\underset{k}{\on{lim}}\, \on{FactGe}^{\BE_k}_{A}(\Gr_T).$$

We claim, however, that the forgetful maps
$$\on{FactGe}^{\BE_k}_{A}(\Gr_T)\to \on{FactGe}^{\BE_1}_{A}(\Gr_T)$$
are all equivalences. 

\medskip

First off, the maps $\on{FactGe}^{\BE_{k+1}}_{A}(\Gr_T)\to \on{FactGe}^{\BE_k}_{A}(\Gr_T)$ are automatically 
equivalences for $k\geq 3$ because $A$-gerbes are 2-categorical objects. Similarly, the forgetful map
$\on{FactGe}^{\BE_{3}}_{A}(\Gr_T)\to \on{FactGe}^{\BE_2}_{A}(\Gr_T)$ is automatically fully faithful.

\medskip

An $\BE_2$-structure on a multiplicative gerbe $\CG$
translates as a datum of commutativity for the squares
\begin{equation} \label{e:lambda com}
\CD
\sigma^*(\CG^{\lambda_1,\lambda_2})  @>{\sigma^*\text{\eqref{e:mult lambda}}}>> \sigma^*(\CG^{\lambda_1}\boxtimes \CG^{\lambda_2}) \\
@VVV  @VVV  \\
\CG^{\lambda_2,\lambda_1}  @>{\text{\eqref{e:mult lambda}}}>> \CG^{\lambda_2}\boxtimes \CG^{\lambda_1}
\endCD
\end{equation} 
that coincides with the one coming from factorization over $X\times X-\Delta$.  

\medskip

Thus, we are already given the datum of commutation of \eqref{e:lambda com} over $X\times X-\Delta$. Therefore, this datum automatically
uniquely extends to all of $X\times X$. This implies that 
$$\on{FactGe}^{\BE_{2}}_{A}(\Gr_T)\to \on{FactGe}^{\BE_1}_{A}(\Gr_T)$$ is an equivalence.

\medskip

An object in $\on{FactGe}^{\BE_{2}}_{A}(\Gr_T)$ comes from $\on{FactGe}^{\BE_{3}}_{A}(\Gr_T)$ if and only if the diagrams \eqref{e:lambda com}
square to the identity, in the sense that the datum of commutativity for the outer square in 
$$
\CD
\CG^{\lambda_1,\lambda_2}  @>>> \CG^{\lambda_1}\boxtimes \CG^{\lambda_2}  \\
@VVV @VVV  \\
\sigma^*\circ \sigma^*(\CG^{\lambda_1,\lambda_2})  @>>> \sigma^*\circ \sigma^*(\CG^{\lambda_1}\boxtimes \CG^{\lambda_2}) \\
@VVV   @VVV  \\
\sigma^*(\CG^{\lambda_2,\lambda_1})  @>>> \sigma^*(\CG^{\lambda_2}\boxtimes \CG^{\lambda_1})  \\
@VVV  @VVV  \\
\CG^{\lambda_1,\lambda_2}  @>>> \CG^{\lambda_1}\boxtimes \CG^{\lambda_2}
\endCD
$$
is the tautological one. But this is automatic because this condition holds over $X\times X-\Delta$.

\begin{rem}  \label{r:Ek AG}
Note that from \propref{p:parameterization} we obtain the following a priori description of the groupoid $\on{FactGe}^{\BE_k}_{A}(\Gr_T)$
(here $k\geq 1$) as 
$$\Maps^{\BE_k}_{\on{PreStk}}(B_{\on{et}}(T) \times X,B^4_{\on{et}}(A(1))\simeq 
\Maps_{\on{Ptd}(\on{PreStk})}(B_{\on{et}}^{1+k}(T)\times X,B^{4+k}_{\on{et}}(A(1)).$$

From \secref{sss:higher Ek} we obtain that the looping map
$$\Maps_{\on{Ptd}(\on{PreStk})}(B_{\on{et}}^{1+k}(T)\times X,B^{4+k}_{\on{et}}(A(1))\to 
\Maps_{\on{Ptd}(\on{PreStk})}(B_{\on{et}}(T)\times X,B^4_{\on{et}}(A(1))$$
has the following properties: 

\begin{itemize}

\item Both sides have vanishing homotopy groups $\pi_i$ for $i\geq 3$, i.e.,
$$H^{i+k}_{\on{et}}(B^{1+k}_{\on{et}}(T)\times X;\on{pt}\times X,A(1))=
H^i_{\on{et}}(B_{\on{et}}(T)\times X;\on{pt}\times X,A(1))=0$$
for $i\leq 1$. 

\item It induces an isomorphism on $\pi_2$ for any $k\geq 1$, i.e., the map 
$$H^{2+k}_{\on{et}}(B^{1+k}_{\on{et}}(T)\times X;\on{pt}\times X,A(1))\to H^2_{\on{et}}(B_{\on{et}}(T)\times X;\on{pt}\times X,A(1))$$
is an isomorphism (note that the RHS identifies with $\Hom(\Lambda,A)$). 

\item It induces an isomorphism on $\pi_1$ for any $k\geq 1$, i.e., the map 
$$H^{3+k}_{\on{et}}(B^{1+k}_{\on{et}}(T)\times X;\on{pt}\times X,A(1))\to H^3_{\on{et}}(B_{\on{et}}(T)\times X;\on{pt}\times X ,A(1))$$
is an isomorphism; in fact both sides are isomorphic to $H^1_{\on{et}}(X,\Hom(\Lambda,A))$. 

\item For any $k\geq 1$, the induced map on $\pi_0$, i.e., the map 
$$H^{4+k}_{\on{et}}(B^{1+k}_{\on{et}}(T)\times X;\on{pt}\times X,A(1))\to H^4_{\on{et}}(B_{\on{et}}(T)\times X;\on{pt}\times X,A(1))$$
is injective with the image corresponding to the subset of $\on{Quad}(\Lambda,A(-1))$, consisting of those quadratic forms, whose
associated bilinear form is zero.

\end{itemize}

\end{rem} 

\begin{rem} \label{r:Ek top}

The isomorphisms of Remark \ref{r:Ek AG} are the \'etale counterparts of the corresponding isomorphisms
in the context of \emph{algebraic topology}, which we will now explain. (We will come back and do a similar 
analysis in the \'etale setting in \secref{ss:etale computation}.) 

\medskip

Let $T$ be a topological torus with coweight lattice $\Lambda$. We can think of $B(T) $ as $B^2(\Lambda)$. 

\medskip

We start with the groupoid
$$\Maps_{\on{Ptd}(\Spc)}(B(\Lambda),B^3(A))\simeq \Maps_{\on{Grp}(\Spc)}(\Lambda,B^2(A)).$$

We can think of its objects as monoidal categories $\CC$ that are groupoids such that 
$\pi_0(\CC)=\Lambda$ (as monoids) and $\pi_1({\bf 1}_\CC)=A$ (as groups). 

\medskip

A datum of lifting of such a point to a point of 
$$\Maps_{\on{Ptd}(\Spc)}(B^2(\Lambda),B^4(A))\simeq \Maps_{\BE_1(\Spc)}(B(\Lambda),B^3(A))$$
amounts to endowing the monoidal category $\CC$ with a braiding. A further lift to an object of
$$\Maps_{\on{Ptd}(\Spc)}(B^{2+k}(\Lambda),B^{4+k}(A))\simeq \Maps_{\BE_{k+1}(\Spc)}(B(\Lambda),B^3(A))$$
for $k\geq 1$ amounts to the \emph{condition} that the resulting braided monoidal category
be \emph{symmetric}. This already implies that the forgetful map
$$\Maps_{\on{Ptd}(\Spc)}(B^{2+k+1}(\Lambda),B^{4+k+1}(A))\to \Maps_{\on{Ptd}(\Spc)}(B^{2+k}(\Lambda),B^{4+k}(A))$$
is an isomorphism for $k\geq 1$ and is fully faithful for $k=0$.

\medskip

Moreover, for $k\geq 0$, 
the group $\pi_2\left(\Maps_{\on{Ptd}(\Spc)}(B^{2+k}(\Lambda),B^{4+k}(A))\right)$
identifies with 
$$\Maps_{\on{Grp}(\Spc)}(\Lambda,A)=\Hom_{\on{Ab}}(\Lambda,A),$$ and $\pi_1\left(\Maps_{\on{Ptd}(\Spc)}(B^{2+k}(\Lambda),B^{4+k}(A))\right)$ 
identifies with
$$\Maps_{\BE_2(\Spc)}(\Lambda,B(A))\simeq \Maps_{\BE_\infty(\Spc)}(\Lambda,B(A))=\Maps_{\on{Ab}}(\Lambda,B(A))=0,$$
where $\on{Ab}$ denotes the $\infty$-category of chain complexes of abelian groups. 

\medskip

Finally, the set of isomorphism classes of braided monoidal categories as above is in bijection with
$\on{Quad}(\Lambda,A)$. Indeed, for a given $\CC$, the corresponding bilinear form $b(\lambda_1,\lambda_2)$
is recovered as the square of the braiding 
$$c^{\lambda_1}\otimes c^{\lambda_2}\to c^{\lambda_2}\otimes c^{\lambda_1}\to c^{\lambda_1}\otimes c^{\lambda_2},$$
and the quadratic form $q(\lambda)$ is recovered as the value of the braiding
$$c^\lambda\otimes c^\lambda\to c^\lambda\otimes c^\lambda.$$

In particular, this braided monoidal category is symmetric if and only if $b(-,-)=0$. 

\end{rem}

\sssec{}

Consider the connective spectrum 
$$\Maps_{\BE_\infty(\Spc)}(\Lambda,B^2(A)).$$

It follows 
from Remark \ref{r:Ek top} that it fits into a fiber sequence
\begin{equation} \label{e:maps_to_B2}
B^2(\Hom(\Lambda,A))\to \Maps_{\BE_\infty(\Spc)}(\Lambda,B^2(A)) \to \Hom(\Lambda,A_{2\on{-tors}}).
\end{equation}

\begin{rem}

Note that since $\Lambda$ is projective in the category of abelian groups, we have
$$B^2(\Hom(\Lambda,A))\simeq \Maps_{\on{Ab}}(\Lambda,B^2(A)).$$

So the map $B^2(\Hom(\Lambda,A))\to \Maps_{\BE_\infty(\Spc)}(\Lambda,B^2(A))$ can be interpreted as a map
$$\Maps_{\on{Ab}}(\Lambda,B^2(A))\to \Maps_{\BE_\infty(\Spc)}(\Lambda,B^2(A))$$
given by the Dold-Kan functor $\on{Ab}^{\leq 0}\to \BE_\infty(\Spc)$.

\end{rem} 

\sssec{}

Similarly, we have a fiber sequence 
\begin{equation} \label{e:maps_to_B2X}
\Maps(X,B^2_{\on{et}}(\Hom(\Lambda,A)))\to \Maps_{\BE_\infty(\Spc)}(\Lambda,\on{Ge}_A(X)) \to \Hom(\Lambda,A_{2\on{-tors}}),
\end{equation}
so that the diagram
$$
\CD
B^2(\Hom(\Lambda,A)) @>>> \Maps_{\BE_\infty(\Spc)}(\Lambda,B^2(A))  \\
@VVV @VVV \\
\Maps(X,B^2_{\on{et}}(\Hom(\Lambda,A))) @>>> \Maps_{\BE_\infty(\Spc)}(\Lambda,\on{Ge}_A(X)) 
\endCD
$$
is a push-out square. 

\sssec{}

The explicit description of $\on{FactGe}^{\on{mult}}_{A}(\Gr_T)$ in \secref{sss:expl mult} implies that we have a canonical identification
\begin{equation} \label{e:Einfty to Theta}
\Maps_{\BE_\infty(\Spc)}(\Lambda,\on{Ge}_A(X))\simeq \on{FactGe}^{\BE_\infty}_{A}(\Gr_T)\simeq \on{FactGe}^{\on{mult}}_{A}(\Gr_T)
\end{equation} 
that fits into the commutative diagram 
$$
\CD
\Maps(X,B^2_{\on{et}}(\Hom(\Lambda,A))) @>>> \Maps_{\BE_\infty(\Spc)}(\Lambda,\on{Ge}_A(X)) \\
@V{\sim}VV @VV{\sim}V \\
\on{FactGe}^0_{A}(\Gr_T) @>>> \on{FactGe}^{\on{mult}}_{A}(\Gr_T).
\endCD
$$

Finally, the fiber sequence \eqref{e:maps_to_B2X} is compatible with the fiber sequence
$$\on{FactGe}^0_{A}(\Gr_T) \to  \on{FactGe}_{A}(\Gr_T)  \to  \on{Quad}(\Lambda,A(-1))$$
via the commutative diagram 
$$ 
\CD
\Maps(X,B^2_{\on{et}}(\Hom(\Lambda,A))) @>>> \Maps_{\BE_\infty(\Spc)}(\Lambda,\on{Ge}_A(X))  @>>>  \Hom(\Lambda,A_{2\on{-tors}}) \\
@V{\sim}VV @VV{\sim}V @VV{=}V \\
\on{FactGe}^0_{A}(\Gr_T) @>>> \on{FactGe}^{\on{mult}}_{A}(\Gr_T) @>>> \Hom(\Lambda,A_{2\on{-tors}}) \\
@V{=}VV @VVV @VVV \\
\on{FactGe}^0_{A}(\Gr_T) @>>>  \on{FactGe}_{A}(\Gr_T)  @>>>  \on{Quad}(\Lambda,A(-1)).
\endCD
$$

\sssec{}

In \secref{ss:splitting} we will see that if $A_{2\on{-tors}}\simeq \BZ/2\BZ$, there a \emph{canonical} identification
$$\Maps_{\BE_\infty(\Spc)}(\Lambda,B^2(A)) \simeq B^2(\Hom(\Lambda,A))\times  \Hom(\Lambda,\BZ/2\BZ).$$

This implies that for $A_{2\on{-tors}}\simeq \BZ/2\BZ$, we have a canonical identification 
$$\on{FactGe}^{\on{mult}}_{A}(\Gr_T)\simeq \on{FactGe}^0_{A}(\Gr_T)
\times  \Hom(\Lambda,\BZ/2\BZ).$$

\ssec{More general abelian groups}

In this section we generalize the discussion of \secref{ss:mult com} to the case when instead of a
lattice $\Lambda$ (thought of as a lattice of cocharacters of a torus) we take a general finitely 
generated abelian group. We need this in order to state the metaplectic version of geometric Satake. 

\sssec{}  \label{sss:Gamma as quot}

Let $\Gamma$ be a finitely generated abelian group whose torsion part is of order prime to $\on{char}(k)$. Let 
$$\Gr_{\Gamma\otimes \BG_m}$$
be the group-prestack defined as in \secref{s:finite Gr}. 

\medskip

It s basic feature is that if $\Gamma$ is written as $\Lambda_1/\Lambda_2$, where $\Lambda_2\subset \Lambda_1$ are lattices, then
we have a map
\begin{equation} \label{e:finite as quotient}
\Gr_{T_1}/\Gr_{T_2}\to \Gr_{\Gamma\otimes \BG_m},
\end{equation} 
which becomes an isomorphism after sheafification in the topology generated by finite surjective maps (in the above formula 
$T_i$ is the torus $\Lambda_i\otimes \BG_m$. 

\medskip

In particular, pullback with respect to \eqref{e:finite as quotient} defines an equivalence on the categories of $A$-gerbes, $A$-torsors, etc. 

\sssec{}

The group-prestack $\Gr_{\Gamma\otimes \BG_m}$ has a natural factorization structure over $\Ran$, so we can talk about the space
$\on{FactGe}_{A}(\Gr_{\Gamma\otimes \BG_m})$.

\medskip

Since $\Gr_{\Gamma\otimes \BG_m}$ is a (commutative) 
group-prestack over $\Ran$, along with $\on{FactGe}_{A}(\Gr_{\Gamma\otimes \BG_m})$, 
we can consider the space (in fact, commutative group in spaces)
\begin{equation} \label{e:mult arb}
\on{FactGe}^{\on{mult}}_{A}(\Gr_{\Gamma\otimes \BG_m}),
\end{equation}
that correspond to gerbes that respect that group structure on $\Gr_{\Gamma\otimes \BG_m}$ over $\Ran$.

\begin{rem} \label{r:com mult arb}
Along with $\on{FactGe}^{\on{mult}}_{A}(\Gr_{\Gamma\otimes \BG_m})$, one can also consider its variants
$$\on{FactGe}^{\BE_k}_{A}(\Gr_{\Gamma\otimes \BG_m}),\,\, \on{FactGe}^{\BE_\infty}_{A}(\Gr_{\Gamma\otimes \BG_m})\simeq
\on{FactGe}^{\on{com}}_{A}(\Gr_{\Gamma\otimes \BG_m}).$$

However, as in \secref{sss:higher Ek}, one shows that the forgetful maps
$$\on{FactGe}^{\BE_k}_{A}(\Gr_{\Gamma\otimes \BG_m})\to \on{FactGe}^{\BE_1}_{A}(\Gr_{\Gamma\otimes \BG_m})=
\on{FactGe}^{\on{mult}}_{A}(\Gr_{\Gamma\otimes \BG_m})$$
are equivalences for all $k\geq 1$.
\end{rem} 

\sssec{}

The following results from \propref{p:classify mult T}: 

\begin{cor}   \label{c:when descends}
Let $\Gamma$ be written as a quotient of two lattices as in \secref{sss:Gamma as quot}. Let $\CG_1$ be a factorization 
$A$-gerbe on $\Gr_{T_1}$, and let $b_1$ and $q_1$ be the associated bilinear and quadratic forms on $\Lambda_1$, 
respectively.   

\smallskip

\noindent{\em(a)} The gerbe $\CG_1$ can be descended to a factorization gerbe $\CG$ on $\Gr_{\Gamma\otimes \BG_m}$ 
\emph{only if}  $b_1(\Lambda_2,-)=0$. 

\smallskip

\noindent{\em(a')} In the situation of \emph{(a)}, a descent exists \'etale-locally on $X$ if and only if the restriction of $q_1$ to 
$\Lambda_2$ is trivial. 

\smallskip

\noindent{\em(a'')} In the situation of \emph{(a')}, a descent datum is equivalent to the trivialization 
of $\CG_2:=\CG_1|_{\Gr_{T_2}}$ as a factorization gerbe on $\Gr_{T_2}$. 

\smallskip

\noindent{\em(b)} In the situation of \emph{(a'')}, the descended gerbe $\CG$ admits a multiplicative structure 
if and only if $b_1$ is trivial. In the latter case, the multiplicative structure
is unique up to a unique isomorphism.

\end{cor}

From here, we obtain: 

\begin{cor}  \label{c:classify mult arb new new} 
We have a canonically defined map 
\begin{equation} \label{e:quad finite}
\on{FactGe}_{A}(\Gr_{\Gamma\otimes \BG_m})\to \on{Quad}(\Gamma,A(-1))
\end{equation} 
such that:

\medskip

\noindent{\em(a)} The forgetful map 
$$\on{FactGe}^{\on{mult}}_{A}(\Gr_{\Gamma\otimes \BG_m})\to \on{FactGe}_{A}(\Gr_{\Gamma\otimes \BG_m})$$
fits in the fiber square 
\begin{equation}  \label{e:mult Gamma}
\CD 
\on{FactGe}^{\on{mult}}_{A}(\Gr_{\Gamma\otimes \BG_m})  @>>> \on{FactGe}_{A}(\Gr_{\Gamma\otimes \BG_m}) \\
@VVV  @VVV  \\
\Hom(\Gamma,A)_{2\on{-tors}}  @>>> \on{Quad}(\Gamma,A(-1)).
\endCD
\end{equation}

\medskip

\noindent{\em(b)}
The fiber of the map \eqref{e:quad finite}, denoted $\on{FactGe}^0_{A}(\Gr_{\Gamma\otimes \BG_m})$, consists
of objects that are \'etale-locally trivial on $X$. 

\medskip

\noindent{\em(c)} We have a push-out square
$$
\CD
\Maps_{\on{Ab}}(\Gamma,B^2(A)) @>>> \Maps_{\BE_\infty(\Spc)}(\Gamma,B^2(A))  \\
@VVV @VVV \\
\on{FactGe}^0_{A}(\Gr_{\Gamma\otimes \BG_m}) @>>> \on{FactGe}^{\on{mult}}_{A}(\Gr_{\Gamma\otimes \BG_m}) 
\endCD
$$

\end{cor}

\begin{rem} \label{r:classify mult arb new new}

Note that it follows from \corref{c:classify mult arb new new}(b) that we have a canonical isomorphism
$$\on{FactGe}^0_{A}(\Gr_{\Gamma\otimes \BG_m})  \simeq 
\Maps_{\on{Ab}}(\Gamma,\on{C}^\bullet_{\on{et}}(X,A)[2])\simeq
\Maps(X,\Maps_{\on{Ab}}(\Gamma,B^2(A))_{\on{et}}),$$
where we denote by $\Maps_{\on{Ab}}(\Gamma,B^2(A))_{\on{et}}$ the \'etale sheafification of the constant presheaf
with value $\Maps_{\on{Ab}}(\Gamma,B^2(A))$. 

\medskip

Note also that $\Maps_{\on{Ab}}(\Gamma,B^2(A))$ has non-trivial homotopy groups in degrees $2$ and $1$, equal to
$\Hom(\Gamma,A)$ and $\Ext^1(\Gamma,A)$, respectively.

\medskip

Similarly, it follows from \corref{c:classify mult arb new new}(c) that we have
$$\on{FactGe}^{\on{mult}}_{A}(\Gr_{\Gamma\otimes \BG_m}) \simeq
\Maps_{\BE_\infty(\Spc)}(\Gamma,\on{Ge}_A(X))\simeq 
\Maps(X,\Maps_{\BE_\infty(\Spc)}(\Gamma,B^2(A))_{\on{et}}),$$
where we denote by $\Maps_{\BE_\infty(\Spc)}(\Gamma,B^2(A))_{\on{et}}$ the \'etale sheafification of the constant presheaf
with value $\Maps_{\BE_\infty(\Spc)}(\Gamma,B^2(A))$. 

\medskip

As we have seen above, the non-trivial homotopy groups of $\Maps_{\BE_\infty(\Spc)}(\Gamma,B^2(A))$ are in degrees $2$, $1$ and $0$,
where the homotopy groups in degrees $-2$ and $-1$ are the same as for $\Maps_{\on{Ab}}(\Gamma,B^2(A))$, while in degree $0$ we have
$\Hom(\Gamma,A_{2\on{-tors}})$. 

\end{rem}

\sssec{}   \label{sss:more general param}   \label{sss:com fact G}

Let now $G$ be a connective reductive group. By \secref{ss:grp to pi1}, we have a canonically defined map 
\begin{equation} \label{e:map of Gr new}
\Gr_G \to \Gr_{\Gamma\otimes \BG_m},
\end{equation}
compatible with the factorization structure. 

%
%
%


\medskip

From here we obtain a map 
\begin{equation} \label{e:com fact G real}
\on{FactGe}_{A}(\Gr_{\pi_{1,\on{alg}}(G)\otimes \BG_m})\to \on{FactGe}_{A}(\Gr_G).
\end{equation} 


\medskip

By construction, the diagram 
$$
\CD
\on{FactGe}_{A}(\Gr_{\pi_{1,\on{alg}}(G)\otimes \BG_m})  @>>>  \on{FactGe}_{A}(\Gr_G)  \\
@VVV   @VVV  \\
\on{Quad}(\pi_{1,\on{alg}}(G),A(-1)) @>>> \on{Quad}(\Lambda,A(-1))^W_{\on{restr}}
\endCD
$$
commutes. Hence, we obtain a map
\begin{equation} \label{e:com fact G real neutral}
\on{FactGe}^0_{A}(\Gr_{\pi_{1,\on{alg}}(G)\otimes \BG_m})\to \on{FactGe}^0_{A}(\Gr_G).
\end{equation} 

\medskip

From Corollaries \ref{c:parameterization neutral} and \ref{c:classify mult arb new new}(b) and Remark \ref{r:classify mult arb new new} we obtain:

\begin{cor}   \label{c:com fact G T}
The map \eqref{e:com fact G real neutral} is an isomorphism.
\end{cor}

From here we obtain:

\begin{cor}
The map \eqref{e:com fact G real} is fully faithful. 
\end{cor} 

\ssec{Splitting multiplicative gerbes}  \label{ss:more mult}  \label{ss:splitting}

In this subsection we will assume that $A_{2\on{-tors}}=\BZ/2\BZ$. (Note that this happens, e.g., if $A\subset E^\times$.)

\medskip

We will show that in this case the fiber sequence
$$\on{FactGe}^0_{A}(\Gr_{\Gamma\otimes \BG_m}) \to \on{FactGe}^{\on{mult}}_{A}(\Gr_{\Gamma\otimes \BG_m}) \to 
\Hom(\Gamma,A_{2\on{-tors}})$$
admits a canonical splitting, functorial in $\Gamma$. 

\sssec{}  \label{sss:univ case}

According to \corref{c:classify mult arb new new}(c), it suffices to construct a splitting of the fiber sequence
$$B^2(\Hom(\Gamma,A))\to \Maps_{\BE_\infty(\Spc)}(\Gamma,B^2(A))\to \Hom(\Gamma,A_{2\on{-tors}}).$$

By functoriality, it suffices to treat the universal case: i.e., when $\Gamma=\BZ/2\BZ$ and we need to construct an object
of $$\Maps_{\BE_\infty(\Spc)}(\BZ/2\BZ,B^2(\BZ/2\BZ))$$
that projects to the identity map $\BZ/2\BZ\to \BZ/2\BZ$.

\sssec{}
 
We will construct the sought-for object in $\Maps_{\BE_\infty(\Spc)}(\BZ/2\BZ,B^2(\BZ/2\BZ))$ as a symmetric monoidal groupoid $\CC$ 
with $\pi_0(\CC)\simeq \pi_1(\CC)\simeq \BZ/2\BZ$.

\medskip

As a monoidal groupoid, we set $\CC$ be the product $\BZ/2\BZ\times B(\BZ/2\BZ)$. A braided monoidal structure on such a $\CC$
is equivalent to a choice of a bilinear form $b'$ on $\BZ/2\BZ$ with values in $\BZ/2\BZ$. We set it to be 
$$b'(\ol{1},\ol{1})=\ol{1}.$$

The resulting braided monoidal structure is automatically symmetric, and the associated quadratic form $q:\BZ/2\BZ\to \BZ/2\BZ$
is the identity map, as required. 

\sssec{}  \label{sss:splitting spec}  

In what follows, for a given element $\epsilon\in \Hom(\Gamma,A_{2\on{-tors}})$, we will denote by $\CG^\epsilon$ the resulting
multiplicative factorization gerbe on $\Gr_{\Gamma\otimes \BG_m}$. 

\medskip

For $\Gamma=\BZ/2\BZ$ and the identity map, we will denote this gerbe by $\CG^{\epsilon_{\on{taut}}}$. We will refer to it
as the \emph{sign gerbe}. 

\begin{rem} \label{r:does not factor}

Note that $\CG^\epsilon$, viewed as a gerbe on $\Gr_{\Gamma\otimes \BG_m}$, equipped with the multiplicative structure, 
admits a canonical trivialization. However, this trivialization is \emph{not} compatible with the factorization structure. 

\end{rem} 

\sssec{}  \label{sss:splitting}

For a given object $\CG\in \on{FactGe}^{\on{mult}}_{A}(\Gr_{\Gamma\otimes \BG_m})$ let us denote by $\epsilon$ the map
$$\Gamma\to A_{2\on{-tors}}\simeq \BZ/2\BZ$$
that measures the obstruction of $\CG$ to belong to $\on{FactGe}^0_{A}(\Gr_{\Gamma\otimes \BG_m})$.

\medskip

We obtain that, canonically attached to $\CG$, there exists an object
$$\CG^0\in \on{FactGe}^0_{A}(\Gr_{\Gamma\otimes \BG_m})\simeq \Maps(X,B^2_{\on{et}}(\Hom(\Gamma,A))),$$
such that
$$\CG\simeq \CG^0\otimes \CG^\epsilon,$$
where $\CG^\epsilon$ is as in \secref{sss:splitting spec}.

\ssec{More on the sign gerbe}

In this subsection we will make a digression, needed for the sequel, in which we will 
discuss several manipulations with the gerbe $\CG^{\epsilon_{\on{taut}}}$ introduced in 
\secref{ss:splitting}.

\sssec{} \label{sss:sq root}

Let $Z$ be a prestack over $\Ran$, equipped with a factorization structure, and equipped with a map
to $\Gr_{\BZ/2\BZ\otimes \BG_m}$, compatible with the factorization structures. 

\medskip

Let $\CL$ be a line bundle on $Z$. We equip it with a $\BZ/2\BZ$-graded structure as follows:
for an affine test scheme $S$ and a map $S\to Z$ such that the composite $S\to Z\to \Gr_{\BZ/2\BZ\otimes \BG_m}$
maps to the even/odd connected component, the grading on $\CL|_S$ is even/odd.

\medskip

Let us be given a factorization structure on $\CL$, viewed as a $\BZ/2\BZ$-graded line bundle. Note that $\CL^{\otimes 2}$
is then a plain factorization line bundle. Assume that $\on{char}(k)\neq 2$, and consider the $\BZ/2\BZ$-gerbe on $Z$ given by 
$(\CL^{\otimes 2})^{\frac{1}{2}}$, equipped with its natural factorization structure.

\medskip

It is easy to see, however, that $(\CL^{\otimes 2})^{\frac{1}{2}}$ identifies canonically with $\CG^{\epsilon_{\on{taut}}}|_Z$. Indeed,
both gerbes are canonically trivialized as plain gerbes, and the factorization structures on both are given by the
sign rules. 

\sssec{An example}

Let us take $Z=\Gr_{\BG_m}$. We can take as $\CL$ the \emph{determinant line bundle} on $\Gr_{\BG_m}$,
denoted $\det_{\BG_m,\on{St}}$, corresponding to the standard one-dimensional representation of $\BG_m$. 

\sssec{}

Let $\CC$ be a sheaf of categories over $\Gr_{\BZ/2\BZ\otimes \BG_m}$, and let $\CC$ be equipped with a factorization
structure, compatible with the factorization structure on $\Gr_{\BZ/2\BZ\otimes \BG_m}$.


\medskip

Viewing $\BZ/2\BZ$ as 2-torsion in $E^{\times}$, and using the twisting construction of \secref{sss:twist sheaves of categ}, 
we can twist $\CC$ by $\CG^{\epsilon_{\on{taut}}}$ and obtain a new factorization sheaf of categories, denoted 
$\CC^{\epsilon_{\on{taut}}}$. 

\medskip

Suppose that in the above situation $\CC$ is endowed with a monoidal (resp., symmetric monoidal) 
structure, compatible with the group structure on $\Gr_{\BZ/2\BZ\otimes \BG_m}$.
Then $\CC^{\epsilon_{\on{taut}}}$ also acquires a monoidal  (resp., symmetric monoidal) structure. 

\begin{rem}
Note that for any $S\to \Ran$, the corresponding categories $\CC(S)$ and $\CC^{\epsilon_{\on{taut}}}(S)$ are canonically identified
(since $\CG^{\epsilon_{\on{taut}}}$, viewed as a plain gerbe, is trivial). However, the factorization structures on $\CC(S)$ and $\CC^{\epsilon_{\on{taut}}}(S)$
are different. The same applies to the monoidal (resp., symmetric monoidal) situation. 
\end{rem} 

\sssec{}  \label{sss:shift and sign}

Assume for a moment that the structure on $\CC$ of sheaf of categories over $\Gr_{\BZ/2\BZ\otimes \BG_m}$ has been 
refined to that of sheaf over $\Gr_{\BG_m}$, also compatible with the factorization structures.

\medskip

Note that $\Gr_{\BG_m}$ carries a locally constant function, denoted $d$, given by the degree. Hence, we have a
well-defined endo-functor on $\CC$,
\begin{equation} \label{e:shift functor}
c\mapsto c[d].
\end{equation}

Note that this functor is \emph{not} compatible with the factorization structure, due to sign rules. 
However, when we view \eqref{e:shift functor} as a functor 
$$\CC\to \CC^{\epsilon_{\on{taut}}},$$
it is an equivalence of factorization categories.

\section{Jacquet functors for factorization gerbes}  \label{s:Jacquet}

In this section we take $G$ to be reductive.  We will study the interaction between factorization gerbes on $\Gr_G$ 
and those on $\Gr_M$, where $M$ is the Levi quotient of a parabolic of $G$. 

\ssec{The naive Jacquet functor} \label{ss:Jacquet}

Let $P$ be a parabolic subgroup of $G$, and we let $P\twoheadrightarrow M$ be its Levi quotient.
Let $N_P$ denote the unipotent radical of $P$. 

\sssec{}

Consider the diagram of the Grassmannians
$$\Gr_G \overset{\sfp}\longleftarrow \Gr_P \overset{\sfq}\longrightarrow \Gr_M.$$

\medskip

We claim that pullback along $\sfq$ defines an equivalence,
\begin{equation} \label{e:gerbe along N gr}
\on{Ge}_A(S\underset{\Ran}\times \Gr_M)\to \on{Ge}_A(S\underset{\Ran}\times \Gr_P)
\end{equation} 
for any $S\to\Ran$, in particular, inducing an equivalence
$$\on{FactGe}_{A}(\Gr_M)\to \on{FactGe}_{A}(\Gr_P).$$

\sssec{}

To show that \eqref{e:gerbe along N gr} is an equivalence, 
let us choose a splitting $M\hookrightarrow P$ of the projection $P\twoheadrightarrow M$. In particular, we obtain
an adjoint action of $M$ on $N_P$. Hence, we obtain an action of the group-prestack $\fL^+(M)$ (see \secref{sss:arc space} for the definition
of this group-prestack) over $\Ran$ on $\Gr_{N_P}$. 

\medskip

We can view $\Gr_M$ as a quotient $\fL(M)/\fL^+(M)$ (see \secref{sss:BL}), and hence we can view the map
$$\fL(M)\to \Gr_M$$
as a $\fL^+(M)$-torsor. Then $\Gr_P$, when viewed as a prestack over $\Gr_M$ is obtained 
by twisting $\Gr_{N_P}$ by the above $\fL^+(M)$-torsor. 

\medskip

Now, the equivalence in \eqref{e:gerbe along N gr}
follows from the fact that for any $S\to \Ran$, pullback defines an isomorphism
$$H^i_{\on{et}}(S,A)\to H^i_{\on{et}}(S\underset{\Ran}\times \Gr_{N_P},A)$$
for all $i$. 

\sssec{}

In terms of the parameterization given by \propref{p:parameterization}, the map 
$$\on{FactGe}_{A}(\Gr_G)\to \on{FactGe}_{A}(\Gr_M)$$
can be interpreted as follows:

\medskip

It corresponds to the map
\begin{multline*} 
\Maps_{\on{Ptd}(\on{PreStk})}(B_{\on{et}}(G) \times X,B^4_{\on{et}}(A(1))) \to \\ \to 
\Maps_{\on{Ptd}(\on{PreStk}_{/X})}(B(P)\times X,B^4_{\on{et}}(A(1)))\overset{\sim}\longleftarrow 
\Maps_{\on{Ptd}(\on{PreStk}_{/X})}(B(M)\times X,B^4_{\on{et}}(A(1))),
\end{multline*} 
where the second arrow is an isomorphism since the map $B(P)\to B(M)$ induces an isomorphism in \'etale cohomology with
constant coefficients.

\medskip

Thus, if $\CG^G$ is a factorization $A$-gerbe on $\Gr_G$, and $\CG^M$ is the corresponding the factorization $A$-gerbe on $\Gr_M$,
the corresponding quadratic forms
$$q:\Lambda\to A(-1)$$
coincide. 

\sssec{}   \label{sss:gerbes G and M}

We now take $A:=E^{\times,\on{tors}}$.
Given a factorization $E^{\times,\on{tors}}$-gerbe $\CG^G$ over $\Gr_G$, consider its pullback to $\Gr_P$, denoted $\CG^P$.  We let $\CG^M$
denote the canonically defined factorization gerbe on $\Gr_M$, whose pullback to $\Gr_P$ gives $\CG^P$. 

\medskip

By construction, for any $S\to \Ran$, we have a well-defined pullback functor
$$\sfp^!:\Shv_{\CG^G}(S\underset{\Ran}\times \Gr_G)\to \Shv_{\CG^P}(S\underset{\Ran}\times \Gr_P).$$

\medskip

Furthermore, since the
morphism $\sfq$ is ind-schematic, we have a well-defined push-forward functor
$$\sfq_*:\Shv_{\CG^P}(S\underset{\Ran}\times \Gr_P)\to 
\Shv_{\CG^M}(S\underset{\Ran}\times \Gr_M).$$

\medskip

Thus, the composite $\sfq_*\circ \sfp^!$ defines a map between factorization sheaves of categories
\begin{equation} \label{e:naive Jacquet}
\Shv_{\CG^G}(\Gr_G)_{/\Ran}\to \Shv_{\CG^M}(\Gr_M)_{/\Ran}.
\end{equation}

We will refer to \eqref{e:naive Jacquet} as the \emph{naive Jacquet functor}. 

\ssec{The critical twist}

 The functor \eqref{e:naive Jacquet} is not quite what we need for the purposes of geometric Satake. 
Namely, we will need to correct this functor by a cohomological shift that depends on the connected component of 
$\Gr_M$ (this is needed in order to arrange that the corresponding functor on the spherical categories maps perverse sheaves
to perverse sheaves). However, this cohomological shift will destroy the compatibility of the Jacquet functor
with factorization, due to sign rules. In order to compensate for this, we will apply an additional twist of
our categories by the square root of the determinant line bundle.

\medskip

The nature of this additional twist will be explained in the present subsection. 

\medskip

For the rest of this subsection we will assume that $\on{char}(k)\neq 2$. 

\sssec{}   \label{sss:det}

Let $\det_\fg$ denote the determinant line bundle on $\Gr_G$, corresponding to the adjoint representation. 
It is constructed as follows. For an affine test 
scheme $S$ and an $S$-point $I\subset \Maps(S,X)$ of $\Ran$, consider the corresponding $G$-bundle
$\CP_G$ on $S\times X$, equipped with an isomorphism
$$\alpha:\CP_G\simeq \CP^0_G$$
over $U_I\subset S\times X$.  Consider the corresponding vector bundles associated with the adjoint representation
$$\fg_{\CP_G}|_{U_I}\simeq \fg_{\CP^0_G}|_{U_I}.$$

\medskip

Then 
\begin{equation} \label{e:rel det}
\on{det.rel.}(\fg_{\CP_G},\fg_{\CP^0_G})
\end{equation}
is a well-defined line bundle\footnote{Note that the line bundle \eqref{e:rel det} is a priori $\BZ$-graded, but since $G$ is reductive,
and in particular, unimodular, this grading is actually trivial (i.e., concentrated in degree $0$).} on $S$.  

\medskip

This construction is compatible with pullbacks under $S'\to S$, thereby giving rise to the sought-for line bundle 
$\det_\fg$ on $\Gr_G$. 

\medskip

It is easy to see that $\det_\fg$ is equipped with a factorization structure over $\Ran$. 

\sssec{}

Consider the factorization $\BZ/2\BZ$-gerbe $\det^{\frac{1}{2}}_\fg$ over $\Gr_G$. 

\medskip

\noindent{\bf Convention:} From now on we will choose a square root, denoted $\omega^{\otimes \frac{1}{2}}_X$ of the canonical line bundle $\omega_X$ on $X$
(see again Remark \ref{r:roots} for our notational conventions).

\medskip

Let $P$ be again a parabolic of $G$. Consider the factorization gerbes 
$\det^{\frac{1}{2}}_\fg|_{\Gr_P}$ and  $\det^{\frac{1}{2}}_\fm|_{\Gr_P}$
over $\Gr_P$. We claim that the choice of $\omega^{\otimes \frac{1}{2}}_X$ gives rise to an identification of the gerbes 
\begin{equation} \label{e:gerbes on P}
\det^{\frac{1}{2}}_\fg|_{\Gr_P} \simeq \det^{\frac{1}{2}}_\fm|_{\Gr_P}\otimes \CG^{\epsilon_P}|_{\Gr_P},
\end{equation} 
where $\CG^{\epsilon_P}$ is the $\BZ/2\BZ$-gerbe on $\Gr_{M/[M,M]}$ corresponding to the map $\epsilon_P$
$$\Lambda_{M/[M,M]}\overset{2\check\rho_{G,M}}\longrightarrow \BZ \to \BZ/2\BZ,$$
where $2\check\rho_{G,M}:M/[M,M]\to \BG_m$ is the determinant of the action of $M$ on $\fn(P)$. 

\medskip

In fact, we claim that the ratio of the line bundles $\det_\fg|_{\Gr_P}$ and  $\det_\fm|_{\Gr_P}$, i.e.,
$$\det_\fg|_{\Gr_P} \otimes (\det_\fm|_{\Gr_P})^{\otimes -1},$$
admits a square root, to be denoted $\det_{\fn(P)}$, which is a $\BZ$-graded
(and, in particular, $\BZ/2\BZ$-graded) factorization line bundle on $\Gr_P$, with the grading given by the map
\begin{equation} \label{e:deg on GrP}
\Gr_P\to \Gr_M \to \Gr_{M/[M,M]} \overset{2\check\rho_{G,M}}\longrightarrow  \Gr_{\BG_m},
\end{equation} 
see Sects. \ref{sss:shift and sign} and \ref{sss:sq root}. 

\begin{rem}
In fact, more is true: the construction of \cite[Sect. 4]{BD2} defines a square root of $\det_\fg$ itself,
which is as a factorization $\BZ/2\BZ$-graded line bundle, where the grading is given by the map
$$\Gr_G\to \Gr_{\pi_{1,\on{alg}}(G)\otimes \BG_m} \to \Gr_{\BZ/2\BZ\otimes \BG_m},$$
where $\pi_{1,\on{alg}}(G)\to \BZ/2\BZ$ is is the canonical map that fits into the diagram
$$
\CD
\Lambda @>>> \pi_{1,\on{alg}}(G) \\
@V{2\check\rho}VV  @VVV  \\
\BZ  @>>>  \BZ/2\BZ,
\endCD
$$
where $2\check\rho$ is the sum of positive roots. 
\end{rem} 

\sssec{}  \label{sss:compare det}

The graded line bundle $\det_{\fn(P)}$ is constructed as follows. For an $S$-point  $(I,\CP_P,\CP_G|_{U_I}\simeq \CP^0_G|_{U_I})$
of $\Gr_P$ we set the value of $\det_{\fn(P)}$ on $S$ to be 
$$\on{rel.det.}(\fn(P)_{\CP_P}\otimes \omega^{\otimes \frac{1}{2}}_X,\fn(P)_{\CP^0_P}\otimes \omega^{\otimes \frac{1}{2}}_X).$$

\medskip

Let us construct the isomorphism
$$(\det_{\fn(P)})^{\otimes 2}\otimes \det_\fm|_{\Gr_P}\simeq \det_\fg|_{\Gr_P}.$$

\medskip

Let us identify the vector space $\fg/\fp$ with the dual of
$\fn(P)$ (say, using the Killing form). For an $S$-point 
$(I,\CP_P,\CP_G|_{U_I}\simeq \CP^0_G|_{U_I})$
of $\Gr_P$, denote
$$\CE:=\fn(P)_{\CP_P} \text{ and } \CE_0:=\fn(P)_{\CP^0_P}.$$

Then the ratio of $\det_\fg|_S$ and $\det_\fm|_S$ identifies with the line bundle
$$\on{rel.det.}(\CE,\CE_0)\otimes \on{rel.det.}(\CE^\vee,\CE^\vee_0)\simeq
\on{rel.det.}(\CE,\CE_0)\otimes \on{rel.det.}(\CE_0^\vee,\CE^\vee)^{\otimes -1}.$$.

Note, however, that for any line bundle $\CL$ on $S\times X$, we have
$$\on{rel.det.}(\CE,\CE_0)\otimes \on{rel.det.}(\CE_0^\vee,\CE^\vee)^{\otimes -1}\simeq
\on{rel.det.}(\CE\otimes \CL,\CE_0\otimes \CL)\otimes \on{rel.det.}(\CE_0^\vee\otimes \CL,\CE^\vee\otimes \CL)^{\otimes -1}.$$

Thus, letting $\CL$ be the pullback of $\omega^{\otimes \frac{1}{2}}_X$, we thus need to construct an isomorphism
$$\on{rel.det.}(\CE\otimes \omega^{\otimes \frac{1}{2}}_X,\CE_0\otimes \omega^{\otimes \frac{1}{2}}_X)\simeq 
\on{rel.det.}(\CE_0^\vee\otimes \omega^{\otimes \frac{1}{2}}_X,\CE^\vee\otimes \omega^{\otimes \frac{1}{2}}_X)^{-1}.$$

However, this follows from the (relative to $S$) local Serre duality on $S\times X$:
$$\BD^{\on{Serre}}_{/S}(\CE\otimes \omega^{\otimes \frac{1}{2}}_X)\simeq \CE^\vee\otimes \omega^{\otimes \frac{1}{2}}_X[1]
\text{ and }  \BD^{\on{Serre}}_{/S}(\CE_0\otimes \omega^{\otimes \frac{1}{2}}_X)\simeq \CE_0^\vee\otimes \omega^{\otimes \frac{1}{2}}_X[1].$$

\begin{rem} \label{r:spin}

The construction of the isomorphism \eqref{e:gerbes on P} depended on the choice of 
$\omega^{\otimes \frac{1}{2}}_X$.

\medskip

By analyzing the above construction one can show that the discrepancy between the two sides in \eqref{e:gerbes on P}
is canonically isomorphic to the factorizable $\BZ/2\BZ$-gerbe pulled back via
$$\Gr_P\to \Gr_M\to \Gr_{M/[M,M]} \overset{2\check\rho_{G,M}}\to \Gr_{\BG_m}$$
from the object in
$$\on{FactGe}^0_{\BZ/2\BZ}(\Gr_{\BG_m})\subset \on{FactGe}_{\BZ/2\BZ}(\Gr_{\BG_m})$$
attached to the \emph{gerbe} 
$$\omega^{\frac{1}{2}}_X\in \on{Ge}_{\BZ/2\BZ}(X)$$
via
$$\on{Ge}_{\BZ/2\BZ}(X)\simeq \on{FactGe}^0_{\BZ/2\BZ}(\Gr_{\BG_m}).$$

\end{rem}

\ssec{The corrected Jacquet functor}

We will now use the gerbe $\CG^{\epsilon_P}$ from the previous subsection in order to introduce a correction to
the naive Jacquet functor from \secref{sss:gerbes G and M}.

\sssec{}

Let $d_{G,M}:\Gr_P\to \BZ$ be locally constant function on $\Gr_P$ corresponding to the map \eqref{e:deg on GrP},
see \secref{sss:shift and sign}. 

\medskip

Given a factorization $E^{\times,\on{tors}}$-gerbe $\CG^G$ on $\Gr_G$ and the corresponding factorization gerbe $\CG^M$
on $\Gr_M$ (see \secref{sss:gerbes G and M}), we will now define the corrected Jacquet functor as a map between
factorization sheaves of categories:
\begin{equation} \label{e:corrected Jacquet}
J^G_M:\Shv_{\CG^G\otimes \det_\fg^{\frac{1}{2}}}(\Gr_G)_{/\Ran}\to 
\Shv_{\CG^M\otimes \det^{\frac{1}{2}}_\fm}(\Gr_M)_{/\Ran}. 
\end{equation}

\sssec{}

Namely, $J^G_M$ is the composition of the following four factorizable operations:

\medskip

\noindent(i) The pullback functor
$$\sfp^!:\Shv_{\CP^G\otimes \det^{\frac{1}{2}}_\fg}(\Gr_G)_{/\Ran}\to 
\Shv_{(\CP^G\otimes \det^{\frac{1}{2}}_\fg)|_{\Gr_P}}(\Gr_P)_{/\Ran};$$

\medskip

\noindent(ii) The identification 
$$\Shv_{(\CP^G\otimes \det^{\frac{1}{2}}_\fg)|_{\Gr_P}}(\Gr_P)_{/\Ran}\simeq
\Shv_{(\CG^M\otimes \det^{\frac{1}{2}}_\fm\otimes \CG^{\epsilon_P})|_{\Gr_P}}(\Gr_P)_{/\Ran},$$
given by the isomorphism of gerbes \eqref{e:gerbes on P};

\medskip

\noindent(iii) The cohomological shift functor $\CF\mapsto \CF[-d_{G,M}]$ 
$$\Shv_{(\CG^M\otimes \det^{\frac{1}{2}}_\fm\otimes \CG^{\epsilon_P})|_{\Gr_P}}(\Gr_P)_{/\Ran}\to
\Shv_{(\CG^M\otimes \det^{\frac{1}{2}}_\fm)|_{\Gr_P}}(\Gr_P)_{/\Ran},$$
see \secref{sss:shift and sign}. 

\medskip

\noindent(iv) The pushforward functor
$$\sfq_*:\Shv_{(\CG^M\otimes \det^{\frac{1}{2}}_\fm)|_{\Gr_P}}(\Gr_P)_{/\Ran}\to 
\Shv_{\CG^M\otimes \det^{\frac{1}{2}}_\fm}(\Gr_M)_{/\Ran}.$$

\begin{rem}
Note that since the identification \eqref{e:gerbes on P} depended on the choice of $\omega_X^{\otimes \frac{1}{2}}$,
so does the functor $J^G_M$. 
\end{rem} 

\section{The metaplectic Langlands dual datum}  \label{s:metap dual}

In section we take $G$ to be reductive\footnote{We will assume that the algebraic fundamental group of the derived 
group of $G$, i.e., the torsion part of $\pi_{1,\on{alg}}(G)$, is of order prime to  $\on{char}(k)$.}.  

\medskip

Given a factorization gerbe $\CG$ on $\Gr_G$, we will define the 
\emph{metaplectic Langlands dual datum} attached to $\CG$, and the corresponding notion of twisted local system on $X$. 

\ssec{The metaplectic Langlands dual \emph{root datum}}

The first component of the metaplectic Langlands dual datum is purely combinatorial and consists of a certain
root datum that only depends on the root datum of $G$ and $q$. This is essentially the same as the root datum defined by
G.~Lusztig as a recipient of the quantum Frobenius. 

\sssec{}

Given a factorization $A$-gerbe $\CG^G$ on $\Gr_G$, let  
$$q:\Lambda\to A(-1)$$
$$b:\Lambda\times \Lambda\to A(-1)$$
be the associated quadratic and bilinear forms, respectively. Let $\Lambda^\sharp \subset \Lambda$ be the kernel of $b$.
Let $\cLambda^\sharp$ be the dual of $\Lambda^\sharp$. Note that the inclusions
$$\Lambda^\sharp\subset \Lambda \text{ and } \cLambda\subset \cLambda^\sharp$$
induce isomorphisms after tensoring with $\BQ$. 

\medskip

Let
$$(\Delta \subset \Lambda,\,\check\Delta\subset \cLambda)$$
be the root datum for $G$. Following \cite{Lus}, we will now define a new root datum 
\begin{equation} \label{e:new root system}
(\Delta^\sharp\subset \Lambda^\sharp, \check\Delta^\sharp\subset \cLambda^\sharp).
\end{equation}

\sssec{}

We let $\Delta^\sharp$ be equal to $\Delta$ as an \emph{abstract set}. For each element $\alpha\in \Delta$, we let
the corresponding element $\alpha^\sharp\in \Delta^\sharp$
be equal to 
$$\on{ord}(q(\alpha))\cdot \alpha\in \Lambda,$$
and the corresponding element $\check\alpha^\sharp\in \check\Delta^\sharp$ be
$$\frac{1}{\on{ord}(q(\alpha))}\cdot \check\alpha\in  \cLambda\underset{\BZ}\otimes \BQ.$$

\medskip

The fact that $q$ lies in $\on{Quad}(\Lambda,A(-1))^W_{\on{restr}}$ implies that 
$\alpha^\sharp$ and $\check\alpha^\sharp$ defined in this way indeed belong to 
$\Lambda^\sharp\subset \Lambda$ and 
$\cLambda^\sharp\subset \cLambda\underset{\BZ}\otimes \BQ$, respectively. 

\sssec{}

Since $q$ is $W$-invariant, the action of $W$ on $\Lambda$ preserves $\Lambda^\sharp$. Moreover, 
for each $\alpha\in \Delta$, the action of the corresponding reflection $s_\alpha\in W$ on $\Lambda^\sharp$
equals that of $s_{\alpha^\sharp}$. 

\medskip

This implies that restriction defines an isomorphism from $W$ to the
group $W^\sharp$ of automorphisms of $\Lambda^\sharp$ generated by the elements $s_{\alpha^\sharp}$.

\medskip

Hence, \eqref{e:new root system} is a finite root system with Weyl group $W^\sharp$, isomorphic to the original
Weyl group $W$. 

\medskip

It follows from the construction that if $\alpha_i$ are the simple coroots of $\Delta$, then the corresponding elements
$\alpha_i^\sharp\in \Lambda^\sharp$ form a set of simple coroots of $\Delta^\sharp$.

\sssec{}

We let $G^\sharp$ denote the reductive group (over $k$) 
corresponding to \eqref{e:new root system}.

\ssec{The ``$\pi_1$-gerbe"}

Let $\CG^G$ be as above. In this subsection we will show that in addition to the reductive group $G^\sharp$, the datum of 
$\CG^G$ defines a certain \emph{multiplicative} factorization gerbe on the affine Grassmannian attached to
the abelian group $\pi_{1,\on{alg}}(G^\sharp)$. 

\sssec{}  \label{sss:construct gerbe pi1}

Let $\CG^T$ be the factorization gerbe on $\Gr_T$, corresponding to $\CG^G$ via \secref{sss:gerbes G and M}. 
Consider the corresponding torus $T^\sharp$.

\medskip

Let $\CG^{T^\sharp}$ be the factorization gerbe on $\Gr_{T^\sharp}$ equal to the pullback of $\CG^T$
under $T^\sharp\to T$. By \propref{p:classify mult T}, the gerbe $\CG^{T^\sharp}$ carries a canonical
multiplicative structure. 

\medskip

Consider the algebraic fundamental group $\pi_{1,\on{alg}}(G^\sharp)$ of $G^\sharp$, and the projection
$\Lambda^\sharp\to \pi_{1,\on{alg}}(G^\sharp)$. Consider the corresponding map
\begin{equation} \label{e:torus to center}
\Gr_{T^\sharp}\to \Gr_{\pi_{1,\on{alg}}(G^\sharp)\otimes \BG_m}.
\end{equation}

\medskip

We claim that there exists a canonically defined multiplicative factorization $A$-gerbe 
$\CG^{\pi_{1,\on{alg}}(G^\sharp)\otimes \BG_m}$ on $\Gr_{\pi_{1,\on{alg}}(G^\sharp)\otimes \BG_m}$, 
whose pullback under \eqref{e:torus to center} identifies with
$\CG^{T^\sharp}$. 

\sssec{}

By \corref{c:when descends}, we need to show that for every simple coroot $\alpha_i$, the pullback of $\CG^T$ to $\Gr_{\BG_m}$ under
\begin{equation} \label{e:root pullback}
\BG_m\overset{\alpha_i^\sharp}\longrightarrow T
\end{equation}
is trivialized. 

\medskip

By the transitivity of the construction in \secref{sss:gerbes G and M}, we can replace $G$ by its Levi subgroup $M_i$ of semi-simple rank $1$,
corresponding to $\alpha_i$. Furthermore, using the map $SL_2\to M_i$, we can assume that $G=SL_2$.

\sssec{}

Let $\det_{SL_2,\on{St}}$ and $\det_{\BG_m,k}$ be as in \secref{sss:theta sl2}. We can assume that object in $\on{FactGe}_A(\Gr_{SL_2})$
is of the form $(\det_{SL_2,\on{St}})^{a}$ for some element $a\in A(-1)$. Its restriction to $\on{FactGe}_A(\Gr_{\BG_m})$ 
identifies with $(\det_{\BG_m,1}\otimes \det_{\BG_m,-1})^a$.

\medskip

The associated quadratic form 
$$q:\BZ\to A(-1)$$
takes value $a$ on the generator $1\in \BZ$. Let $n:=\on{ord}(a)$. 

\medskip

We need to show that the pullback of $(\det_{\BG_m,1}\otimes \det_{\BG_m,-1})^a$ under the isogeny
\begin{equation} \label{e:isogeny}
\BG_m\overset{x\mapsto x^n}\longrightarrow \BG_m
\end{equation} 
is canonically trivial as a factorization gerbe on $\Gr_{\BG_m}$. For this, it suffices to show that the pullback of
the factorization line bundle $\det_{\BG_m,1}\otimes \det_{\BG_m,-1}$ under the above isogeny admits a canonical 
$n$-th root.

\medskip

We note that the pullback of $\det_{\BG_m,1}\otimes \det_{\BG_m,-1}$ under \eqref{e:isogeny} identifies with
$\det_{\BG_m,n}\otimes \det_{\BG_m,-n}$. However, a local calculation shows that we have a canonical isomorphism 
$$\det_{\BG_m,n}\otimes \det_{\BG_m,-n} \simeq (\det_{\BG_m,1}\otimes \det_{\BG_m,-1})^{\otimes n^2},$$
which implies the desired assertion. 

%

\begin{rem}
One can see explicitly what $\CG^{T^\sharp}$ is in terms of $\CG^G$ by interpreting both as 
$\Theta$-data, see \secref{sss:theta G new 2}.
In terms of {\it loc.cit.}, the $\Theta$-data corresponding to $\CG^{T^\sharp}$ is obtained from one for $\CG^G$
by restriction along 
$$\Lambda^\sharp\to \Lambda.$$

\end{rem}

\ssec{The metaplectic Langlands dual datum as a triple}

In this subsection we take $A:=E^{\times,\on{tors}}$.

\sssec{}

By \secref{sss:splitting}, to $\CG^{\pi_{1,\on{alg}}(G^\sharp)\otimes \BG_m}$ we can canonically attach an object
$$(\CG^{\pi_{1,\on{alg}}\otimes \BG_m})^0\in 
\on{FactGe}^0_A(\Gr_{\pi_{1,\on{alg}}(G^\sharp)\otimes \BG_m})\simeq \Maps(X,B^2_{\on{et}}(\Hom(\pi_{1,\on{alg}}(G^\sharp),E^{\times,\on{tors}})))$$
and a map
$$\epsilon:\pi_{1,\on{alg}}(G^\sharp)\to \BZ/2\BZ.$$

\sssec{}

Let $H$ denote the Langlands dual of $G^\sharp$, viewed as an algebraic group over $E$. Note that 
$$\Hom(\pi_{1,\on{alg}}(G^\sharp),E^\times)$$
identifies with $Z_{H}(E)$, where $Z_{H}$ denotes the center of $H$. 

\medskip

Hence, we can think of $(\CG^{\pi_{1,\on{alg}}\otimes \BG_m})^0$ as a 
$Z_{H}$-gerbe on $X$, to be denoted $\CG_Z$. Furthermore, we interpret the above map $\epsilon$ as a homomorphism 
\begin{equation} \label{e:epsilon to}
\epsilon:\BZ/2\BZ\to Z_{H}(E).
\end{equation} 

\sssec{} \label{sss:dual triple}

We will refer to the triple
\begin{equation} \label{e:MLLD}
(H,\CG_Z,\epsilon)
\end{equation}
as the \emph{metaplectic Langlands dual datum} corresponding to $\CG^G$. 

\sssec{Example}

Suppose that $\CG^G$ is trivial. Then $T^\sharp=T$ and $G^\sharp=G$, so $H=\cG$. In this case 
$\CG^{\pi_{1,\on{alg}}(G^\sharp)\otimes \BG_m}$ and $\epsilon$ are trivial.

\sssec{Example} \label{sss:det example}

Take now $\CG^G=(\det_{\fg})^{\frac{1}{2}}$. In this case we will have $T^\sharp=T$ and $G^\sharp=G$,
so $H=\cG$.

\medskip

However, the element $\epsilon$ equals now the image of $-1\in \BZ/2\BZ$ under the homomorphism
\begin{equation} \label{e:rho -1}
\BZ/2\BZ \to Z_{\cG}
\end{equation}
given by
$$\BZ/2\BZ \to \BG_m\overset{2\check\rho}\to \cT\to \cG.$$

Further, by Remark \ref{r:spin}, the $Z_H$-gerbe $\CG_Z$ is induced by means of \eqref{e:rho -1}
from the $\BZ/2\BZ$-gerbe $\omega_X^{\frac{1}{2}}$.





\section{Factorization gerbes on loop groups}  \label{s:loop}

In this section we will perform a crucial geometric construction that will explain why our definition of 
geometric metaplectic datum was ``the right thing to do": 

\medskip

We will show that a factorization gerbe on $\Gr_G$ give rise to a (factorization) gerbe on (the factorization version of) the loop 
group of $G$. 

\ssec{Digression: factorization loop and arc spaces}


Up until this point, the geometric objects that have appeared in this paper were all locally of finite type, considered 
as prestacks. However, the objects that we will introduce below \emph{do not} have this property. 

\sssec{}

For an affine test scheme $S$ and an $S$-point of $\Ran$, given by a finite set $I\subset \Maps(S,X)$, let $\hat\cD_I$ be the corresponding
relative formal disc:

\medskip

By definition, $\hat\cD_I$ is the formal scheme equal to the completion of $S\times X$ along the union $\Gamma_I$ of the graphs of the maps $S\to X$
corresponding to the elements of $I$. 

\medskip

Note that for a finite set $J$ and a point
$$\{I_j,j\in J\}\in  \Ran^J_{\on{disj}},$$
we have
\begin{equation} \label{e:disc factor}
\hat\cD_I\simeq \underset{j}\sqcup\, \hat\cD_{I_j},
\end{equation} 
where $I=\underset{j}\sqcup\, I_j$.

\sssec{}

Since $S$ was assumed affine, $\hat\cD_I$ is an ind-object in the category $\Sch^{\on{aff}}$. 
Let $\cD_I$ be the affine scheme corresponding to the formal scheme $\hat\cD_I$, i.e., the image of $\hat\cD_I$ under the functor
$$\on{colim}: \on{Ind}(\Sch^{\on{aff}}) \to \Sch^{\on{aff}}.$$

In other words, if 
$$\hat\cD_I\simeq \underset{\alpha}{\on{colim}}\, Z_\alpha,$$
where $Z_\alpha=\Spec(A_\alpha)$ and the colimit is taken in $\on{PreStk}$, then $\cD_I=\Spec(A)$, where
$$A=\underset{\alpha}{\on{lim}}\, A_\alpha.$$

\medskip

Let $\ocD_I$ be the open subscheme of $\cD_I$, obtained by removing the closed subscheme $\Gamma_I$ equal to the union of the graphs
of the maps $S\to X$ corresponding to the elements of $I$. 

\sssec{}  \label{sss:arc space}

Let $Z$ be a prestack. We define the prestacks $\fL^+(Z)$ (resp., $\fL(Z)$) over $\Ran$ as follows. 

\medskip

For an affine test scheme $S$ and an $S$-point of $\Ran$, given by a finite set $I\subset \Maps(S,X)$,
its lift to an $S$-point of $\fL^+(Z)$ (resp., $\fL(Z)$) is the datum of a map
$\cD_I\to Z$ (resp., $\ocD_I\to Z$). 

\medskip

The isomorphisms \eqref{e:disc factor} imply that $\fL^+(Z)$ and $\fL(Z)$ are naturally factorization prestacks 
over $\Ran$.

\sssec{}

Assume for a moment that $Z$ is an affine scheme. 
Note that in this case the definition of $\fL^+(Z)$, the datum of a map $\cD_I\to Z$ 
is equivalent to that of a map of prestacks $\hat\cD_I\to Z$. 

\medskip

Assume now that $Z$ is a smooth scheme of finite type (but not necessarily affine). Then one shows that 
for every $S\to \Ran$, the fiber product
$$S\underset{\Ran}\times \fL^+(Z)$$ is a projective
limit (under smooth maps) of smooth affine schemes over $S$.

\ssec{Factorization loop and arc groups}

\sssec{}  

Let us recall that the Beauville-Laszlo Theorem says that the definition of $\Gr_G$ can be rewritten in terms
of the pair $\ocD_I\subset \cD_I$. 

\medskip

Namely, given $I$ as above, the datum of its lift to a point of $\Gr_G$ is a pair $(\CP_G,\alpha)$, where 
$\CP_G$ is a $G$-bundle on $\cD_I$, and $\alpha$ is the trivialization of $\CP_G|_{\ocD_I}$.
(Note that restriction along $\hat\cD_I\to \cD_I$ induces an equivalence between the category of $G$-bundles
on $\cD_I$ and that on $\hat\cD_I$.) 

\medskip

In other words, the Beauville-Laszlo says that restriction along
$$(\ocD_I\subset \cD_I)  \to (U_I\subset S\times X)$$
induces a bijection on the corresponding pairs $(\CP_G,\alpha)$. In the above formula, the notation $U_I$ 
is as in \secref{sss:U}.

\sssec{}  \label{sss:BL}

This interpretation of $\Gr_G$ shows that the group-prestack $\fL(G)$ acts naturally on $\Gr_G$,
with the stabilizer of the unit section being $\fL^+(G)$. Furthermore, the natural map
\begin{equation} \label{e:Gr as quotient}
\fL(G)/\fL^+(G)\to \Gr_G,
\end{equation} 
is an isomorphism, where the quotient is understood in the sense of stacks in the \'etale topology. 

\medskip 

The isomorphism \eqref{e:Gr as quotient} implies that for every $S\to \Ran$, the fiber product
$$S\underset{\Ran}\times \fL(G),$$ 
is an ind-scheme over $S$. 

\sssec{}

Recall that given a group-prestack $\CH$ over a base $Z$, we can talk about a gerbe over $\CH$
being \emph{multiplicative}, i.e., compatible with the group-structure.

\medskip

In particular, we can consider the spaces 
$$\on{FactGe}^{\on{mult}}_A(\fL(G)) \text{ and } \on{FactGe}^{\on{mult}}_A(\fL^+(G))$$
of multiplicative factorization gerbes on $\fL(G)$ and $\fL^+(G)$, respectively. 

\sssec{}

The isomorphism \eqref{e:Gr as quotient} defines a map
\begin{equation} \label{e:gerbe on loop group}
\on{FactGe}^{\on{mult}}_A(\fL(G)) 
\underset{\on{FactGe}^{\on{mult}}_A(\fL^+(G))}\times *\to 
\on{FactGe}_A(\Gr_G).
\end{equation} 

\medskip

We have the following result\footnote{This result was claimed in \cite[Theorem III.2.10]{Re}, but the proof was incomplete;
specifically the argument in Proposition III.2.8 of {\it loc.cit.} is incomplete.}:

\begin{prop}  \label{p:gerbe on loop group}
The map \eqref{e:gerbe on loop group} is an isomorphism.
\end{prop}

We will sketch the proof of this proposition in \secref{ss:constr equiv}. 
It consists of explicitly constructing the inverse map. 


\sssec{}

Let us restate \propref{p:gerbe on loop group} in words. It says that, given a factorization gerbe on $\Gr_G$, its pullback 
under the projection
$$\fL(G)\to \Gr_G,$$
carries a uniquely defined multiplicative structure that is compatible with that of factorization and the trivialization of the further
restriction of our gerbe to $\fL^+(G)$. 

\ssec{Sketch of proof of \propref{p:gerbe on loop group}} \label{ss:constr equiv}

The proof we will sketch was suggested to us by Y.~Zhao. We will produce an inverse map to \eqref{e:gerbe on loop group} by appealing to
\propref{p:parameterization}. 

\medskip

Namely, using  a variant of the construction described in Sects. \ref{sss:construct gerbe prel}-\ref{sss:construct gerbe end},
we will show that a based map
\begin{equation} \label{e:based map again}
B_{\on{et}}(G) \times X\to B^4_{\on{et}}(A(1)),
\end{equation}
gives rise to an object of $\on{FactGe}^{\on{mult}}_A(\fL(G)) \underset{\on{FactGe}^{\on{mult}}_A(\fL^+(G))}\times *$.

\sssec{}

Let 
$$\on{Hecke}^{\on{loc}}_G$$ denote the version of the local Hecke stack over $\Ran$.

\medskip

Namely, an $S$-points of $\on{Hecke}^{\on{loc}}_G$ is a triple 
$$(I,\CP'_G,\CP''_G,\alpha),$$
where:

\smallskip

\noindent--$I\subset \Hom(S,X)$ is a $S$-point of $\Ran$;

\smallskip

\noindent--$\CP'_G$ and $\CP''_G$ are $G$-bundles on $\cD_I$;

\smallskip

\noindent--$\alpha$ is an isomorphism between $\CP'_G$ and $\CP''_G$ over $\ocD_I$.

\medskip

We can identify
$$\on{Hecke}^{\on{loc}}_G\simeq \fL^+(G)\backslash \Gr_G\simeq \fL^+(G)\backslash \fL(G)/\fL^+(G).$$

\medskip

Consider also the prestack 
$$\fL^+(G)\backslash \Ran,$$
whose $S$-points are pairs $(I,\CP_G)$, where $I$ is as above and $\CP_G$ is a $G$-bundle on $\cD_I$.

\medskip

We have the two projections
$$\on{Hecke}^{\on{loc}}_G\rightrightarrows \fL^+(G)\backslash \Ran$$
that remember the data of $\CP'_G$ and $\CP''_G$, respectively.

\medskip

Furthermore, composition of isomorphisms defines on $\on{Hecke}^{\on{loc}}_G$ a natural structure of groupoid
acting on $\fL^+(G)\backslash \Ran$. This structure is compatible with factorization in a natural way. 

\medskip

Hence, it makes sense to talk about factorization gerbes on $\on{Hecke}^{\on{loc}}_G$, equipped with a
\emph{multiplicative structure} with respect to the groupoid operation.

\medskip

It is clear that the space $\on{FactGe}^{\on{mult}}_A(\fL(G)) \underset{\on{FactGe}^{\on{mult}}_A(\fL^+(G))}\times *$
is canonically equivalent to the space of such gerbes. 

\medskip

We will now show how to associate to a based map \eqref{e:based map again} such a gerbe on $\on{Hecke}^{\on{loc}}_G$.

\sssec{} 

Consider the diagram
$$
\CD
\Gamma_I @>{\hat\iota}>> \cD_I @<{\wh\jmath}<< \ocD_I\\
@V{\pi}VV \\
S.
\endCD
$$

Given a based map \eqref{e:based map again} and the data of $(I,\CP'_G,\CP''_G)$, we obtain two maps
\begin{equation} \label{e:two maps from disc}
\cD_I\rightrightarrows B^4_{\on{et}}(A(1)).
\end{equation} 

The datum of $\alpha$ defines a trivialization of the difference between these two maps to $\ocD_I$. Hence, this difference
can be viewed as a 4-cochain in 
$$\on{C}^\bullet_{\on{et}}(\Gamma_I,\hat\iota{}^!(A_{\cD_I}(1))).$$

To this data we need to associate a 2-cochain in
$$\on{C}^\bullet_{\on{et}}(S,A_S).$$

\sssec{} \label{sss:construct gerbe local}

We have a commutative diagram
$$
\CD
\Gamma_I @>{\hat\iota}>> \cD_I @<{\wh\jmath}<< \ocD_I \\
@V{=}VV @VVV @VVV \\
\Gamma_I @>{\iota}>> S\times X @<{\jmath}<< U_I
\endCD
$$

Not it follows from the Fujiwara-Gabber comparison theorem\footnote{We are grateful to Ofer Gabber for explaining this to us.},
(see \cite[Corollary 6.6.4]{Fu}), that the natural map
$$\hat\iota{}^!(A_{\cD_I}(1))\to \iota^!(A_{S\times X}(1))$$
is an isomorphism. 

\medskip

Hence, from \eqref{e:upper !}, we obtain an isomorphism
$$\hat\iota{}^!(A_{\cD_I}(1))\simeq \pi^!(A_S)[-2].$$

The rest of the construction proceeds as in Sects. \ref{sss:construct gerbe inter}-\ref{sss:construct gerbe end}. 
The multiplicative structure on the resulting gerbe on $\on{Hecke}^{\on{loc}}_G$ follows from the construction.

\sssec{An alternative construction}

We will now sketch a construction of the map from
$$\Maps_{\on{Ptd}(\on{PreStk})}(B_{\on{et}}(G) \times X,B^4_{\on{et}}(A(1)))$$ to the space
of multiplicative factorization gerbes on $\on{Hecke}^{\on{loc}}_G$
that avoids the use of the Gabber-Fujiwara theorem (see \secref{sss:construct gerbe local}). Its main geometric ingredient is of independent interest.

\medskip

Let $\Bun_G^{\on{loc,Zar}}$ be the prestack that attaches to a test affine scheme $S$ the groupoid of pairs $(I,\CP_G)$,
where $I\subset \Hom(S,X)$ is an $S$-point of $\Ran$, and $\CP_G$ is a $G$-bundle defined on \emph{some} Zariski-open
subset 
$$\cD_I^{\on{Zar}}\subset S\times X$$
that contains $\Gamma_I$, and such that $\CP_G$ can be trivialiazed \'etale-locally on $S$. Morphisms in this groupoid are
isomorphisms of $G$-bundles defined over the intersection of their domains of definition. 

\medskip

Let $\on{Hecke}^{\on{loc,Zar}}_G$ be the following version of $\on{Hecke}^{\on{loc}}_G$. Its $S$-points are quadruples $(I,\CP'_G,\CP''_G,\alpha)$,
where $I$ is as above, $\CP'_G$ and $\CP''_G$ are $G$-bundles defined on some $\cD_I^{\on{Zar}}$ as above, and $\alpha$ is an isomorphism 
between $\CP'_G$ and $\CP''_G$ defined over 
$$\ocD{}_I^{\on{Zar}}:=\cD_I^{\on{Zar}}-\Gamma_I.$$

\medskip

The prestack $\on{Hecke}^{\on{loc,Zar}}_G$ has a natural structure of groupoid acting on $\Bun_G^{\on{loc,Zar}}$. So we can talk about
mutiplicative factorization gerbes on $\on{Hecke}^{\on{loc,Zar}}_G$. The construction in Sects. \ref{sss:construct gerbe inter}-\ref{sss:construct gerbe end}
defines a map
$$\Maps_{\on{Ptd}(\on{PreStk})}(B_{\on{et}}(G) \times X,B^4_{\on{et}}(A(1)))\to \on{FactGe}^{\on{mult}}_A(\on{Hecke}^{\on{loc,Zar}}_G).$$

Now, we claim:  

\begin{prop}
The forgetful map $\on{Hecke}^{\on{loc,Zar}}_G\to \on{Hecke}^{\on{loc}}_G$ defines equivalences
$$\on{FactGe}_A(\on{Hecke}^{\on{loc}}_G)\to \on{FactGe}_A(\on{Hecke}^{\on{loc,Zar}}_G)$$
and
$$\on{FactGe}^{\on{mult}}_A(\on{Hecke}^{\on{loc}}_G)\to \on{FactGe}^{\on{mult}}_A(\on{Hecke}^{\on{loc,Zar}}_G)$$
\end{prop}

\begin{proof}

Note that the map $\on{Hecke}^{\on{loc,Zar}}_G\to \on{Hecke}^{\on{loc}}_G$ fits into a Cartesian square
$$
\CD
\on{Hecke}^{\on{loc,Zar}}_G @>>>  \on{Hecke}^{\on{loc}}_G \\
@VVV @VVV \\
\Bun_G^{\on{loc,Zar}} @>>> \fL^+(G)\backslash \Ran,
\endCD
$$
where the vertical maps can be taken to be \emph{either} of the projections. The same is true for the other terms of the \v{C}ech nerve
of $\on{Hecke}^{\on{loc}}_G$ (resp., $\on{Hecke}^{\on{loc,Zar}}_G$) over $\fL^+(G)\backslash \Ran$ (resp., $\Bun_G^{\on{loc,Zar}}$). 

\medskip

Hence, to prove the proposition, it suffices to show that the map
$$\Bun_G^{\on{loc,Zar}} \to \fL^+(G)\backslash \Ran$$
is \emph{universally homologically contractible} in the sense of \cite[Sect. 2.5]{Ga7}. 

\medskip

Fix a map $S\to \Ran$ corresponding to a finite non-empty set $I\subset \Hom(S,X)$. It suffices to show that the map
$$S\underset{\Ran}\times \Bun_G^{\on{loc,Zar}} \to S\underset{\Ran}\times (\fL^+(G)\backslash \Ran)$$
is universally homologically contractible. Let us view $\Gamma_I$ as subscheme in $S\times X$, which is finite and flat over $S$, 
and let $G_I$ be the group-scheme over $S$ that classifies maps from $\Gamma_I$ to $G$. Evaluation defines a map of group-schemes over $S$
$$S\underset{\Ran}\times \fL^+(G)\to G_I$$
Denote by $\CK_I$ its kernel.

\medskip

It suffices to show that the map
$$(S\underset{\Ran}\times \Bun_G^{\on{loc,Zar}})\underset{G_I\backslash S}\times S \to
(S\underset{\Ran}\times (\fL^+(G)\backslash \Ran)) \underset{G_I\backslash S}\times S\simeq \CK_I\backslash S$$
is universally homologically contractible. 

\medskip

Since $\CK_I$ is pro-unipotent, it suffices to show that the projection
$$(S\underset{\Ran}\times \Bun_G^{\on{loc,Zar}})\underset{G_I\backslash S}\times S \to S$$
is universally homologically contractible. 

\medskip

Note that $(S\underset{\Ran}\times \Bun_G^{\on{loc,Zar}})\underset{G_I\backslash S}\times S$ is the \'etale sheafification 
of the prestack $B(\CK_I^{\on{Zar}})$, where $\CK_I^{\on{Zar}}$ is the group-prestack over $S$ that attaches to $S'\to S$
the group of maps
$$g:\cD_I^{\on{Zar}}\to G,\quad S'\underset{S}\times \Gamma_I \subset \cD_I^{\on{Zar}}\subset S'\times X,$$
such that 
$$g|_{S'\underset{S}\times \Gamma_I}=1.$$

Hence, it suffices to show that 
$$\CK_I^{\on{Zar}}\to S$$ is universally homologically contractible. 
However, this is the assertion of (the easy case of) \cite[Theorem 3.3.2]{GL2}, see Proposition 3.5.3 in {\it loc. cit}. 

\end{proof} 

\ssec{What does \propref{p:gerbe on loop group} say in concrete terms?} \label{ss:equiv concrete}

In this subsection we will give a ``hands-on" explanation of the concrete meaning of \propref{p:gerbe on loop group}.

\sssec{}

Denote
$$\Gr_{G,X}:=X\underset{\Ran}\times \Gr_G,\,\, \Gr_{G,X^2}:=X^2\underset{\Ran}\times \Gr_G,$$
$$\on{Hecke}^{\on{loc}}_{G,X}:=X\underset{\Ran}\times  \on{Hecke}^{\on{loc}}_G,\,\, \on{Hecke}^{\on{loc}}_{G,X^2}:=X^2\underset{\Ran}\times  \on{Hecke}^{\on{loc}}_G,$$
$$\fL^+(G)_X:=X\underset{\Ran}\times \fL^+(G),\,\, \fL^+(G)_{X^2}:=X^2\underset{\Ran}\times \fL^+(G),$$
etc. 

\medskip

Consider the fiber product
\begin{equation} \label{e:pre-convolution}
\Gr_{G,X^2}\underset{\fL^+(G)_{X^2} \backslash X^2}\times \on{Hecke}^{\on{loc}}_{G,X^2}.
\end{equation}

It classifies the data of $(x_1,x_2,\CP^1_G,\CP^2_G,\alpha_1,\alpha_2)$, where

\smallskip

\noindent--$x_1,x_2$ are $S$-points of X; 

\smallskip

\noindent--$\CP^1_G$ and $\CP^2_G$ are $G$-bundles on $\cD_{x_1\cup x_2}$;

\smallskip

\noindent--$\alpha_1$ is an isomorphism between $\CP^0_G$ (the trivial bundle) and $\CP^1_G$ over $\ocD_{x_1\cup x_2}$;

\smallskip

\noindent--$\alpha_2$ is an isomorphism between $\CP^1_G$ (the trivial bundle) and $\CP^2_G$ over $\ocD_{x_1\cup x_2}$. 

\medskip

Let 
$$\wt\Gr_{G,X^2}\subset \Gr_{G,X^2}\underset{\fL^+(G)_{X^2} \backslash X^2}\times \on{Hecke}^{\on{loc}}_{G,X^2}$$
be the ``convolution diagram". I.e., this is a closed subfunctor of \eqref{e:pre-convolution} defined by the conditions that:

\medskip

\noindent--$\alpha_1$ extends to an isomorphism over 
$$(\cD_{x_1\cup x_2}-\Gamma_{x_1})\supset \ocD_{x_1\cup x_2}=(\cD_{x_1\cup x_2}-\Gamma_{x_1\cup x_2}).$$

\medskip

\noindent--$\alpha_2$ extends to an isomorphism over 
$$(\cD_{x_1\cup x_2}-\Gamma_{x_2})\supset \ocD_{x_1\cup x_2}=(\cD_{x_1\cup x_2}-\Gamma_{x_1\cup x_2}).$$

\medskip

The groupoid structure on $\on{Hecke}^{\on{loc}}_{G,X^2}$ gives rise a map
\begin{equation} \label{e:conv map}
\on{conv}: \wt\Gr_{G,X^2}\to \Gr_{G,X^2},
\end{equation} 
defined by sending
$$(x_1,x_2,\CP^1_G,\CP^2_G,\alpha_1,\alpha_2) \mapsto (x_1,x_2,\CP^2_G,\alpha_2\circ \alpha_1).$$

\sssec{}

We also have a projection
$$\on{pr}: \wt\Gr_{G,X^2} \to \Gr_{G,X}\times X,$$
defined by sending 
$$(x_1,x_2,\CP^1_G,\CP^2_G,\alpha_1,\alpha_2) \mapsto (x_1,\CP^1_G,\alpha_1,x_2).$$

This projection allows to view $\wt\Gr_{G,X^2}$ as a \emph{twisted product} 
\begin{equation} \label{e:tw prod gr}
(\Gr_{G,X}\times X) \underset{X^2}{\widetilde\times} (X\times \Gr_{G,X}),
\end{equation} 
by which we mean that it is associated to a $\fL^+(G)_{X^2}$-torsor over $\Gr_{G,X}\times X$ 
(thought of as the first factor in \eqref{e:tw prod gr}) and the ind-scheme $X\times \Gr_{G,X}$
(thought of as the second factor in \eqref{e:tw prod gr}), equipped with an action of $\fL^+(G)_{X^2}$. 

\medskip

The above torsor and action are obtained as follows. The $\fL^+(G)_{X^2}$-torsor over $\Gr_{G,X}\times X$ corresponds to the map
$$\Gr_{G,X}\times X \hookrightarrow \Gr_{G,X^2}\to (\fL^+(G)_{X^2} \backslash X^2),$$
and the $\fL^+(G)_{X^2}$-action on $X\times \Gr_{G,X}$ is induced by the $\fL^+(G)_X$-action on $\Gr_{G,X}$ via the projection 
$$\fL^+(G)_{X_2} \overset{p_2}\to X\times \fL^+(G)_X,$$
obtained by restriction along $\cD_{x_1\cup x_2} \hookleftarrow \cD_{x_2}$. 

\medskip

Note, in particular, that the induced $(X\times \fL^+(G)_X)$-torsor on $\Gr_{G,X}\times X$ is canonocally trivialized over
$$X^2-\Delta\subset X^2.$$

In particular, we have a canonical identification
\begin{equation} \label{e:conv away}
\wt\Gr_{G,X^2}|_{X^2-\Delta} \simeq (\Gr_{G,X}\times \Gr_{G,X})|_{X^2-\Delta}. 
\end{equation} 

Under this identification, the projection $\on{conv}$ of \eqref{e:conv map} identifies with the factorization isomorphism 
$$(\Gr_{G,X}\times \Gr_{G,X})|_{X^2-\Delta}\simeq \Gr_{G,X^2}|_{X^2-\Delta}$$
for $\Gr$. 

\sssec{}

Now, (part of) the assertion of \propref{p:gerbe on loop group} can be formulated as follows. Let $\CG$ be an object of 
$\on{FactGe}_A(\Gr_G)$. Let $\CG_X$ (resp., $\CG_{X^2}$) be its restriction to $\Gr_{G,X}$ (resp., $\Gr_{G,X^2}$). 

\medskip

Then the claim is that $\CG_X$ \emph{admits a unique structure of equivariance} with resect to $\fL^+(G)_X$ such that
the following holds:

\medskip

On the one hand, having a structure of $\fL^+(G)_X$-equivariance on $\CG_X$, we can form the twisted product
$$\CG_X\wt\boxtimes \CG_X\in \on{Ge}_A(\wt\Gr_{G,X^2}),$$
see formula \eqref{e:tw prod gr}. 

\medskip

On the other hand, we can consider the object
$$\on{conv}^*(\CG_{X^2})\in \on{Ge}_A(\wt\Gr_{G,X^2}),$$
see formula \eqref{e:conv map}. 

\medskip

Now, the factorization structure on $\CG$ and the identification \eqref{e:conv away} imply that we have a canonical isomorphism
\begin{equation} \label{e:conv away gerbes}
\CG_X\wt\boxtimes \CG_X|_{X^2-\Delta}\simeq \on{conv}^*(\CG_{X^2})|_{X^2-\Delta}.
\end{equation} 

The requirement on the $\fL^+(G)_X$-equivariance on $\CG_X$ is that the identification \eqref{e:conv away gerbes} extends to
an identification
\begin{equation} \label{e:conv away gerbes extends}
\CG_X\wt\boxtimes \CG_X\simeq \on{conv}^*(\CG_{X^2})
\end{equation} 
over all of $X^2$. 

\ssec{The case of tori}  \label{ss:another interp}

We will now explain how the contents of \secref{ss:equiv concrete} play out in the case when $G=T$ is a torus. 

\sssec{}

Up to nilpotents, we can identify 
$$\Gr_{T,X}\simeq \Lambda \times X,$$
and the action of $\fL^+(T)_X$ on it is trivial. 

\medskip

Hence, the a datum of $\fL^+(T)_X$-equivariance on $\CG_X$ amounts to a map from $\Lambda$ to the set of
multiplicative $A$-torsors on $\fL^+(T)_X$. 

\medskip

Since elements of $A$ have orders prime to $\on{char}(k)$, such torsors are pulled back via the evaluation map
\begin{equation} \label{e:ev T}
\fL^+(T)_X \to X\times T
\end{equation} 
from multiplicative $A$-torsors on $T$. The latter are, by Kummer theory, in bijection with maps $\Lambda\to A(-1)$.

\medskip

We will show that the resulting map
$$\Lambda \to \Hom_{\on{Ab}}(\Lambda,A(-1)),$$
which can be viewed as a map 
$$\wt{b}:\Lambda \times \Lambda \to A(-1),$$
is given by $b(-,-)$, the bilinear form corresponding to $\CG$. 

\sssec{}

Fix a connected component of $\Gr_{G,X}$ corresponding to $\lambda\in \Lambda$, and the corresponding connected component 
in $\Gr_{G,X}\times X$, thought if as the base of the fibration in \eqref{e:tw prod gr}. Consider the induced $X\times \fL^+(T)_X$-torsor,
and further the $X\times T$-torsor, induced by \eqref{e:ev T}. 

\medskip

This torsor is trivialized over $X^2-\Delta$, and over all of $X^2$ it identifies with $\lambda\cdot \CO(\Delta)$. 

\medskip

Consider a connected component
$$\wt\Gr^{\lambda,\mu}_{T,X^2}:=(\Gr^\lambda_{T,X}\times X) \underset{X^2}{\widetilde\times} (X\times \Gr^\mu_{T,X})$$
of $\wt\Gr_{T,X^2}$, corresponding $\lambda,\mu\in \Lambda$. 

\medskip

Since the action of $X\times \fL^+(T)_X$ on $X\times \Gr^\mu_{T,X}$ is trivial, we obtain an identification
\begin{equation} \label{e:conv T}
\wt\Gr^{\lambda,\mu}_{T,X^2}\simeq \Gr^\lambda_{T,X}\times \Gr^\mu_{T,X}\simeq X\times X.
\end{equation} 

We obtain that with respect to this identification, the gerbe 
$$\CG^\lambda\wt\boxtimes \CG^\mu:=\CG_X\wt\boxtimes \CG_X|_{\wt\Gr^{\lambda,\mu}_{T,X^2}}$$
identifies with
$$\CO(\Delta)^{\wt{b}(\lambda,\mu)}.$$

\sssec{}

Consider the restriction of the map $\on{conv}$ to the connected component $\wt\Gr^{\lambda,\mu}_{T,X^2}$. This map is an isomorphism onto an 
irreducible component $\Gr^{\lambda,\mu}_{T,X^2}$ of $\Gr_{T,X^2}$. 

\medskip

Furthermore, the factorization isomorphism 
$$\Gr^{\lambda,\mu}_{T,X^2}|_{X^2-\Delta}\simeq (\Gr^\lambda_{T,X}\times \Gr^\mu_{T,X})|_{X^2-\Delta}$$
also extends to an isomorphism over all of $X^2$. 

\medskip

The isomorphism
$$\CG^\lambda\wt\boxtimes \CG^\mu\simeq \CG^{\lambda,\mu}$$
of \eqref{e:conv away gerbes extends} extends the tautological identification of the two sides with
$$\CG^\lambda\boxtimes \CG^\mu$$
over $X$. Hence, by \eqref{e:form as pole}, we obtain
$$\wt{b}(\lambda,\mu)=b(\lambda,\mu),$$
as claimed.

\section{Factorization category of representations}  \label{s:tw rep}

From now on, until the end of the paper we will assume that $k=E$ and we will work in the context of D-modules. 

\ssec{Digression: factorization categories arising from symmetric monoidal categories}  \label{ss:fact sym mon} 

In this subsection we will explain a procedure that produces a factorization sheaf of categories from 
a sheaf symmetric monoidal categories on $X$. The source of the metaplectic geometric Satake functor will be
a factorization sheaf of categories obtained in this way. 

\medskip

For a more detailed discussion see \cite[Sect. 6]{Ras2}.

\sssec{}  \label{sss:fact sym mon}

Let $\CC$ be a sheaf of symmetric monoidal categories over $X$. To it we will associate 
a sheaf of symmetric monoidal categories over $\Ran$, equipped with a factorization structure, denoted $\on{Fact}(\CC)$.

\medskip

We will construct  $\on{Fact}(\CC)$ as a family of sheaves of symmetric monoidal categories over $X^I$ for all finite non-empty
sets $I$, compatible under surjections $I_1\twoheadrightarrow I_2$. 
We will use \thmref{t:1-aff} that says that the datum of sheaf of categories over $X^I$ is equivalent to that of
a category acted on by $\Shv(X^I)$. So, we will produce system of symmetric monoidal categories $\on{Fact}(\CC)(X^I)$,
compatible under 
\begin{equation} \label{e:I1 and I2}
\on{Fact}(\CC)(X^{I_2})\simeq \Shv(X^{I_2})\underset{\Shv(X^{I_1})}\otimes \on{Fact}(\CC)(X^{I_1}).
\end{equation} 

\sssec{}

Let $\CC(X)$ denote the category of sections of $\CC$ over $X$; this is a symmetric monoidal category over $\Shv(X)$. 
For a finite set $J$ we let $\CC^{\otimes J}(X)$ the $J$-fold tensor product of $\CC(X)$ \emph{over} $\Shv(X)$.

\medskip

Note that for a surjection of finite sets $I\twoheadrightarrow J$ we have a canonical isomorphism
\begin{equation} \label{e:finite set decomp}
\CC^{\otimes I}(X)\simeq \left(\underset{j\in J}\bigotimes\, \CC^{\otimes I_j}(X)\right)\underset{\Shv(X^J)}\otimes \Shv(X),
\end{equation}  
where $I_j$ denotes the preimage of $j\in J$ under $I\to J$. 

\medskip

In addition, for $I\twoheadrightarrow J$, the symmetric monoidal structure on $\CC(X)$ gives rise to the functors
\begin{equation} \label{e:sym mon mult}
\CC^{\otimes I}(X)\to \CC^{\otimes J}(X). 
\end{equation} 

\begin{rem} \label{r:constant sym mon}

We will be particularly interested in the case when $\CC$ is constant, i.e., $$\CC(X)\simeq \CC_{\on{pt}}\otimes \Shv(X)$$
for a symmetric monoidal category $\CC_{\on{pt}}$. Note that in this case $\CC^{\otimes J}(X)$ is just $\CC_{\on{pt}}^{\otimes J}\otimes \Shv(X)$.  

\end{rem}

\sssec{}

For a given $I$,  let $\on{Tw}(I)$ be the category whose objects are pairs
\begin{equation} \label{e:Tw}
I\twoheadrightarrow J\twoheadrightarrow K
\end{equation}
(here $J$ and $K$ are sets (automatically, finite and non-empty)),
and where morphisms from $(J,K)$ to $(J',K')$ are commutative diagrams
\begin{equation} \label{e:map Tw}
\CD
I  @>>>  J @>>> K \\
@V{\on{id}}VV   @VVV   @AAA  \\   
I  @>>>  J' @>>> K'.
\endCD
\end{equation} 
(Note that the arrows between the $K$'s go in the opposite direction.)

\sssec{}

Consider the functor
\begin{equation} \label{e:Tw functor}
\on{Tw}(I)\to \on{DGCat}
\end{equation}
that sends an object \eqref{e:Tw} to
$$\underset{k\in K}\bigotimes\, \CC^{\otimes J_k}(X),$$
where $J_k$ is the preimage under $J\to K$ of the element $k\in K$. The above tensor product is naturally a symmetric
monoidal category over $\Shv(X^K)$.

\medskip

For a morphism \eqref{e:map Tw} in $\on{Tw}(I)$, we let the corresponding functor 
$$\underset{k\in K}\bigotimes\, \CC^{\otimes J_k}(X) \to \underset{k'\in K'}\bigotimes\, \CC^{\otimes J'_{k'}}(X)$$
be given by the composition
\begin{multline*} 
\underset{k\in K}\bigotimes\, \CC^{\otimes J_k}(X) \overset{\text{\eqref{e:sym mon mult}}}\longrightarrow \underset{k\in K}\bigotimes\, \CC^{\otimes J'_k}(X) 
\overset{\text{\eqref{e:finite set decomp}}} \simeq 
\underset{k\in K}\bigotimes\,  \left(\Bigl(\underset{k'\in K'_k}\bigotimes\, \CC^{\otimes J'_{k'}}(X)\Bigr)\underset{\Shv(X^{K'_k})}\otimes \Shv(X)\right) = \\
=\left(\underset{k'\in K'}\bigotimes\, \CC^{\otimes J'_{k'}}(X)\right)\underset{\Shv(X^{K'})}\otimes \Shv(X^K) \to 
\underset{k'\in K'}\bigotimes\, \CC^{\otimes J'_{k'}}(X),
\end{multline*}
where the the last arrow is given by the direct image functor along $X^K\to X^{K'}$. 

\sssec{}

We let $\on{Fact}(\CC)(X^I)$ on be the object of $\on{DGCat}$ equal to the colimit of the functor
\eqref{e:Tw functor} over $\on{Tw}(I)$. 

\medskip

The compatibilities \eqref{e:I1 and I2}, as well as the factorization structure on $\on{Fact}(\CC)$ 
follow from the construction. 

\sssec{}

Let $\on{Fact}(\CC)(\Ran)$ denote the category of global sections of $\on{Fact}(\CC)$ over $\Ran$. 

\medskip

As in \cite[Sect. 4.2]{Ga5}, the (symmetric) monoidal structure on $\on{Fact}(\CC)$ as a sheaf of categories
over $\Ran$ and
the operation of union of finite sets makes $\on{Fact}(\CC)(\Ran)$ into
a \emph{non-unital} (symmetric) monoidal category.

\ssec{Twisting procedures on the Ran space}  \label{ss:tw on Ran}

In this subsection we will start with a symmetric monoidal category $\CC$ and some twisting data,
and associate to it a sheaf of categories on the Ran space. 

\sssec{}

First, to $\CC$ we associate the constant sheaf of symmetric monoidal categories over $X$,
which, by a slight abuse of notation we denote by the same symbol $\CC$; we have  
$\CC(X)=\CC\otimes \Shv(X)$, see Remark \ref{r:constant sym mon}. 

\medskip

Consider the corresponding factorization sheaf $\on{Fact}(\CC)$ of
symmetric monoidal categories over $\Ran$. 

\sssec{}

Let now $A$ be a torsion abelian group that acts by automorphisms of the identity functor on $\CC$ (viewed as  
a symmetric monoidal category), and let $\CG_A$ be an $A$-gerbe on $X$.

\medskip

Using \secref{sss:twist sheaves of categ}, we can twist $\CC$ by $\CG_A$ and obtain a new
sheaf of symmetric monoidal categories over $X$, denoted $\CC_{\CG_A}$. 

\medskip

In particular,
we have the symmetric monoidal category $\CC_{\CG_A}(X)$ over $\Shv(X)$.

\sssec{}  \label{sss:tw sym mon gerbe} 

Applying to $\CC_{\CG_A}$ the construction from \secref{ss:fact sym mon}, we obtain a new sheaf
of symmetric monoidal categories over $\Ran$, denoted $\on{Fact}(\CC)_{\CG_A}$. 

\medskip

In particular, we obtain the non-unital symmetric monoidal category $\on{Fact}(\CC)_{\CG_A}(\Ran)$.

\medskip

Note that the value of $\on{Fact}(\CC)_{\CG_A}$ on $X$ under the canonical map $X\to \Ran$ is the
symmetric monoidal category $\CC_{\CG_A}(X)$

\sssec{} \label{sss:tw sym mon sign}

Let now $\epsilon$ be a 2-torsion element of $A$. Then we can further twist 
$\on{Fact}(\CC)_{\CG_A}$ to obtain a factorization sheaf of symmetric
monoidal DG categories, denoted $\on{Fact}(\CC)_{\CG_A}^\epsilon$.

\medskip

Namely, the element $\epsilon$ can be used to modify the braiding on $\CC$ and thereby
obtain a \emph{new} symmetric monoidal category, denoted $\CC^\epsilon$. We let
$$\on{Fact}(\CC)_{\CG_A}^\epsilon:=\on{Fact}(\CC^\epsilon)_{\CG_A}.$$

A key feature of the latter twist is that we have a canonical isomorphism
\begin{equation} \label{e:epsilon twist trivial}
\on{Fact}(\CC)_{\CG_A}\simeq \on{Fact}(\CC)_{\CG_A}^\epsilon,
\end{equation}
as sheaves of \emph{monoidal} categories over $\Ran$. But this identification is \emph{not} compatible
with either the symmetric monoidal nor factorization structure. 

\begin{rem}
At the level of underlying triangulated categories, the modification
$$\on{Fact}(\CC)_{\CG_A}\rightsquigarrow \on{Fact}(\CC)_{\CG_A}^\epsilon$$
can be described as follows\footnote{However, it may not be so straightforward to perform
this construction at the level of $\infty$-categories as it involves ``explicit formulas".}:

\medskip

We let $\on{Fact}(\CC)_{\CG_A}^\epsilon$ be the same as $\on{Fact}(\CC)_{\CG_A}$ as a plain sheaf of 
monoidal categories. We define the factorization structure on $\on{Fact}(\CC)_{\CG_A}^\epsilon$ as follows:

\medskip

The action of $\epsilon$ on $\CC$ defines a direct sum decomposition 
$$\on{Fact}(\CC)_{\CG_A}(S)\simeq \on{Fact}(\CC)_{\CG_A}(S)^{1}\oplus \on{Fact}(\CC)_{\CG_A}(S)^{-1}$$
for any $S\to \Ran$.

\medskip

Hence, for $S\to \Ran^J$ we have a direct sum decomposition
\begin{equation} \label{e:sign split}
(\on{Fact}(\CC)_{\CG_A})^{\otimes J}(S)\simeq \underset{\gamma^J:J\to \pm 1}\oplus\, (\on{Fact}(\CC)_{\CG_A})^{\otimes J}(S)^{\gamma^J}.
\end{equation} 

\medskip

For a given $\gamma^J$, let $J_{-1}\subset J$ be the preimage of the element $-1\in \pm 1$. 

\medskip

We define the factorization functor
for $\on{Fact}(\CC)^\epsilon_{\CG_A}(S)$ and $S\to \Ran^J_{\on{disj}}$ to be equal to the one for $\on{Fact}(\CC)_{\CG_A}(S)$ on each factor
of \eqref{e:sign split}, for every choice of an ordering on $J_{-1}$. A change of ordering will result in 
multiplication by the sign character of the group of permutations of $J_{-1}$. 

\end{rem} 

\begin{rem}
A general framework that performs both twistings
$$\on{Fact}(\CC)\rightsquigarrow \on{Fact}(\CC)_{\CG_A}^\epsilon$$
in one shot is explained in \secref{s:twist by sign}.  

\medskip

The construction in {\it loc.cit.} also makes the identification \eqref{e:epsilon twist trivial}
as plain sheaves of monoidal categories over $\Ran$, manifest. In particular, we have an identification
\begin{equation} \label{e:epsilon twist Ran}
\on{Fact}(\CC)_{\CG_A}(\Ran)\simeq \on{Fact}(\CC)_{\CG_A}^\epsilon(\Ran),
\end{equation}
as monoidal (but \emph{not} symmetric monoidal) categories. 
\end{rem}

\ssec{Twisting the category of representations}  \label{ss:tw rep}

In this subsection we will introduce a factorization sheaf of symmetric monoidal categories on the Ran space,
which will appear as the source of the metaplectic geometric Satake functor. 

\sssec{}

Let $H$ be an algebraic group. We apply the discussion in \secref{ss:tw on Ran} to the pair 
$$\CC=\Rep(H), \quad A=Z_{H}(E)^{\on{tors}}.$$

\medskip

Thus, let $\CG_Z$ be a gerbe (of finite order) on $X$ with respect to $Z_H$, and let 
$\epsilon$ be an element of order $2$ in $Z_H$. 

\sssec{}

Thus, we obtain the symmetric monoidal category $\Rep(H)_{\CG_Z}(X)$, and sheaves of symmetric monoidal categories
over $\Ran$:
$$\on{Fact}(\Rep(H))_{\CG_Z} \text{ and } \on{Fact}(\Rep(H))_{\CG_Z}^\epsilon,$$
and a \emph{monoidal} equivalence
\begin{equation} \label{e:remove epsilon}
\on{Fact}(\Rep(H))_{\CG_Z}(\Ran)\simeq \on{Fact}(\Rep(H))_{\CG_Z}^\epsilon(\Ran).
\end{equation} 

The case of interest for us is when the triple $(H,\CG_Z,\epsilon)$ is the metaplectic datum
attached to a geometric metaplectic datum of a reductive group $G$. 

\sssec{Example of tori}

Consider the particular case when $G=T$ is a torus, and we start with a factorization gerbe
$\CG^T$ on $\Gr_T$ that is multiplicative\footnote{Recall that ``multiplicative"=``commutative", see Remark \ref{r:com mult arb}.}.
In this case, 
$$\Shv_{\CG^T}(\Gr_T)_{/\Ran}$$
is naturally a sheaf of symmetric monoidal DG categories on $\Ran$, equipped with a factorization structure.

\medskip

Note also that by \propref{p:classify mult T}(a), we have $T^\sharp=T$, and so $H\simeq \cT$.
Let $(\CG_Z,\epsilon)$ be as in \secref{sss:dual triple}.

\medskip

One shows explicitly 
(see, e.g., \cite[Proposition IV.5.2]{Re}) that there is a canonical isomorphism
\begin{equation} \label{e:simple Sat tori}
\on{Fact}(\Rep(\cT))^\epsilon_{\CG_Z}\simeq \Shv_{\CG^T}(\Gr_T)_{/\Ran}
\end{equation}
as sheaves of factorization symmetric monoidal categories. 

\ssec{Twisted local systems}   \label{ss:tw loc syst}

Let $(H,\CG_Z)$ be as in \secref{ss:tw rep}. In this subsection we will 
introduce the notion of \emph{twisted local system} for $(H,\CG_Z)$.

\sssec{}  \label{sss:tw loc system naive}

Note that the category $\Rep(H)_{\CG_Z}(X)$ is naturally equipped with a t-structure. Namely, it is one for which the functor 
$$\Rep(H)_{\CG_Z}(X') \simeq \Rep(H)_{\CG_Z}(X)\simeq \Rep(H)(X') \overset{\on{forget}}\longrightarrow \Shv(X')$$
is t-exact for every \'etale $X'\to X$ over which $\CG_Z$ admits a trivialization. 

\medskip

By definition, a $\CG_Z$-twisted local system on $X$ with respect to $H$ is a t-exact $\Shv(X)$-linear symmetric monoidal functor
$$\Rep(H)_{\CG_Z}(X)\to \Shv(X).$$

\medskip

In \secref{ss:global Hecke} we will formulate a precise relationship between twisted local systems in the above sense
and objects appearing in the global metaplectic geometric theory. 

\begin{rem}
Presumably, twisted local systems as defined above are the same as Galois representations into the metaplectic L-group,
as defined in \cite{We}.
\end{rem}

\sssec{}  \label{sss:ev tw loc syst}

Let $\sigma$ be a twisted local system on $X$ as defined as above. The functoriality 
of the construction in \secref{ss:fact sym mon} defines a symmetric monoidal functor
$$\on{Fact}(\Rep(H))_{\CG_Z}(\Ran) \to \Shv(\Ran).$$

In particular, we obtain a \emph{monoidal} functor 
$$\on{Fact}(\Rep(H))^\epsilon_{\CG_Z}(\Ran) \to \Shv(\Ran).$$

\sssec{}

Assume now that $X$ is complete. Composing with the functor of direct image 
$$\Shv(\Ran)\to \Vect,$$
we thus obtain a functor
\begin{equation} \label{e:ev loc sys}
\on{Ev}_\sigma:\on{Fact}(\Rep(H))^\epsilon_{\CG_Z}(\Ran)\to \Vect.
\end{equation}

We will use the functor \eqref{e:ev loc sys} for the definition of the notion of \emph{twisted Hecke eigensheaf} 
with respect to $\sigma$. 

\sssec{}

Assume that $X$ is complete. We will now 
construct the \emph{derived} stack $\on{LocSys}_{H}^{\CG_Z}$ of $\CG_Z$-twisted local systems
on $X$. Its $k$-points will be the twisted local systems as defined in \secref{sss:tw loc system naive}.

\medskip

We follow the strategy of \cite[Sect. 10.2]{AG}. For a derived affine scheme $S$, we set
$$\Maps(S,\on{LocSys}_{H}^{\CG_Z})$$
to be the space of \emph{right t-exact} $\Shv(X)$-linear symmetric monoidal functors
$$\Rep(H)_{\CG_Z}(X)\to \QCoh(S)\otimes\Shv(X).$$ 

\medskip

One shows that $\on{LocSys}_{H}^{\CG_Z}$ defined in this way is representable 
by a quasi-smooth derived algebraic stack (see \cite[Sect. 8.1]{AG} for what this means). 

\sssec{}

As in \cite[Sect. 4.3]{Ga5}, we have a canonically defined (symmetric) monoidal functor
\begin{equation} \label{e:Loc}
\on{Loc}:\on{Fact}(\Rep(H))_{\CG_Z}(\Ran)\to \QCoh\left(\on{LocSys}_{H}^{\CG_Z}\right).
\end{equation}

\medskip

The following is proved in the same way as \cite[Proposition 4.3.4]{Ga5}\footnote{The proof is reproduced in \cite[Sect. 1.3]{Ro}.}:

\begin{prop} \label{p:Loc}
The functor \eqref{e:Loc} is a \emph{localization}, i.e., it admits a fully faithful right adjoint.
\end{prop} 

\section{Metaplectic geometric Satake}   \label{s:Satake}

We take $G$ to be a reductive group. We will continue to assume that the order of the algebraic fundamental group
of the derived group of $G$ is prime to  $\on{char}(k)$. 

\medskip

We will define the metaplectic geometric Satake
functor and formulate the ``metaplectic vanishing conjecture" about the global Hecke action. 

\ssec{The metaplectic spherical Hecke category}

In this subsection we introduce the metaplectic spherical Hecke category, which is 
the recipient of the metaplectic geometric Satake functor. 

\sssec{}

Let ${\CG^G}$ be a factorization $E^{\times,\on{tors}}$-gerbe on $\Gr_G$. We define the sheaf of categories
$(\Sph_{\CG^G})_{/\Ran}$ as follows. For an affine test scheme $S$ and an $S$-point of $\Ran$, we define the corresponding category
by
\begin{equation} \label{e:Sph}
\Sph_{\CG^G}(S):=\Shv_{\CG^G\otimes \det_\fg^{\frac{1}{2}}|_S}(S\underset{\Ran}\times \Gr_G)^{\fL^+(G)|_S}.
\end{equation} 

In the above formula, $\fL^+(G)|_S$ denotes the value on $S$ of the factorization group-scheme $\fL^+(G)$.
The superscript $\fL^+(G)|_S$ indicates the equivariant category with respect to that group-scheme. Note that the latter makes 
sense due to the structure of equivariance
on the gerbe $\CG^G\otimes \det_\fg^{\frac{1}{2}}|_S$ with respect to $\fL^+(G)|_S$, which was constructed in \secref{ss:constr equiv}.

\medskip

By \propref{p:gerbe on loop group}, we obtain that the operation of convolution product defines on $(\Sph_{\CG^G})_{/\Ran}$
a structure of factorization sheaf of \emph{monoidal} categories over $\Ran$\footnote{We wish to emphasize that the above monoidal
structure $(\Sph_{\CG^G})_{/\Ran}$ \emph{cannot} be promoted to a symmetric monoidal structure in a way compatible with factorization,
even when $G$ is a torus and $\CG^G$ is trivial.}.

\medskip

By construction, $(\Sph_{\CG^G})_{/\Ran}$ carries a natural factorization structure, see \secref{sss:fact from Z}.

\sssec{}

Let $P$ be a parabolic subgroup of $G$ with Levi quotient $M$. Let us denote by $\CG^M$
the factorization gerbe on $\Gr_M$ corresponding to $\CG^G$.

\medskip

The functor \eqref{e:corrected Jacquet} naturally upgrades to a functor
between sheaves of categories
\begin{equation} \label{e:Jacquet for Sph}
J^G_M:(\Sph_{\CG^G})_{/\Ran}\to (\Sph_{\CG^M})_{/\Ran}. 
\end{equation} 

By construction, \eqref{e:Jacquet for Sph} respects the factorization structures, i.e., it is a 
functor between factorization sheaves of categories. 

\begin{rem}
We note that the functor \eqref{e:Jacquet for Sph} is \emph{not at all} compatible with the monoidal structures!
\end{rem} 

\ssec{The metaplectic geometric Satake functor}

Let $(H,\CG_Z,\epsilon)$ be the triple of \secref{sss:dual triple} corresponding to the factorization gerbe $\CG^G$.

\medskip

Metaplectic geometric Satake is a canonically defined functor between factorization sheaves of monoidal DG categories
\begin{equation} \label{e:twisted Satake}
\on{Sat}:\on{Fact}(\Rep(H))_{\CG_Z}^\epsilon\to (\Sph_{\CG^G})_{/\Ran}. 
\end{equation}

We will now explain how to obtain this functor from \cite[Theorem IV.8.3]{Re}\footnote{For a more detailed discussion on how to carry out this
extension see \cite[Sect. 6]{Ras2}, where the classical (i.e., non-metaplectic situation) is considered, but for this step, there is no difference between
the two cases.}.

\sssec{}

By \secref{sss:Ran expl}, the datum of a functor \eqref{e:twisted Satake} amounts to a compatible collection of functors
\begin{equation} \label{e:twisted Satake I}
\on{Sat}(I):\on{Fact}(\Rep(H))_{\CG_Z}^\epsilon(X^I)\to (\Sph_{\CG^G})_{/\Ran}(X^I),
\end{equation}
where $I$ runs over the category of finite non-empty sets and surjective morphisms.

\medskip

Both sides in \eqref{e:twisted Satake I} are equipped with t-structures; moreover one shows that 
$\on{Fact}(\Rep(H))_{\CG_Z}^\epsilon(X^I)$ identifies with the \emph{derived category} of
the heart of its t-structure\footnote{Here, the derived category is understood as a DG category, see \cite[Sect. 1.3.2]{Lu2}.},
i.e., the canonical map of \cite[Theorem 1.3.3.2]{Lu2}
$$D\left(\left(\on{Fact}(\Rep(H))_{\CG_Z}^\epsilon(X^I)\right)^\heartsuit\right)\to \on{Fact}(\Rep(H))_{\CG_Z}^\epsilon(X^I)$$
is an equivalence. 

\medskip

Now, \cite[Theorem IV.8.3]{Re} constructs an \emph{equivalence} of abelian categories
\begin{equation} \label{e:twisted Satake I ab}
\left(\on{Fact}(\Rep(H))_{\CG_Z}^\epsilon(X^I)\right)^\heartsuit\to \left((\Sph_{\CG^G})_{/\Ran}(X^I)\right)^\heartsuit.
\end{equation} 

Applying \cite[Theorem 1.3.3.2]{Lu2} again, we obtain a canonically defined functor
$$D\left(\left(\on{Fact}(\Rep(H))_{\CG_Z}^\epsilon(X^I)\right)^\heartsuit\right)\to (\Sph_{\CG^G})_{/\Ran}(X^I),$$
thus giving rise to the desired functor \eqref{e:twisted Satake I}.

\medskip

The functoriality with respect to the finite sets $I$, as well as compatibility with factorization is built into the
construction. 

\sssec{Example}

Take $\CG=\det_\fg^{\frac{1}{2}}$, so that $\Sph_{\CG^G}$ corresponds to the \emph{untwisted} category of
sheaves on $\Gr_G$.

\medskip

Note that in this case, the source of geometric Satake is a \emph{twisted} category of representations of $\cG$,
see \secref{sss:det example}. 

\ssec{Example: metaplectic geometric Satake for tori}

In this subsection we let $G=T$ be a torus\footnote{We will omit the gerbe $\det_\ft^{\frac{1}{2}}$ from the notation since it is trivial
in the case of tori.}. 

\sssec{}

Let $\Lambda^\sharp\subset \Lambda$ denote the kernel of $b$. 

\medskip

Direct image along the inclusion
\begin{equation} \label{e:Grass incl}
\Gr_{T^\sharp}\to \Gr_T
\end{equation} 
is a fully faithful functor
\begin{equation} \label{e:Grass functor incl}
\Shv_{\CG^{T^\sharp}}(\Gr_{T^\sharp})_{/\Ran}\to \Shv_{\CG^T}(\Gr_T)_{/\Ran},
\end{equation}
where we denote by $\CG^{T^\sharp}$ the restriction of $\CG^T$ along \eqref{e:Grass incl}.

\medskip

In this case, it follows from \secref{ss:another interp} that the forgetful functor
$$(\Sph_{\CG^T})_{/\Ran}\to \Shv_{\CG^T}(\Gr_T)_{/\Ran}$$
factors through the essential image of \eqref{e:Grass functor incl}, thereby giving rise to a functor
\begin{equation} \label{e:Grass Sph}
(\Sph_{\CG^T})_{/\Ran}\to \Shv_{\CG^{T^\sharp}}(\Gr_{T^\sharp})_{/\Ran},
\end{equation}
compatible with the factorization structures. 

\sssec{}

Furthermore, since the action of $\fL^+(T)$ on $\Gr_T$ is trivial, the functor \eqref{e:Grass Sph}
admits a canonically defined right inverse
\begin{equation} \label{e:Grass Sph inv}
\Shv_{\CG^{T^\sharp}}(\Gr_{T^\sharp})_{/\Ran}\to (\Sph_{\CG^T})_{/\Ran},
\end{equation}
which is \emph{monoidal} and compatible with the factorization structures. 

\sssec{}

By \propref{p:classify mult T}(b), the factorization gerbe $\CG^{T^\sharp}$ carries a canonical
multiplicative structure. Recall the equivalence
\begin{equation} \label{e:simple Sat tori again}
\on{Fact}(\Rep(H))_{\CG_Z}^\epsilon\simeq \Shv_{\CG^{T^\sharp}}(\Gr_{T^\sharp})_{/\Ran}
\end{equation} 
of \eqref{e:simple Sat tori}. 

\medskip

The geometric Satake functor for $T$ is the composite of \eqref{e:simple Sat tori again} and \eqref{e:Grass Sph inv}. 

\ssec{Compatibility with Jacquet functors}

\sssec{}

A key feature of the assignment
$$\CG^G\rightsquigarrow \CG^{\pi_{1,\on{alg}}(G^\sharp)\otimes \BG_m}$$
of \secref{sss:construct gerbe pi1} is compatibility with parabolics in the following sense.

\medskip

Note that for a parabolic $P$ of $G$ with Levi quotient $M$, the corresponding reductive group
$M^\sharp$ identifies with the Levi subgroup of $G^\sharp$, attached to the same subset of the Dynkin diagram.

\medskip

We have a canonical surjection 
\begin{equation} \label{e:map of fundamental groups}
\pi_{1,\on{alg}}(M^\sharp)\to \pi_{1,\on{alg}}(G^\sharp), 
\end{equation}
and the corresponding map of factorization Grassmannians
\begin{equation} \label{e:map of fundamental groups Gr}
\Gr_{\pi_{1,\on{alg}}(M^\sharp)\otimes \BG_m}\to \Gr_{\pi_{1,\on{alg}}(G^\sharp)\otimes \BG_m}.
\end{equation}

\medskip 

Let $\CG^M$ be the factorization gerbe on $\Gr_M$ such that corresponds to 
$\CG^G$ under the map of \secref{sss:gerbes G and M}. 

\medskip

Then the multiplicative gerbe $\CG^{\pi_{1,\on{alg}}(M^\sharp)\otimes \BG_m}$ on $\Gr_{\pi_{1,\on{alg}}(M^\sharp)\otimes \BG_m}$
attached to $\CG^M$ by \secref{sss:construct gerbe pi1}
identifies with the pullback with respect to \eqref{e:map of fundamental groups Gr}
of the multiplicative gerbe $\CG^{\pi_{1,\on{alg}}(G^\sharp)\otimes \BG_m}$ on $\Gr_{\pi_{1,\on{alg}}(G^\sharp)\otimes \BG_m}$
attached to $\CG^G$. 

\sssec{}

Let $M_H$ be the standard Levi quotient in $H$ corresponding to standard Levi $M^\sharp$ of $G^\sharp$. 
Corresponding to \eqref{e:map of fundamental groups} we have the inclusion
$$Z_{H}\to Z_{M_H}.$$

By the above, this inclusion is compatible with the corresponding datum of 
$$\epsilon:\pm 1\to Z_{H}(E),\,\, \epsilon:\pm 1\to Z_{M_H}(E)$$
and the corresponding $Z_{H}$- and $Z_{M_H}$-gerbes on $X$
(we denote both by $\CG_Z$). 

\medskip

Therefore, restriction along $M_H\to H$ defines a monoidal functor
$$\on{Res}^G_M:\on{Fact}(\Rep(H))_{\CG_Z}^\epsilon\to \on{Fact}(\Rep(M_H))_{\CG_Z}^\epsilon,$$
compatible with the factorization structures.

\sssec{}

The key feature of the monoidal functor \eqref{e:twisted Satake} is that it makes the following diagram commute:
$$
\CD
\on{Fact}(\Rep(H))_{\CG_Z}^\epsilon   @>{\on{Sat}}>>    (\Sph_{\CG^G})_{/\Ran}   \\
@V{\on{Res}^G_M}VV    @VV{J^G_M}V    \\
\on{Fact}(\Rep(M_H))_{\CG_Z}^\epsilon   @>{\on{Sat}}>>    (\Sph_{\CG^M})_{/\Ran},
\endCD
$$
where $J^G_M$ is the Jacquet functor of \eqref{e:Jacquet for Sph}. 

\ssec{Global Hecke action}    \label{ss:global Hecke}

In this subsection we will assume that $X$ is complete. We will define the notion of Hecke eigensheaf on $\Bun_G$ with
respect to a given twisted local system. 

\sssec{}  \label{sss:global Hecke action}

Consider category of \emph{global sections} of $(\Sph_{\CG^G})_{/\Ran}$ over $\Ran$ (see \secref{sss:global sections}), denote it by
$$\Sph_{\CG^G}(\Ran),$$
and note that it identifies with
$$\Shv_{\CG^G\otimes \det_\fg^{\frac{1}{2}}}(\Gr_G)^{\fL^+(G)}.$$

\medskip

As in \cite[Sect. 4.4]{Ga5}, the monoidal structure on $(\Sph_{\CG^G})_{/\Ran}$, and the operation of union of finite sets,
define a (non-unital) monoidal structure on $\Sph_{\CG^G}(\Ran)$. 

\medskip

Moreover, the Hecke action defines a monoidal action of $\Sph_{\CG^G}(\Ran)$ on $\Shv_{\CG^G\otimes \det_\fg^{\frac{1}{2}}}(\Bun_G)$, 
where by a slight abuse of notation we denote by the same symbols $\CG^G$ and $\det_\fg^{\frac{1}{2}}$ the corresponding 
$E^{\times,\on{tors}}$-gerbes on $\Bun_G$, see \secref{sss:global}. 

\sssec{}  \label{sss:Rep action}

Passing to global sections over $\Ran$ in \eqref{e:twisted Satake}, we obtain a monoidal functor 
$$\on{Fact}(\Rep(H))_{\CG_Z}^\epsilon(\Ran)\to \Sph_{\CG^G}(\Ran),$$
where we remind that $\on{Fact}(\Rep(H))_{\CG_Z}^\epsilon(\Ran)$ denotes the monoidal
category of global sections of $\on{Fact}(\Rep(H))_{\CG_Z}^\epsilon$.

\medskip

Thus, we obtain a monoidal action of $\on{Fact}(\Rep(H))_{\CG_Z}^\epsilon(\Ran)$ on 
$\Shv_{\CG^G\otimes \det_\fg^{\frac{1}{2}}}(\Bun_G)$. 

\sssec{Hecke eigensheaves}  

Let $\sigma$ be a twisted local system on $X$, as defined in \secref{sss:tw loc system naive}. Recall
(see \secref{sss:ev tw loc syst}) that $\sigma$ gives rise to a (symmetric) monoidal functor
$$\on{Ev}_\sigma:\on{Fact}(\Rep(H))_{\CG_Z}(\Ran)\to \Vect,$$
and hence, via the monoidal equivalence \eqref{e:remove epsilon}
to a monoidal functor
$$\on{Fact}(\Rep(H))_{\CG_Z}^\epsilon(\Ran)\to \Vect,$$
which we will denote by the same symbol $\on{Ev}_\sigma$. 

\medskip

We define the category of twisted Hecke eigensheaves with respect to $\sigma$ to be the DG category of functors
of $\on{Fact}(\Rep(H))_{\CG_Z}^\epsilon(\Ran)$-module categories
$$\Vect\to \Shv_{\CG^G\otimes \det_\fg^{\frac{1}{2}}}(\Bun_G),$$
where $\on{Fact}(\Rep(H))_{\CG_Z}^\epsilon(\Ran)$ acts on $\Vect$ via $\on{Ev}_\sigma$ and on 
$\Shv_{\CG^G\otimes \det_\fg^{\frac{1}{2}}}(\Bun_G)$ as in \secref{sss:Rep action}.

\ssec{The metaplectic vanishing conjecture}   \label{ss:global Hecke Dmod}

We continue to assume that $X$ is complete. 
Recall the (derived) stack $\on{LocSys}_{H}^{\CG_Z}$, see \secref{ss:tw loc syst}

\medskip

We will state a conjecture to the effect that the (non-unital) monoidal category
$$\QCoh(\on{LocSys}_{H}^{\CG_Z})$$ acts on the category
$$\Shv_{\CG^G\otimes \det_\fg^{\frac{1}{2}}}(\Bun_G).$$

\sssec{}

Recall (see \propref{p:Loc}) that we have a (symmetric) monoidal functor
$$\on{Loc}:\on{Fact}(\Rep(H))_{\CG_Z}(\Ran)\to \QCoh\left(\on{LocSys}_{H}^{\CG_Z}\right)$$
of \eqref{e:Loc} with a fully faithful right adjoint. Hence, by \eqref{e:remove epsilon}, we obtain a 
monoidal functor, denoted by the same symbol 
$$\on{Loc}:\on{Fact}(\Rep(H))_{\CG_Z}^\epsilon(\Ran)\to \QCoh\left(\on{LocSys}_{H}^{\CG_Z}\right),$$
also with a fully faithful right adjoint. 

\medskip

The following is an analog of \cite[Theorem 4.5.2]{Ga5} in the metaplectic case:

\begin{conj} \label{c:vanishing}
If an object of $\on{Fact}(\Rep(H))_{\CG_Z}^\epsilon(\Ran)$ lies in the kernel of the functor $\on{Loc}$, then this
object acts by zero on $\Shv_{\CG^G\otimes \det_\fg^{\frac{1}{2}}}(\Bun_G)$. 
\end{conj} 

This conjecture can be restated as follows:

\begin{conj} \label{c:vanishing bis}
The action of $\on{Fact}(\Rep(H))_{\CG_Z}^\epsilon(\Ran)$ on 
$\Shv_{\CG^G\otimes \det_\fg^{\frac{1}{2}}}(\Bun_G)$ (uniquely) factors through an action of
$\QCoh\left(\on{LocSys}_{H}^{\CG_Z}\right)$.
\end{conj}

\begin{rem}
Using Fourier-Mukai transform, one can show that \conjref{c:vanishing} holds when $G=T$ is a torus, see \cite{Lys}. 
\end{rem} 

\sssec{}

Let us assume \conjref{c:vanishing bis}, so that $\Shv_{\CG^G\otimes \det_\fg^{\frac{1}{2}}}(\Bun_G)$ becomes a module category
over $\QCoh\left(\on{LocSys}_{H}^{\CG_Z}\right)$. 

\medskip

As in the classical (i.e., non-metaplectic case), one expects that $\Shv_{\CG^G\otimes \det_\fg^{\frac{1}{2}}}(\Bun_G)$ is ``almost" free
of rank one, and the ``almost" has to do with temperedness.

\medskip

More precisely, one expects that the metaplectic geometric Satake functor \eqref{e:twisted Satake} extends to a 
\emph{derived metaplectic geometric Satake equivalence}, generalizing \cite[Sects. 4.6 and 4.7]{Ga5},
which one can use in order to define the \emph{tempered part} of $\Shv_{\CG^G\otimes \det_\fg^{\frac{1}{2}}}(\Bun_G)$,
as in \cite[Sect. 12.8]{AG}. 

\medskip

Now, one expects that the tempered subcategory of $\Shv_{\CG^G\otimes \det_\fg^{\frac{1}{2}}}(\Bun_G)$ \emph{is}
free of rank one as a module over $\QCoh\left(\on{LocSys}_{H}^{\CG_Z}\right)$. 

\medskip

However, it is not clear whether this module admits a distinguished generator. 

\sssec{}  \label{sss:nilp conj}

Furthermore, one expects that the entire $\Shv_{\CG^G\otimes \det_\fg^{\frac{1}{2}}}(\Bun_G)$ is \emph{non-canonically} equivalent to the
category $\IndCoh_{\on{nilp}}\left(\on{LocSys}_{H}^{\CG_Z}\right)$, where we refer the reader to \cite[Sect. 11.1]{AG} for 
the $\IndCoh_{\on{nilp}}$ notation. 

\sssec{}

When $G=T$ is a torus, we have 
$$\IndCoh_{\on{nilp}}\left(\on{LocSys}_{H}^{\CG_Z}\right)=\QCoh\left(\on{LocSys}_{H}^{\CG_Z}\right).$$

In particular, the equivalence of \secref{sss:nilp conj} says that for each $\sigma\in \on{LocSys}_{H}^{\CG_Z}$,
the corresponding category of Hecke eigensheaves is non-canonically equivalent to $\Vect$. This equivalence
can be made explicit as follows (see \cite{Lys} for more details): 

\medskip

A point $\sigma\in \on{LocSys}_{H}^{\CG_Z}$ gives rise to a trivialization of the pullback of the gerbe  
$\CG^T$ from $\Bun_T$ to $\Bun_{T^\sharp}$. Hence, it gives rise to a central extension
$$1\to E^\times\to \on{Heis}_\sigma\to \Bun_{\on{ker}(T^\sharp\to T)}\to 1,$$
which is easily seen to be of Heinsenberg type, i.e., corresponding to a \emph{non-degenerate} symplectic form
on $\on{ker}(T^\sharp\to T)$ with values in $E^\times$. 

\medskip

The category of Hecke eigensheaves with respect to $\sigma$ is \emph{canonically} equivalent to
$$(\Shv_{\CG^T}(\Bun_T))^{\Bun_{T^\sharp}},$$
where the $\Bun_{T^\sharp}$-equivariance makes sense due to the above trivialization of $\CG|_{\Bun_{T^\sharp}}$.
This category is \emph{canonically} equivalent to the category
of representations of $\on{Heis}_\sigma$, on which $E^\times$ acts by the standard character. 

\medskip

Since $\on{Heis}_\sigma$
is of Heinsenberg type, the above category is \emph{non-canonically} equivalent to $\Vect$. 

\sssec{}

At the moment, we \emph{do not} have a conjecture as to how to explicitly describe the category of Hecke eigensheaves in
the tempered subcategory of $\Shv_{\CG^G\otimes \det_\fg^{\frac{1}{2}}}(\Bun_G)$ with respect to a given $\sigma$
for a general reductive $G$. 

\appendix

\section{The affine Grassmannian attached to finitely generated abelian groups} \label{s:finite Gr}

\ssec{The group-stack attached to a finitely generated abelian group}  \label{ss:class finite} 

Let $\Gamma$ be a finitely generated abelian group, whose torsion part has order prime to  $\on{char}(k)$. 
We attach to it the group-stack 
$$\Gamma\otimes \BG_m,$$
as follows:

\medskip

Write $\Gamma$ as a quotient of two lattices $\Lambda_1/\Lambda_2$ with associated tori $T_i$. We set
$$\Gamma\otimes \BG_m:=T_1/T_2.$$

It is easy to see that this definition is canonically independent of the presentation of $\Gamma$ as a quotient.

\medskip

Explicitly, writing $\Gamma$ as $\Gamma^{\on{free}}\oplus \Gamma^{\on{tors}}$, we have
$$\Gamma\otimes \BG_m \simeq T \times B_{\on{et}}(\Gamma^{\on{tors}}(1)),$$
where $T$ is the torus whose lattice of cocharacters is $\Gamma^{\on{free}}$. 

\ssec{Maps from an algebraic group}    \label{ss:grp to pi1} 

Let $G$ be a reductive group (such that the torsion part of the fundamental group 
has order prime to $\on{char}(k)$). 

\medskip

We claim that we have a canonically defined map
$$G\to \pi_{1,\on{alg}}\otimes \BG_m.$$

Indeed, write $G$ as 
\begin{equation} \label{e:cover group again}
1\to T_2\to \wt{G}_1\to G\to 1
\end{equation}
as in \eqref{e:cover group}. Recall that $T_1$ denotes the torus $\wt{G}_1/[\wt{G}_1,\wt{G}_1]$. 

\medskip

From here we obtain a canonical map
$$G \simeq \wt{G}_1/T_2\to T_1/T_2 \simeq \pi_{1,\on{alg}}\otimes \BG_m.$$

\ssec{The affine Grassmannian}

Since $\Gamma\otimes \BG_m$ is a (commutative) group-object of $\on{PreStk}$, we can consider its
classifying space
$$B_{\on{et}}(\Gamma\otimes \BG_m)\in \on{Ptd}(\on{PreStk}),$$
and the corresponding affine Grassmannian,
$$\Gr_{\Gamma\otimes \BG_m}.$$ 

\sssec{}

Note that if $\Gamma$ is finite, we have
$$B_{\on{et}}(\Gamma\otimes \BG_m)=B^2_{\on{et}}(\Gamma(1)).$$

So in this case $\Gr_{\Gamma\otimes \BG_m}$ classifies $\Gamma(1)$-gerbes on the curve with a trivialization of the
punctured curves. 

\medskip

In other words, for a test scheme $S$ and an $S$-point of $\Ran$ given by $I\subset \Hom(S,X)$,
its lift to a point of $\Gr_{\Gamma\otimes \BG_m}$ is the same as a section of
$$\on{C}^\bullet(\Gamma_I,\iota^!(\Gamma_{S\times X}(1)[2]))\simeq 
\on{C}^\bullet(\Gamma_I,\pi^!(\Gamma_S)),$$
where we recall that $\pi$ denotes the projection $\Gamma_I\to S$. 

\sssec{}

Assume again that $\Gamma$ is written as a quotient of two lattices $\Gamma_1/\Gamma_2$.
Consider the corresponding map
\begin{equation} \label{e:Gr lattive to finite}
\Gr_{T_1}\to \Gr_{\Gamma\otimes \BG_m}.
\end{equation}

We claim:

\begin{thm} \label{t:Gr lattive to finite} \hfill

\smallskip

\noindent{\em(a)} The map \eqref{e:Gr lattive to finite} is ind-finite (i.e., its 
base change by an affine scheme $S$ yields an ind-scheme, which is ind-finite over $S$).

\smallskip

\noindent{\em(b)} The resulting map
$$\Gr_{T_1}/\Gr_{T_2}\to  \Gr_{\Gamma\otimes \BG_m}$$
is an isomorphism in the topology generated by finite surjective maps.

\end{thm} 

The rest of this subsection is devoted to the proof of this theorem. 
It is easy to see that it is sufficient to consider the case of
the short exact sequence
$$0\to \BZ \overset{n\cdot}\to \BZ\to \BZ/n\BZ\to 0.$$

\sssec{} \label{sss:bijective points}

First, we claim that the map 
$$\on{coFib}(\Gr_{\BG_m} \overset{n\cdot}\to \Gr_{\BG_m}) \to \Gr_{\BZ/n\BZ\otimes \BG_m}$$
is bijective at the level of field-valued points, where $\on{coFib}(-)$ is taken in the category
$\on{ComGrp}(\on{PreStk})$. 

\medskip

Indeed, for a curve $X$ with a point $x$, the space of its lifts to a point of $\Gr_{\BZ/n\BZ\otimes \BG_m}$
is the space of $\mu_n$-gerbes on $X$, equipped with a trivialization over $X-x$. 
The Kummer sequence
\begin{equation} \label{e:Kummer again again}
0\to \mu_n\to \CO^\times \overset{n}\to \CO^\times \to 0.
\end{equation}
identifies this space with $\BZ/n\BZ$, as required. 

\sssec{}

It follows from the definitions that the fiber of the map
$$\Gr_{T_1}\to \Gr_{\Gamma\otimes \BG_m}$$
(as group prestacks over $\Ran$) 
identifies canonically with $\Gr_{T_2}$. Hence, in order to prove point (b), it suffices to show that the map
\eqref{e:Gr lattive to finite} is surjective after sheafification in the topology generated by finite surjective maps.

\medskip

Given \secref{sss:bijective points}, we obtain that point (b) of the theorem follows from point (a). 

\medskip

Furthermore, by \secref{sss:bijective points}, for point (a), it suffices to show that the morphism \eqref{e:Gr lattive to finite} is ind-proper 
(i.e., its base change by a scheme $S$ yields an ind-scheme, which is ind-proper over $S$).

\medskip

We proceed with the proof of the latter fact. 

\sssec{}

For an affine test-scheme $S$, let us be given an $S$-point of $\Gr_{\BZ/n\BZ\otimes \BG_m}$. I.e., we have a finite 
non-empty subset
$$I\subset \Hom(S,X),$$
a $\mu_n$-gerbe $\CG$ on $S\times X$ equipped with a trivialization $\alpha$ over $U_I$. The fiber product
$$S\underset{\Gr_{\BZ/n\BZ\otimes \BG_m}}\times \Gr_{\BG_m}$$
is the functor on 
schemes over $S$ that attaches to $S'\to S$ the space of $(\CL',\alpha',\beta)$, where $\CL'$ is a line bundle on $S'\times X$, 
$\alpha'$ is its trivialization over $U'_I:=S'\underset{S}\times U_I$, and $\beta$ is an identification 
of $(\CG,\alpha)|_{S'}$ with the datum induced by $(\CL',\alpha')$
by the Kummer sequence \eqref{e:Kummer again again}. We need to show that this functor s representable 
by an ind-scheme locally of finite type over $S$.

\medskip

With no restriction of generality, we can assume that $X$ is affine. Then, 
after \'etale localization with respect to $S$, we can assume that $\CG$ can be trivialized. In this case, we can assume that the pair
$(\CG,\alpha)$ also comes from the Kummer sequence. I.e., it is given by a line bundle $\CL_U$ on $U_I$ whose $n$-th power
is extended over $S\times X$. 

\medskip

In this case, for $S'\to S$, the datum of $(\CL,\alpha',\beta)$ as above is equivalent to the datum 
of extension of $\CL_U|_{S'}$ to a line bundle over $S'\times X$. We wish to show that this functor is
representable by an ind-scheme which is ind-proper over $S$. This is a particular case of the 
following general assertion:

\begin{thm} \label{t:extension} Let $G$ be an algebraic group. Then for $S$ and $I\subset \Hom(S,X)$ as above, 
for a $G$-bundle $\CP_U$ over $U_I$, the functor $\on{Ext}(\CP_U)$ that attaches to $S'\to S$ the set of extensions of $\CP_U|_{U'_I}$ to
$S'\times X$ is representable by an ind-scheme locally of finite type over $S$. If $G$ is reductive, then this ind-scheme is ind-proper.
\end{thm}

\qed[\thmref{t:Gr lattive to finite}]

\ssec{Proof of \thmref{t:extension}}

\sssec{}

First off, it is easy to see that we can assume that $S$ is itself locally of finite type over the ground field. Further, it is easy to see that
in this case, the functor $\on{Ext}(\CP_U)$ is locally of finite type (see \cite[Chapter 3, Sect.1.5.2]{GR1}). 

\medskip

Below we will show that $\on{Ext}(\CP_U)$ is an ind-scheme\footnote{It is automatically of ind-finite type because 
$\on{Ext}(\CP_U)$ is locally of finite type as a functor.}. Once this is done, the proof that it is ind-proper 
(for $G$ reductive) would follow from the valuative criterion:

\medskip

Indeed, choose an embedding $G\hookrightarrow GL_n$. Since $G$ is reductive, the quotient $GL_n/G$ is affine.
A standard argument thus reduces the proof to the case $G=GL_n$.

\medskip

Let $S'=C$ be a curve with a marked point $c$, and let be given a map $C\to S$ and $(C-c)\to \on{Ext}(\CP_U)$. 
We wish to be able to extend the above map to a map $C\to \on{Ext}(\CP_U)$.

\medskip

By definition, we are given a vector bundle on $(C-c)\times X$ and $C\underset{S}\times U_I$ with a datum of
compatibility on the overlap. Now, the scheme $C\times X$ is a regular surface and the open subscheme
$$((C-c)\times X) \cup (C\underset{S}\times U_I) \subset C\times X$$
has a complement of codimension $2$. This implies that our $G$-bundle indeed admits a unique extension.

\sssec{}

To prove ind-representability, with no restriction of generality, we can assume that $X$ is complete and that
$U_I$ contains a subset of the form $S\times \{x_0\}$ for a point $x_0\in X$. 

\medskip

After localization with respect to $S$ we may assume that $\CP_U|_{S\times x_0}$ is trivial.

\medskip

Let $\Bun_G^{x_0}$ be the modui space of $G$-bundles on $X$ with a \emph{full} level structure at $x_0$. It is
known to be a \emph{scheme} (but of infinite type). We have a forgetful map
$$\on{Ext}(\CP_U)\to S\times \Bun_G^{x_0},$$
and it is sufficient to show that that it is a relative ind-scheme. 

\medskip

Given $S'\to S$ and an $S'$-point of $\Bun_G^{x_0}$, corresponding to a $G$-bundle $\CP'$ on $S'\times X$,
the datum of its lift to an $S'$-point of $\on{Ext}(\CP_U)$ is a choice of an isomorphism 
$$\CP'|_{U'_I}\simeq \CP_U|_{U'_I}$$
compatible with the level structures at $x_0$. 

\medskip

Let $\on{Isom}$ be the (affine) scheme over $U'_I$ that attaches to $T\to U'_I$ the set of isomorphisms
$$\CP'|_T\simeq \CP_U|_T$$
compatible with the level structures at $x_0$. 

\medskip

We claim that the functor on schemes over $S'$ that sends $S''\to S'$ to the set of maps of $S'$-schemes
$$S''\to \on{Isom}$$
is ind-representable. 

\sssec{}

We claim that the latter fact holds for $\on{Isom}$ replaced by any affine scheme $Z$ over $U'_I$. 

\medskip

Indeed, by a standard argument we can replace $Z$ by the affine line $\BA^1$ over $S'$. Then the functor in question
attaches to $S''\to S'$ the set of regular functions on $U'_I$. 

\medskip

Let $\CF$ be the direct image of the structure sheaf along $U'_I\to S'$. We can write it as a union of vector bundles 
$\CE_i$. The the above functor is the direct limit of the total spaces of these vector bundles. 
 
\qed[\thmref{t:extension}]

\section{Calculation of the \'etale cohomology of $B_{\on{et}}(G) $}  \label{s:BG}

\ssec{The Leray spectral sequence}

The calculation is based on considering the Leray spectral sequence associated with the 
projection
$$\pi:B(B)\to B_{\on{et}}(G) ,$$
where $B\subset G$ is the Borel subgroup. 

\medskip

Namely, let $\ul{A}$ denote the constant \'etale sheaf on either $B_{\on{et}}(G) $ or $B(B)$ with coefficients in $A$,
and let us consider the exact triangle
\begin{equation} \label{e:Leray}
\ul{A}\to \pi_*(\ul{A})\to \tau^{\geq 1}(R\pi_*(\ul{A})).
\end{equation} 

We note that each individual cohomology sheaf $R^i\pi_*(\ul{A})$ is constant with fiber $H^i_{\on{et}}(G/B,A)$. 

\medskip

Note also that the projection $B(B)\to B_{\on{et}}(T) $ defines an isomorphism an \'etale cohomology, so we obtain:

\begin{equation}
H^i_{\on{et}}(B(B),A)\simeq H^i_{\on{et}}(B_{\on{et}}(T) ,A)\simeq 
\begin{cases}
&0 \text{ for $i$ odd}; \\
&\Hom(\Lambda,A(-1)) \text{ for $i=2$}; \\
&\on{Quad}(\Lambda,A(-2)) \text{ for $i=4$}.
\end{cases}
\end{equation} 

\ssec{Cohomology in degrees $\leq 3$}

From the long exact cohomology sequence associated with \eqref{e:Leray} we immediately obtain that $H^1_{\on{et}}(B_{\on{et}}(G) ,A)=0$.

\medskip

Next, the fact that $H^1_{\on{et}}(G/B,A)=0$ implies that the map
$$H^2_{\on{et}}(B_{\on{et}}(G) ,A)\to H^2_{\on{et}}(B(B),A)$$ is injective with image equal to the kernel of the map
\begin{equation} \label{e:map H2}
H^2_{\on{et}}(B(B),A)\to H^2_{\on{et}}(G/B,A).
\end{equation} 

We identify $H^2_{\on{et}}(G/B,A)=\Hom(\Lambda_{\on{sc}},A(-1))$, where $\Lambda_{\on{sc}}$ is the coroot lattice in $\Lambda$, 
and the map \eqref{e:map H2} becomes the restriction map
\begin{equation} \label{e:map H2 again}
\Hom(\Lambda,A(-1)) \to \Hom(\Lambda_{\on{sc}},A(-1)).
\end{equation} 

Since $\pi_{1,\on{alg}}(G)=\Lambda/\Lambda_{\on{sc}}$, we obtain the desired identification
$$H^2_{\on{et}}(B_{\on{et}}(G) ,A)\simeq \Hom(\pi_{1,\on{alg}}(G),A).$$

Now, since $A$ was assumed divisible, the map \eqref{e:map H2 again} is surjective. Hence, the map
$$H^3_{\on{et}}(B_{\on{et}}(G) ,A)\to H^3_{\on{et}}(B(B),A)$$
is injective. Since $H^3_{\on{et}}(B(B),A)=0$, we obtain the desired $H^3_{\on{et}}(B_{\on{et}}(G) ,A)=0$.

\ssec{Cohomology in degree 4: injectivity}

We will now show that the map
$$H^4_{\on{et}}(B_{\on{et}}(G) ,A)\to H^4_{\on{et}}(B(B),A)$$
is injective.

\medskip

For this, it suffices to show that
$$H^3_{\on{et}}(B_{\on{et}}(G) ,\tau^{\geq 1}(R\pi_*(\ul{A})))=0.$$

Since, $H^3_{\on{et}}(G/B,A)=0$, we have
$$H^3_{\on{et}}(B_{\on{et}}(G) ,\tau^{\geq 1}(R\pi_*(\ul{A})))=H^1_{\on{et}}(B_{\on{et}}(G) ,H^2_{\on{et}}(G/B,A)),$$
and the latter vanishes as $H^1_{\on{et}}(B_{\on{et}}(G) ,-)=0$.

\medskip

Thus, we obtain an injection 
$$H^4_{\on{et}}(B_{\on{et}}(G) ,A)\hookrightarrow H^4_{\on{et}}(B(B),A)\simeq H^4_{\on{et}}(B_{\on{et}}(T) ,A)\simeq  \on{Quad}(\Lambda,A(-2)),$$
and out task is to show that its image equals $\on{Quad}(\Lambda,A(-2))^W_{\on{restr}}$. 

\ssec{Containment in one direction} \label{ss:containment one}

We will first show that the image of $H^4_{\on{et}}(B_{\on{et}}(G) ,A)$ in $\on{Quad}(\Lambda,A(-2))$
is contained in $\on{Quad}(\Lambda,A(-2))^W_{\on{restr}}$. 

\medskip 

Note that realizing $T$ is a Cartan \emph{subgroup} of $G$, we obtain a commutative diagram
$$
\CD
H^4_{\on{et}}(B_{\on{et}}(T) ,A)  @>{\simeq}>> \on{Quad}(\Lambda_G,A(-2)) \\
@AAA @AA{\on{id}}A \\
H^4_{\on{et}}(B_{\on{et}}(G) ,A)  @>{\simeq}>> \on{Quad}(\Lambda_G,A(-2)),
\endCD
$$
from which it follows that the image of $H^4_{\on{et}}(B_{\on{et}}(G) ,A)$ in $\on{Quad}(\Lambda,A(-2))$
is a priori contained in $\on{Quad}(\Lambda,A(-2))^W$. 

\medskip

Thus, it remains to show that for any $q\in \on{Quad}(\Lambda,A(-2))$ that lies in the image
of the above map, and any \emph{simple} coroot $\alpha_i$, we have
$$b(\alpha_i,\lambda)=\langle \check\alpha_i,\lambda\rangle\cdot q(\alpha_i)\text{ for any} \lambda\in \Lambda.$$ 

Let $P_i$ be the subminimal parabolic associated with $i$, and let $M_i$ be its Levi quotient. We have a commutative diagram
$$
\CD
H^4_{\on{et}}(B(M_i),A) @>>> H^4_{\on{et}}(B(B),A) \\
@V{\sim}VV  @VV{=}V   \\
H^4_{\on{et}}(B(P_i),A) @>>> H^4_{\on{et}}(B(B),A)  \\
@AAA  @AA{=}A \\
H^4_{\on{et}}(B_{\on{et}}(G) ,A) @>>> H^4_{\on{et}}(B(B),A), \\
\endCD
$$
which implies that it is sufficient to prove our claim for $G$ replaced by $M_i$, which is a reductive group of semi-simple rank $1$.

\ssec{Calculation for groups of semi-simple rank $1$} \label{ss:ss rank 1}

Any group $G$ of semi-simple rank $1$ is of the form 
$$G'\times T',$$
where $G'$ is $SL_2$, $PGL_2$ or $GL_2$ and $T'$ is a torus. 

\medskip

If $G'=SL_2$, then 
$$H^4_{\on{et}}(B_{\on{et}}(G) ,A)\simeq H^4_{\on{et}}(B(G'),A)\oplus H^4_{\on{et}}(B(T'),A).$$

Similarly, in this case, it is easy to see that in this case\footnote{Note that the formula below would be \emph{false} if we considered 
$\on{Quad}(\Lambda_G,A(-2))^W$ instead of $\on{Quad}(\Lambda_G,A(-2))^W_{\on{restr}}$; for example it would be false for the group
$SL_2\times \BG_m$.}
$$\on{Quad}(\Lambda_G,A(-2))^W_{\on{restr}}=\on{Quad}(\Lambda_{G'},A(-2))^W_{\on{restr}}\oplus \on{Quad}(\Lambda_{T'},A(-2)).$$
So the assertion reduces to the case $G=G'=SL_2$. Note, however, that in this case, the inclusions 
$$\on{Quad}(\Lambda_{SL_2},A(-2))^W_{\on{restr}}\subseteq \on{Quad}(\Lambda_{SL_2},A(-2))^W \subseteq \on{Quad}(\Lambda_{SL_2},A(-2))$$
are equalities, and the assertion follows.

\medskip

In the two cases of $G'=PGL_2$ or $G'=GL_2$, since the unique positive coroot is divisible by 2,  
the inclusion
$$\on{Quad}(\Lambda_G,A(-2))^W_{\on{restr}}\subseteq \on{Quad}(\Lambda_G,A(-2))^W$$
is an equality, and the assertion follows.

\ssec{The opposite containment}

It remains to show that any element $q\in \on{Quad}(\Lambda,A(-2))^W_{\on{restr}}$ lies in the image of
$H^4_{\on{et}}(B_{\on{et}}(G) ,A)$ in $\on{Quad}(\Lambda,A(-2))$.

\medskip

According to \secref{sss:quad restr bis}, we have to consider the following two cases: 

\medskip

\noindent(I) $q$ lies in the image of the map
$$\on{Quad}(\Lambda,\BZ)^W\underset{\BZ}\otimes A(-2)\to  \on{Quad}(\Lambda,A(-2))^W_{\on{restr}}.$$

\medskip

\noindent(II) $q$ comes from a quadratic form on $\pi_{1,\on{alg}}(G)$. 

\medskip

We first deal with case II. 
Consider the (2)-stack $B_{\on{et}}(\pi_{1,\on{alg}}(G)\otimes \BG_m)$, see \secref{ss:class finite}. 

\medskip

We have a canonical isomorphism $$H^4_{\on{et}}(B_{\on{et}}(\pi_{1,\on{alg}}(G)\otimes \BG_m),A)\simeq \on{Quad}(\pi_{1,\on{alg}}(G),A(-2))$$
(see \corref{c:classify mult arb new new}) which fits into the commutative diagram
$$
\CD
H^4_{\on{et}}(B_{\on{et}}(T) ,A)   @>{\sim}>>  \on{Quad}(\Lambda,A(-2))  \\
@AAA   @AAA  \\
H^4_{\on{et}}(B_{\on{et}}(\pi_{1,\on{alg}}(G)\otimes \BG_m),A)  @>{\sim}>> \on{Quad}(\pi_{1,\on{alg}}(G),A(-2)).
\endCD
$$

Now the desired containment follows from the commutative diagram
$$
\CD
H^4_{\on{et}}(B_{\on{et}}(G) ,A) @>>> H^4_{\on{et}}(B(B),A) @<{\sim}<< H^4_{\on{et}}(B_{\on{et}}(T) ,A)  \\
@AAA   & &  @AA{\on{id}}A  \\
H^4_{\on{et}}(B_{\on{et}}(\pi_{1,\on{alg}}(G)\otimes \BG_m),A)  & @>>> & H^4_{\on{et}}(B_{\on{et}}(T) ,A),
\endCD
$$
where the left vertical arrow comes from the canonical projection 
$$B_{\on{et}}(G) \to B_{\on{et}}(\pi_{1,\on{alg}}(G)\otimes \BG_m),$$
see \secref{ss:grp to pi1}. 

\medskip

In order to deal with case I, it suffices to show that for any $\ell$ comprime with $\on{char}(k)$, the map 
$$H^4_{\on{et}}(B_{\on{et}}(G) ,\BZ_\ell) \to \on{Quad}(\Lambda,\BZ_\ell(-2))^W$$
is an isomorphism. 

\ssec{Computation of the integral cohomology}

From the long exact cohomology sequence associated with \eqref{e:Leray}, we obtain that the image of
\begin{equation} \label{e:H4 int}
H^4_{\on{et}}(B_{\on{et}}(G) ,\BZ_\ell) \to H^4_{\on{et}}(B_{\on{et}}(T) ,\BZ_\ell)
\end{equation}
equals 
$$\on{ker}\left(\on{ker}(H^4_{\on{et}}(B_{\on{et}}(T) ,\BZ_\ell)\to H^4_{\on{et}}(G/B,\BZ_\ell))\to H^2_{\on{et}}(B_{\on{et}}(G) ,H^2_{\on{et}}(G/B,\BZ_\ell))\right).$$

Since both groups $H^4_{\on{et}}(G/B,\BZ_\ell)$ and 
$$H^2_{\on{et}}(B_{\on{et}}(G) ,H^2_{\on{et}}(G/B,\BZ_\ell))\simeq \Hom(\pi_{1,\on{alg}}(G),H^2_{\on{et}}(G/B,\BZ_\ell))$$
are torsion-free, we obtain that
the image of \eqref{e:H4 int} equals
$$H^4_{\on{et}}(B_{\on{et}}(T) ,\BZ_\ell)\cap \on{Im}\left(H^4_{\on{et}}(B_{\on{et}}(G) ,\BQ_\ell) \to H^4_{\on{et}}(B_{\on{et}}(T) ,\BQ_\ell)\right).$$

However,
$$H^4_{\on{et}}(B_{\on{et}}(T) ,\BZ_\ell)\simeq \on{Quad}(\Lambda,\BZ_\ell(-2)),$$
and rationally, we know that
$$H^4_{\on{et}}(B_{\on{et}}(G) ,\BQ_\ell) \simeq \on{Quad}(\Lambda,\BQ_\ell(-2))^W.$$

Hence, the image of \eqref{e:H4 int} equals
$$ \on{Quad}(\Lambda,\BZ_\ell(-2))\cap \on{Quad}(\Lambda,\BQ_\ell(-2))^W=\on{Quad}(\Lambda,\BZ_\ell(-2))^W,$$
as desired. 

\section{Factorization gerbes via bilinear forms}

The material in this section is informed by the contents of Deligne's letter to Lusztig, \cite[P.S.]{Del}. 
It will allow us to obtain an even more explicit parameterization of factorization gerbes. 

\ssec{The complex of bilinear forms}

Let $\Lambda$ be a lattice and $A$ a divisible torsion abelian group. We consider the following complex, to be denoted $\cD(\Lambda)$,
placed in degrees $[-2,0]$: 

\begin{equation} \label{e:Deligne}
\on{Quad}(\Lambda,A(-1)) \overset{d_2}\to \on{Bilin}(\Lambda,A(-1)) \overset{d_1}\to \on{Bilin}(\Lambda,A(-1)),
\end{equation} 
where 

\medskip

\noindent--the map $\on{Quad}(\Lambda,A(-1)) \overset{d_2}\to \on{Bilin}(\Lambda,A(-1))$ is the usual map from quadratic forms to (symmetric) bilinear forms;

\medskip

\noindent--the map $\on{Bilin}(\Lambda,A(-1)) \overset{d_1}\to \on{Bilin}(\Lambda,A(-1))$
is
$$b''\mapsto b', \quad b'(\lambda,\mu)=b''(\lambda,\mu)-b''(\mu,\lambda).$$

\sssec{}

This complex is acyclic in degree $-1$. Its $0$th cohomology identifies with 
$$\on{Quad}(\Lambda,A(-1))$$ via
$$b'\mapsto q, \quad q(\lambda)=b'(\lambda,\lambda).$$

Its cohomology in degree $(-2)$ identifies with
$$\Hom(\Lambda,A(-1)_{2\on{-tors}})\simeq \Hom(\Lambda,A_{2\on{-tors}}).$$

\sssec{}

We note that the complex $\cD(\Lambda)$ contains a quasi-isomorphic sub-complex, denoted $\wt\cD(\Lambda)$:

\begin{equation} \label{e:pre Deligne}
\on{Quad}(\Lambda,A_{2\on{-tors}})\overset{\wt{d}_2}\to \on{Alt}(\Lambda,A(-1)) \overset{\wt{d}_1}\to \on{Bilin}(\Lambda,A(-1)),
\end{equation} 
where 
$$\on{Alt}(\Lambda,A(-1)) \subset \on{Bilin}(\Lambda,A(-1))$$
is the subset of alternating forms $b''$, i.e., those forms for which $b''(\lambda,\lambda)=0$,
and we identify
$$\on{Quad}(\Lambda,A_{2\on{-tors}})=\on{Quad}(\Lambda,A(-1)_{2\on{-tors}})$$
with the preimage of $\on{Alt}(\Lambda,A(-1))$ under the differential $d_2$. 

\medskip

Note that the differential $\wt{d}_1$ is given by multiplication by $2$. 

\sssec{}

Note also that have an inclusion
$$\Hom(\Lambda,A_{2\on{-tors}})=\on{ker}(\wt{d}_2)\subset \on{Quad}(\Lambda,A_{2\on{-tors}}),$$
corresponding to those quadratic forms, whose associated symmetric bilinear form vanishes. 

\sssec{}

The object fundamental for this subsection will be the push-out
\begin{equation} \label{e:key pushout}
\CR(\Lambda):=B^2(\Hom(\Lambda,A))\underset{B^2(\Hom(\Lambda,A_{2\on{-tors}}))}\sqcup
\wt\cD(\Lambda).
\end{equation} 

We have:
$$\pi_i(\CR(\Lambda))=
\begin{cases}
&\on{Quad}(\Lambda,A(-1)), \text{ for } i=0,\\
&0, \text{ for } i=1, \\
&\Hom(\Lambda,A), \text{ for } i=2.
\end{cases}$$

\sssec{} \label{sss:basis splits}

We claim that a choice of a basis $\lambda_1,...,\lambda_n$ of $\Lambda$ gives rise to a splitting
$$\CR(\Lambda) \simeq \on{Quad}(\Lambda,A(-1)) \times B^2(\Hom(\Lambda,A)).$$

Indeed, the choice of a basis gives rise to a left inverse
\begin{equation} \label{e:left in split}
\Hom(\Lambda,A) \underset{\Hom(\Lambda,A_{2\on{-tors}})}\sqcup \on{Quad}(\Lambda,A_{2\on{-tors}})\to \Hom(\Lambda,A)
\end{equation}
of the inclusion
$$\Hom(\Lambda,A)\hookrightarrow \Hom(\Lambda,A) \underset{\Hom(\Lambda,A_{2\on{-tors}})}\sqcup \on{Quad}(\Lambda,A_{2\on{-tors}}).$$

\medskip

Namely, the restriction of \eqref{e:left in split} to $\on{Quad}(\Lambda,A_{2\on{-tors}})$ sends a given $q''\in \on{Quad}(\Lambda,A_{2\on{-tors}})$ 
to the map $\Lambda\to A$ defined by $\lambda_i\mapsto q''(\lambda_i)$.

\ssec{Relation to braided monoidal categories}


\medskip

Let $\cD_{\on{top}}(\Lambda,A)$ (resp., $\CR_{\on{top}}(\Lambda,A)$) 
be a version of $\CR_{\on{top}}(\Lambda)$, in which we replace $A(-1)$ by $A$. 

\medskip

The key assertion in \cite[P.S.]{Del} is that the complex $\CR_{\on{top}}(\Lambda,A)$, viewed as a connective spectrum, is naturally isomorphic to
$$\Maps_{\on{Ptd}(\Spc)}(B^2(\Lambda),B^4(A)),$$
i.e., it classifies braided monoidal groupoids $\CC$ with
$\pi_0(\CC)\simeq \Lambda$ and $\pi_1(\CC)\simeq A$, 
see Remark \ref{r:Ek top}. 

\sssec{}

Let us construct the corresponding map
\begin{equation} \label{e:from bilin to categ}
\cD_{\on{top}}(\Lambda,A)\to \Maps_{\on{Ptd}(\Spc)}(B^2(\Lambda),B^4(A)).
\end{equation} 


For a bilinear form $b'$ we let $\CC_{b'}$ be $\Lambda\times B(A)$ as a \emph{monoidal} groupoid, with the braiding defined
by $b'$.  

\medskip

When we modify $b'$ by the coboundary of $b''\in \on{Bilin}(\Lambda,A)$, we let
$$\phi_{b''}:\CC_{b'}\simeq \CC_{b'+d_1(b'')}$$
be the identity functor at the level of plain groupoids, but with the monoidal structure given by $b''$:
$$
\CD
\phi_{b''}(\lambda_1) \otimes \phi_{b''}(\lambda_2) @>>> \phi_{b''}(\lambda_1+\lambda_2) \\
@V{=}VV @VV{=}V  \\
\lambda_1+\lambda_2   @>{b''(\lambda_1,\lambda_2)}>> \lambda_1+\lambda_2. 
\endCD
$$

When we modify $b''$ by the coboundary of $q''\in \on{Quad}(\Lambda,A)$, we construct an isomorphism
$$\phi_{b''}\simeq \phi_{b''+d_2(q'')}$$
by letting its value on $\lambda$ be equal to $q''(\lambda)$.

\sssec{}

Note that $\Maps_{\on{Ptd}(\Spc)}(B^2(\Lambda),B^4(A))$ is the same as
$$\Maps_{\on{Ptd}(\Spc)}(B(T_{\on{top}}),B^4(A)),$$
where $T_{\on{top}}$ is the topological torus corresponding to $\Lambda$. Hence, its homotopy groups are given by
$$\pi_2(\Maps_{\on{Ptd}(\Spc)}(B^2(\Lambda),B^4(A)))\simeq H^2_{\on{top}}(B(T_{\on{top}}),A)\simeq \Hom(\Lambda,A),$$ 
$$\pi_1(\Maps_{\on{Ptd}(\Spc)}(B^2(\Lambda),B^4(A)))\simeq H^3_{\on{top}}(B(T_{\on{top}}),A)=0,$$
$$\pi_0(\Maps_{\on{Ptd}(\Spc)}(B^2(\Lambda),B^4(A)))\simeq H^4_{\on{top}}(B(T_{\on{top}}),A)=\on{Quad}(\Lambda,A),$$
and the other homotopy groups vanish.

\medskip

At the level of homotopy groups, the map \eqref{e:from bilin to categ} induces the maps
$$\pi_2(\cD_{\on{top}}(\Lambda,A)) \simeq \Hom(\Lambda,A_{2\on{-tors}}) \to \Hom(\Lambda,A) \simeq \pi_2(\Maps_{\on{Ptd}(\Spc)}(B^2(\Lambda),B^4(A)))$$
$$\pi_1(\cD_{\on{top}}(\Lambda,A)) =0 = \pi_1(\Maps_{\on{Ptd}(\Spc)}(B^2(\Lambda),B^4(A)))$$
and
$$\pi_0(\cD_{\on{top}}(\Lambda,A)) \simeq \on{Quad}(\Lambda,A)\simeq \pi_0(\Maps_{\on{Ptd}(\Spc)}(B^2(\Lambda),B^4(A))).$$
 
Hence, the map \eqref{e:from bilin to categ} induces an \emph{isomorphism} 
\begin{equation} \label{e:from bilin to categ isom}
\CR_{\on{top}}(\Lambda,A)\simeq \Maps_{\on{Ptd}(\Spc)}(B^2(\Lambda),B^4(A)).
\end{equation} 

\sssec{}

Let $\lambda_1,...,\lambda_n$ be a basis of $\Lambda$, and recall (see \secref{sss:basis splits})
that in this case we have a canonically defined splitting 
$$\CR_{\on{top}}(\Lambda,A)\simeq \on{Quad}(\Lambda,A)\times B^2(\Hom(\Lambda,A)).$$

The resulting splitting
$$\Maps_{\on{Ptd}(\Spc)}(B^2(\Lambda),B^4(A)) \simeq \on{Quad}(\Lambda,A)\times B^2(\Hom(\Lambda,A))$$
is defined as follows: the corresponding map
$$\Maps_{\on{Ptd}(\Spc)}(B^2(\Lambda),B^4(A)) \to B^2(\Hom(\Lambda,A))\simeq (B^2(A))^{\times n}$$ is
given by
$$\Maps_{\on{Ptd}(\Spc)}(B^2(\Lambda),B^4(A)) \to \Maps_{\on{Ptd}(\Spc)}(\Lambda,B^2(A))  \overset{\lambda_i}\to B^2(A).$$

\sssec{Variant}

Let $\Gamma$ be a finitely generated abelian group. Write it as $\Gamma\simeq \Lambda/\Lambda'$, where $\Lambda$
is a lattice. Set
$$\CR_{\on{top}}(\Gamma,A):=\on{Fib}\left(\CR_{\on{top}}(\Lambda,A)\to \CR_{\on{top}}(\Lambda',A)\right) \underset{\on{Quad}(\Lambda,A)}\times \on{Quad}(\Gamma,A).$$

It is easy to see that $\CR_{\on{top}}(\Gamma,A)$ is canonically independent of the presentation of $\Gamma$ as a quotient. 

\medskip

It follows from \eqref{e:from bilin to categ isom} that we have a canonical idomorphism 
\begin{equation} \label{e:from bilin to categ isom tors}
\CR_{\on{top}}(\Gamma,A) \simeq \Maps_{\on{Ptd}(\Spc)}(B^2(\Gamma),B^4(A)).
\end{equation} 

\ssec{Relation to the \'etale theory} \label{ss:etale computation}

\sssec{}

Let $T$ be the algebro-geometric torus over $k$ with cocharacter lattice $\Lambda$. 

\medskip

We are going to construct a canonical isomorphism
\begin{equation} \label{e:etale Q}
\Maps_{\on{Ptd}(\on{PreStk})}(B_{\on{et}}(T),B_{\on{et}}^4(A(1)))\simeq \CR(\Lambda).
\end{equation}

\sssec{}

Let $\Lambda'\subset \Lambda$
be a sub-lattice such that
$$\Gamma:=\Lambda/\Lambda'$$
is torsion (of order prime to  $\on{char}(k)$). Let $T'$ be the (algebro-geometric) torus corresponding to the lattice 
$\Lambda'$.

\medskip

Note that we have a short exact sequence
$$0\to \Gamma(1)\to T'\to T\to 0,$$
and corresponding to it the fiber sequence 
$$B_{\on{et}}( \Gamma(1))\to B_{\on{et}}(T')\to B_{\on{et}}(T)$$
in group-objects in $\on{Stk}$. 

\medskip

From here we obtain a map
$$B_{\on{et}}(T)\to B_{\on{et}}^2(\Gamma(1)),$$
and thus a map
$$\Maps_{\on{Ptd}(\on{PreStk})}(B_{\on{et}}^2(\Gamma(1)),B_{\on{et}}^4(A(1)))\to 
\Maps_{\on{Ptd}(\on{PreStk})}(B_{\on{et}}(T),B_{\on{et}}^4(A(1))).$$

\medskip

A computation of cohomology groups shows that the resulting map
$$\underset{\Lambda'\subset \Lambda}{\on{colim}}\,  
\Maps_{\on{Ptd}(\on{PreStk})}(B_{\on{et}}^2(\Gamma(1)),B_{\on{et}}^4(A(1)))\to 
\Maps_{\on{Ptd}(\on{PreStk})}(B_{\on{et}}(T),B_{\on{et}}^4(A(1)))$$
is an isomorphism. 

\sssec{}

Note that since $\Gamma(1)$ is discrete (as an object of algebraic geometry), the map
$$\Maps_{\on{Ptd}(\Spc)}(B^2(\Gamma(1)),B^4(A(1)))\to \Maps_{\on{Ptd}(\on{PreStk})}(B_{\on{et}}^2(\Gamma(1)),B_{\on{et}}^4(A(1)))$$
is an isomorphism. 

\medskip

Hence, in order to establish \eqref{e:etale Q}, it suffices to construct an isomorphism 
\begin{equation} \label{e:pullback from finite colim top}
\underset{\Lambda'\subset \Lambda}{\on{colim}}\,  \Maps_{\on{Ptd}(\Spc)}(B^2(\Gamma(1)),B^4(A(1))) \simeq \CR(\Lambda).
\end{equation} 

\sssec{}

By \eqref{e:from bilin to categ isom tors}, we can identify
$$\Maps_{\on{Ptd}(\Spc)}(B^2(\Gamma(1)),B^4(A(1))) \simeq \CR_{top}(\Gamma(1),A(1)).$$

Note that we have a canonical identification
$$\CR_{\on{top}}(\Gamma(1),A(1)) \simeq \CR(\Gamma,A(-1)).$$

Hence, the colimit in the left-hand side of \eqref{e:pullback from finite colim top} can be rewritten as 
$$\underset{\Lambda'\subset \Lambda}{\on{colim}}\, \CR(\Gamma,A),$$
which does indeed map isomorphically to $\CR(\Lambda,A)$.

\ssec{Relation to factorization gerbes}

We now claim that there exists a canonically defined map
\begin{equation} \label{e:Del to fact}
\CR(\Lambda)\to \Theta(\Lambda),
\end{equation} 
where $\Theta(\Lambda)$ is as in \secref{sss:Theta}. 

\sssec{}

We start by constructing a map
$$\cD(\Lambda)\to \Theta(\Lambda).$$

\medskip

Namely, given $b'\in \on{Bilin}(\Lambda,A(-1))$, we attach to it the quadratic form
$$q(\lambda):=b'(\lambda,\lambda),$$
and the system of $A$-gerbes
$$\CG^\lambda:=(\omega_X^{\otimes -1})^{q(\lambda)}.$$

\medskip

The isomorphisms $c_{\lambda_1,\lambda_2}$ are defined as follows: they are obtained by tensoring 
the tautological isomorphism
$$(\omega_X^{\otimes -1})^{q(\lambda_1+\lambda_2)}\simeq (\omega_X^{\otimes -1})^{(q(\lambda_1)}\otimes 
(\omega_X^{\otimes -1})^{q(\lambda_2)}\otimes (\omega_X^{\otimes -1})^{b(\lambda_1,\lambda_2)},$$
by the $A$-torsor $(-1)^{b'(\lambda_1,\lambda_2)}$. The datum of $h_{\lambda_1,\lambda_2}$ comes from the
identification
$$(-1)^{b'(\lambda_2,\lambda_1)}\simeq (-1)^{b'(\lambda_1,\lambda_2)}\otimes (-1)^{b(\lambda_1,\lambda_2)},$$
which in turn results from 
$$(-1)^{b'(\lambda_1,\lambda_2)}\simeq (-1)^{-b'(\lambda_1,\lambda_2)} \text{ and }
b(\lambda_1,\lambda_2)=b'(\lambda_1,\lambda_2)+b'(\lambda_2,\lambda_1).$$

\sssec{}

When we modify $b'$ by the coboundary of $b''\in \on{Bilin}(\Lambda,A(-1))$, we apply the system of automorphisms
of $\CG^\lambda$, given by $(-1)^{b''(\lambda,\lambda)}$.

\medskip

Finally, the result of the modification of $b''$ by 
$$q''\in \on{Quad}(\Lambda,A(-1))$$
acts as an automorphism of the identity map on $\CG^\lambda$ given by $q''(\lambda)$. 

\sssec{}

By construction, we have a map of fiber sequences
$$
\CD
B^2(\Hom(\Lambda,A_{2\on{-tors}})) @>>> \wt\cD(\Lambda,A(-1)) @>>> \on{Quad}(\Lambda,A(-1)) \\
@VVV @VVV @VV{\on{id}}V \\
\Theta^0(\Lambda)  @>>> \Theta(\Lambda)   @>>> \on{Quad}(\Lambda,A(-1)),
\endCD 
$$
where the left vertical arrow is the map
$$B^2(\Hom(\Lambda,A_{2\on{-tors}})) \to B^2(\Hom(\Lambda,A))
\to \Maps(X,B^2_{\on{et}}(\Hom(\Lambda,A)))\simeq \Theta^0(\Lambda).$$

\medskip

From here, we obtain the desired map \eqref{e:Del to fact}, which fits in the diagram of fiber sequences
$$
\CD
B^2(\Hom(\Lambda,A)) @>>> \CR(\Lambda) @>>> \on{Quad}(\Lambda,A(-1)) \\
@VVV @VVV @VV{\on{id}}V \\
\Theta^0(\Lambda)  @>>> \Theta(\Lambda)   @>>> \on{Quad}(\Lambda,A(-1)).
\endCD 
$$

In particular, we obtain that the induced map
\begin{equation} \label{e:Del to fact equiv}
\Maps(X,B^2_{\on{et}}(\Hom(\Lambda,A))) \underset{B^2(\Hom(\Lambda,A))}\sqcup\, \CR(\Lambda)\to \Theta(\Lambda)
\end{equation} 
is an equivalence. 

\sssec{}

We claim

\begin{thm} \label{t:from bilin to fact}
The following diagram commutes
\begin{equation} \label{e:from bilin to fact}
\CD
\Maps_{\on{Ptd}(\on{PreStk})}(B_{\on{et}}(T),B_{\on{et}}^4(A(1))) @>{\rm{pullback}}>> \Maps_{\on{Ptd}(\on{PreStk})}(B_{\on{et}}(T)\times X,B_{\on{et}}^4(A(1)))  \\
@V{\sim}V{\text{\eqref{e:etale Q}}}V    @V{\sim}V{\text{\eqref{e:from BG infty}}}V  \\
\CR(\Lambda)  & & \on{FactGe}_A(\Gr_T) \\
@VV{\text{\eqref{e:Del to fact}}}V  @V{\sim}V{\text{\eqref{e:theta}}}V  \\
\Theta(\Lambda) @>{\on{id}}>> \Theta(\Lambda).
\endCD
\end{equation} 
\end{thm} 

The proof of this theorem will be given in the next subsection. 

\sssec{} \label{sss:neutral case}

Before we launch the proof of \thmref{t:from bilin to fact}, let us note that the pre-composition of the diagram \eqref{e:from bilin to fact} with the map
$$B^2(\Hom(\Lambda,A)) \to \Maps_{\on{Ptd}(\on{PreStk})}(B_{\on{et}}(T),B_{\on{et}}^4(A(1)))$$
commutes, and so does the post-composition with the map
$$\Theta(\Lambda)\to \on{Quad}(\Lambda,A(-1)).$$

In other words, the two circuits of the diagrams give rise to canonically isomorphic maps 
$$B^2(\Hom(\Lambda,A))\to \Theta^0(\Lambda).$$

\sssec{}

Consider now the full subspace of $\Maps_{\on{Ptd}(\on{PreStk})}(B_{\on{et}}(T),B_{\on{et}}^4(A(1)))$ equal to
\begin{equation} \label{e:mult subspace}
\Maps_{\on{Ptd}(\on{PreStk})}(B_{\on{et}}(T),B_{\on{et}}^4(A(1))) \underset{\on{Quad}(\Lambda,A(-1))}\times \Hom(\Lambda,A_{2\on{-tors}}).
\end{equation}

Recall (see Remark \ref{r:Ek top}) that the forgetful functor
\begin{multline} \label{e:forget Einfty}
\Maps_{\BE_\infty(\Spc)}(\Lambda,B^2(A)) \simeq \Maps_{\BE_\infty(\Spc)}(\Lambda(1),B^2(A(1))) \simeq \\
\simeq \Maps_{\BE_\infty(\Spc)}(B^2(\Lambda(1)),B^4(A(1)))
\to \Maps_{\on{Ptd}(\Spc)}(B^2(\Lambda(1)),B^4(A(1)))
\end{multline}
is fully faithful with essential image equal to \eqref{e:mult subspace}.

\medskip

\medskip

Let us note that we have three(!) {\it a priori} different ways of mapping the space $\Maps_{\BE_\infty(\Spc)}(\Lambda,B^2(A))$ to $\Theta(\Lambda)$:

\medskip

\noindent(i) The first way is the composition of \eqref{e:forget Einfty} and clockwise circuit in \eqref{e:from bilin to fact}. 

\medskip

\noindent(ii) The second way is the composition of \eqref{e:forget Einfty} and counter-clockwise circuit in \eqref{e:from bilin to fact}. 

\medskip

\noindent(iii) The third way comes from \eqref{e:Einfty to Theta}. In more detail, a point of $\Maps_{\BE_\infty(\Spc)}(\Lambda,B^2(A))$
can be seen as a system of assignments
$$\lambda\in \Lambda \rightsquigarrow \CG^\lambda\in B^2(A)\simeq \on{Ge}_A(\on{pt}),$$
endowed with a associativity and commutativity constraints, and we can pull it back along $X\to \on{pt}$ to form an
object of $\Theta(\Lambda)$ as in \secref{sss:expl mult}. 

\medskip

The assertion of \thmref{t:from bilin to fact} implies, in particular, that (i) and (ii) are canonically isomorphic.  We claim that (ii) and (iii)
differ by a map 
\begin{equation} \label{e:mult obstr}
\Hom(\Lambda,A_{2\on{-tors}})\to \Theta^0(\Lambda)
\end{equation}
that can be described explicitly as follows.

\medskip

The map \eqref{e:mult obstr} is obtained as a tensor product of the following two maps:

\medskip

One is the map that sends $q\in \Hom(\Lambda,A_{2\on{-tors}})$ to the object of $\Theta^0(\Lambda)$ given by
\begin{equation} \label{e:omega gerbe}
\lambda \mapsto (\omega_X^{\otimes -1})^{q(\lambda)}.
\end{equation} 

The second map factors as
$$\Hom(\Lambda,A_{2\on{-tors}})\to B^2(\Hom(\Lambda,A))\to \Theta^0(\Lambda),$$
where the first map is described as follows:

\medskip

For a given 
$$q\in \Hom(\Lambda,A_{2\on{-tors}}),$$ a choice of 
$$\wt{q}\in \Hom(\Lambda,A),$$ 
with $2\wt{q}=q$
trivializes the resulting point of $B^2(\Hom(\Lambda,A))$. Two choices for $\wt{q}$, which differ
by 
$$\wt{\wt{q}}\in \Hom(\Lambda,A_{2\on{-tors}}),$$ 
result in the change of trivialization by the map
$$\Lambda\to B(\Hom(\Lambda,A)), \quad \lambda\mapsto (-1)^{\wt{\wt{q}}(\lambda)}.$$

\ssec{Proof of \thmref{t:from bilin to fact}}

\sssec{}

By \secref{sss:neutral case}, the obstruction to the commutativity of \eqref{e:from bilin to fact} is a map
\begin{equation} \label{e:obstruction}
\on{Quad}(\Lambda,A(-1))\to \Theta^0(\Lambda).
\end{equation} 

\medskip

We wish to show that this map vanishes. 

\sssec{}

Consider the canonical maps
$$\on{Cone}\left(\on{Bilin}(\Lambda,A(-1))\overset{d_1}\to \on{Bilin}(\Lambda,A(-1))\right) \to  \cD(\Lambda)\to  \on{Quad}(\Lambda,A(-1)).$$
We claim that their composition with \eqref{e:obstruction} does vanish.

\medskip

Indeed, this calculation has been performed in \secref{ss:matching}: we factor the given bilinear form $b'$ via a finite quotient
$$\Lambda\to \Gamma,$$
and we take $A_1=A_2=\Gamma(1)$. 

\medskip

Thus, mapping
$$\on{Cone}\left(\on{Alt}(\Lambda,A(-1))\overset{d_1}\to \on{Bilin}(\Lambda,A(-1))\right)\to
\on{Cone}\left(\on{Bilin}(\Lambda,A(-1))\overset{d_1}\to \on{Bilin}(\Lambda,A(-1))\right),$$ 
we obtain that the obstruction to the commutativity of \eqref{e:from bilin to fact} is a map 
$$\on{ker}(d_1)\to \Hom(\Lambda,A),$$
where $\Hom(\Lambda,A)$ is 
the group of automorphisms of the identity map of the unit object of $\Theta^0(\Lambda)$. 

\medskip

We identify 
$$\on{ker}(d_1)\simeq \on{Alt}(\Lambda,A_{2\on{-tors}}).$$

\medskip

Thus, we have refined the obstruction to a map of abelian groups
\begin{equation} \label{e:obstruction refined}
\on{Alt}(\Lambda,A_{2\on{-tors}})\to \Hom(\Lambda,A).
\end{equation}

\medskip

The map \eqref{e:obstruction refined} automatically factors as 
\begin{equation} \label{e:obstruction refined refined}
\on{Alt}(\Lambda/2\Lambda,A_{2\on{-tors}})\to \Hom(\Lambda/2\Lambda,A_{2\on{-tors}}).
\end{equation}

\sssec{}

By functoriality with respect to $A$, we can assume that $A_{2\on{-tors}}\simeq \BZ/2\BZ$. Hence, 
we can interpret \eqref{e:obstruction refined refined} as a map
\begin{equation} \label{e:obstruction refined refined ref}
\Lambda^2(V)\to V
\end{equation}
for a finite-dimensional vector space $V$ over $\BZ/2\BZ$ (our $V=(\Lambda/2\Lambda)^\vee$), functorial in $V$. 

\medskip

However, we claim any map as in \eqref{e:obstruction refined refined ref}, functorial in $V$, is necessarily zero. 
Indeed, $\Lambda^2(V)$ is an irreducible representation of $GL(V)$, non-isomorphic to $V$. 

\qed[\thmref{t:from bilin to fact}]

%

\ssec{The reductive case} \label{ss:expl K}

Let $G$ be a reductive group with Cartan group $T$. 

\medskip

Let $\CR(\Lambda)_G$ denote the following subcomplex of $\CR(\Lambda)$:

\medskip

\begin{itemize}

\item In the $0$th term we replace $\on{Bilin}(\Lambda,A(-1))$ with
$$\on{Bilin}(\Lambda,A(-1))\underset{\on{Quad}(\Lambda,A9-1))}\times \on{Quad}(\Lambda,A(-1))^W_{\on{restr}},$$

\item The $(-1)$st term remains the same;

\item In the $(-2)$nd term we replace 
$$\Hom(\Lambda,A)\underset{\Hom(\Lambda,A_{2\on{-tors}})}\sqcup\, \on{Quad}(\Lambda,A_{2\on{-tors}})$$
with the kernel of the map
\begin{equation} \label{e:ev on roots again}
\Hom(\Lambda,A)\underset{\Hom(\Lambda,A_{2\on{-tors}})}\sqcup\, \on{Quad}(\Lambda,A_{2\on{-tors}})\to
\underset{i\in I}\Pi\, A,
\end{equation} 
where:

\begin{itemize}

\item $I$ is the set of simple roots;

\item The restriction of the map \eqref{e:ev on roots again} to $\Hom(\Lambda,A)$ is given by evaluation on the
simple roots;

\item The restriction of the map \eqref{e:ev on roots again} to $\on{Quad}(\Lambda,A_{2\on{-tors}})$ is given by evaluation on the
simple roots followed by $A_{2\on{-tors}}\hookrightarrow A$.

\end{itemize}

\end{itemize} 

\sssec{}

The $0$-th cohomology of $\CR(\Lambda)_G$ identifies with $\on{Quad}(\Lambda,A(-1))^W_{\on{restr}}$, the cohomology in degree $(-1)$
vanishes, and the cohomology in degree $(-2)$ identifies with
$$\Hom(\pi_{1,\on{alg}}(G),A).$$

\sssec{} \label{sss:expl K}

Our current goal is to prove the following assertion:

\begin{thm} \label{t:expl K}
There exists a canonical equivalence
\begin{equation} \label{e:etale Q G}
\Maps_{\on{Ptd}(\on{PreStk})}(B_{\on{et}}(G),B_{\on{et}}^4(A(1)))\simeq \CR(\Lambda)_G.
\end{equation}
\end{thm} 

Note that the assertion of this theorem identifies explicitly the complex $\CQ$ from \eqref{e:K}. The rest of this subsection is devoted
to the proof of the theorem.

\sssec{}

Consider the diagram
\begin{equation} \label{e:G T diagram}
\CD
\Maps_{\on{Ptd}(\on{PreStk})}(B_{\on{et}}(G),B_{\on{et}}^4(A(1))) & & \CR(\Lambda)_G \\
@VVV @VVV \\
\Maps_{\on{Ptd}(\on{PreStk})}(B_{\on{et}}(T),B_{\on{et}}^4(A(1))) @>{\sim}>> \CR(\Lambda). 
\endCD
\end{equation}

By \secref{ss:containment one}, we obtain that at the level of $\pi_0$, the image
of the left vertical arrow indeed hits 
$$\on{Quad}(\Lambda,A(-1))^W_{\on{restr}}\subset \on{Quad}(\Lambda,A(-1)).$$

\medskip

Hence, in order to construct the upper horizontal arrow in \eqref{e:G T diagram}, it suffices to show that
for each simple root, the map

\begin{multline} \label{e:root eval}
\Maps_{\on{Ptd}(\on{PreStk})}(B_{\on{et}}(G),B_{\on{et}}^4(A(1))) \to \Maps_{\on{Ptd}(\on{PreStk})}(B_{\on{et}}(T),B_{\on{et}}^4(A(1))) \to \\
\to \CR(\Lambda) \to B^2(A)
\end{multline} 
admits a canonical null-homotopy, where the last arrow is
\begin{multline*}
\CR(\Lambda,A) = B^2(\Hom(\Lambda,A))\underset{B^2(\Hom(\Lambda,A_{2\on{-tors}}))}\sqcup
\wt\cD(\Lambda) \to \\
\to B^2(\Hom(\Lambda,A))\underset{B^2(\Hom(\Lambda,A_{2\on{-tors}}))}\sqcup
B^2(\on{Quad}(\Lambda,A_{2\on{-tors}}))\to B^2(A),
\end{multline*}
with the first arrow given by the projection on the term in degree $-2$
$$\wt\cD(\Lambda)\to B^2(\on{Quad}(\Lambda,A_{2\on{-tors}})),$$
and the last arrow is given by evaluation on the simple root $\alpha_i$.

\medskip

Restricting to the copy of $SL_2$ mapping to $G$ corresponding to $i$, we reduce the last assertion to the case
when $G=SL_2$.

\sssec{}

For $G=SL_2$ and $\Lambda=\BZ$ we note that the projection
$$\on{Bilin}(\BZ,A(-1))\to \on{Quad}(\BZ,A(-1))\simeq A(-1)$$
is an isomorphism, while $\on{Alt}(\BZ,A(-1))=0$.

\medskip

Hence, the complex $\wt\cD(\Lambda)$ equals, as a complex, to
$$B^2(A_{2\on{-tors}})\oplus A(-1),$$
so
\begin{equation} \label{e:R Z}
\CR(\Lambda) \simeq B^2(A)\oplus A(-1).
\end{equation}

In terms of this identification, the map
$$\Maps_{\on{Ptd}(\on{PreStk})}(B_{\on{et}}(T),B_{\on{et}}^4(A(-1))) \to \CR(\Lambda) \overset{\text{evaluation on root}}\longrightarrow B^2(A)$$
that appears in \eqref{e:root eval} is the projection 
$$B^2(A)\oplus A(-1)\to B^2(A).$$

\medskip

Thus, we need to show that the map 
$$\Maps_{\on{Ptd}(\on{PreStk})}(B_{\on{et}}(G),B_{\on{et}}^4(A(1))) \to \Maps_{\on{Ptd}(\on{PreStk})}(B_{\on{et}}(T),B_{\on{et}}^4(A(1))) \simeq  \CR(\Lambda)
\simeq B^2(A)\oplus A(-1).$$
factors via
$$\Maps_{\on{Ptd}(\on{PreStk})}(B_{\on{et}}(G),B_{\on{et}}^4(A(1))) \to A(-1)\to B^2(A)\oplus A(-1).$$

\sssec{}

For an integer $n$, we have
$$\Maps_{\on{Ptd}(\on{PreStk})}(B_{\on{et}}(T),B_{\on{et}}^4(A_{n\on{-tors}}(1)))\simeq
\on{Fib}(\CR(\Lambda) \overset{n\cdot}\to \CR(\Lambda)) \simeq B^2(A_{n\on{-tors}})\oplus A_{n\on{-tors}}(-1).$$

It suffices to show that for all $n$, the map
\begin{multline*} 
\Maps_{\on{Ptd}(\on{PreStk})}(B_{\on{et}}(G),B_{\on{et}}^4(A_{n\on{-tors}}(1))) \to 
\Maps_{\on{Ptd}(\on{PreStk})}(B_{\on{et}}(T),B_{\on{et}}^4(A_{n\on{-tors}}(1))) \to \\
\to B^2(A_{n\on{-tors}})\oplus A_{n\on{-tors}}(-1)
\end{multline*}
factors via 
$$\Maps_{\on{Ptd}(\on{PreStk})}(B_{\on{et}}(G),B_{\on{et}}^4(A_{n\on{-tors}}(1))) \to A_{n\on{-tors}}(-1)\to B^2(A_{n\on{-tors}})\oplus A_{n\on{-tors}}(-1).$$

With no restriction of generality we can replace $A_{n\on{-tors}}$ by $\BZ/n\BZ$. Passing to the inverse limit over $n$, its suffices to show that the map

\begin{equation} \label{e:R Z 1 inv}
\Maps_{\on{Ptd}(\on{PreStk})}(B_{\on{et}}(G),B_{\on{et}}^4(\wh\BZ(1))) \to 
\Maps_{\on{Ptd}(\on{PreStk})}(B_{\on{et}}(T),B_{\on{et}}^4(\wh\BZ(1))) 
\simeq B^2(\wh\BZ)\oplus \wh\BZ(-1)
\end{equation}
factors via
$$\Maps_{\on{Ptd}(\on{PreStk})}(B_{\on{et}}(G),B_{\on{et}}^4(\wh\BZ(1))) \to \wh\BZ(-1)  \to B^2(\wh\BZ)\oplus \wh\BZ(-1).$$

\sssec{}

Note now that the map
$$\Maps_{\on{Ptd}(\on{PreStk})}(B_{\on{et}}(G),B_{\on{et}}^4(\wh\BZ(1))) \to 
\Maps_{\on{Ptd}(\on{PreStk})}(B_{\on{et}}(T),B_{\on{et}}^4(\wh\BZ(1)))$$
factors via
\begin{multline*}  
\Maps_{\on{Ptd}(\on{PreStk})}(B_{\on{et}}(G),B_{\on{et}}^4(\wh\BZ(1))) \to 
\Maps_{\on{Ptd}(\on{PreStk})}(B_{\on{et}}(T),B_{\on{et}}^4(\wh\BZ(1)))^W\to \\
\to \Maps_{\on{Ptd}(\on{PreStk})}(B_{\on{et}}(T),B_{\on{et}}^4(\wh\BZ(1))),
\end{multline*}
where $W\simeq S_2$ is the Weyl group, which acts by inversion on $T$.

\medskip

In terms of the identification 
$$\Maps_{\on{Ptd}(\on{PreStk})}(B_{\on{et}}(T),B_{\on{et}}^4(\wh\BZ(1)))\simeq  B^2(\wh\BZ)\oplus \wh\BZ(-1),$$
the Weyl group acts as identity on $\wh\BZ(-1)$ and as inversion on $B^2(\wh\BZ)$. 

\medskip

Now, the required factorization of \eqref{e:R Z 1 inv} follows from the fact that the forgetful map
$$(B^2(\wh\BZ))^W\to B^2(\wh\BZ)$$
is zero, which in turn follows from the corresponding fact for the map
$$(\wh\BZ)^W\to \wh\BZ,$$
where $W=S_2$ acts on $\wh\BZ$ by inversion. 

\sssec{}

Thus, we have constructed a map
\begin{equation} \label{e:etale Q G bis}
\Maps_{\on{Ptd}(\on{PreStk})}(B_{\on{et}}(G),B_{\on{et}}^4(A(1))) \to \CR(\Lambda)_G.
\end{equation}

By construction, the maps on the homotopy groups are
$$\pi_2\left(\Maps_{\on{Ptd}(\on{PreStk})}(B_{\on{et}}(G),B_{\on{et}}^4(A(1)))\right) \simeq
\Hom(\pi_{1,\on{alg}}(G),A)\simeq \pi_2(\CR(\Lambda)_G),$$
$$\pi_1\left(\Maps_{\on{Ptd}(\on{PreStk})}(B_{\on{et}}(G),B_{\on{et}}^4(A(1))) \right) =0=
\pi_1(\CR(\Lambda)_G)$$
and
$$\pi_0\left(\Maps_{\on{Ptd}(\on{PreStk})}(B_{\on{et}}(G),B_{\on{et}}^4(A(1))) \right) \simeq \on{Quad}(\Lambda,A(-1))^W_{\on{restr}}
\simeq 
\pi_0(\CR(\Lambda)_G).$$

Hence, \eqref{e:etale Q G bis} is an isomorphism. 

\qed 

\ssec{Factorization gerbes in the reductive case}

\sssec{}

Recall the space $\Theta(\Lambda)_G$, see \secref{sss:theta G new 1}.  It follows from the construction of
the map \eqref{e:Del to fact} that we have a canonically defined map
\begin{equation} \label{e:Del to fact G}
\CR(\Lambda)_G\to \Theta(\Lambda)_G
\end{equation}
that fits into the commutative diagram
\begin{equation} \label{e:Del to fact G diag}
\CD
\CR(\Lambda)_G @>>>  \Theta(\Lambda)_G \\
@VVV @VVV  \\
 \CR(\Lambda)_T @>>>  \Theta(\Lambda)_T
 \endCD
 \end{equation}
and the commutative diagram of fiber sequences:
$$
\CD
B^2(\Hom(\pi_{1,\on{alg}}(G),A)) @>>> \CR(\Lambda)_G @>>> \on{Quad}(\Lambda,A(-1))^W_{\on{restr}} \\
@VVV @VVV @VV{\on{id}}V \\
\Theta^0(\Lambda)_G  @>>> \Theta(\Lambda)_G   @>>> \on{Quad}(\Lambda,A(-1))^W_{\on{restr}}.
\endCD 
$$

In particular, we obtain that the induced map
$$\Maps(X,B^2_{\on{et}}(\Hom(\pi_{1,\on{alg}}(G),A))) \underset{B^2(\Hom(\pi_{1,\on{alg}}(G),A))}\sqcup\, \CR(\Lambda)_G\to \Theta(\Lambda)_G$$
is an equivalence.

\sssec{}

We claim: 

\begin{thm} 
The following diagram commutes
$$
\CD
\Maps_{\on{Ptd}(\on{PreStk})}(B_{\on{et}}(G),B_{\on{et}}^4(A(1))) @>{\rm{pullback}}>> \Maps_{\on{Ptd}(\on{PreStk})}(B_{\on{et}}(G)\times X,B_{\on{et}}^4(A(1)))  \\
@V{\sim}V{\rm{\thmref{t:expl K}}}V    @V{\sim}V{\text{\eqref{e:from BG infty}}}V  \\
\CR(\Lambda)_G  & & \on{FactGe}_A(\Gr_G) \\
@VV{\text{\eqref{e:Del to fact G}}}V  @V{\sim}V{\text{\eqref{e:theta enh}}}V  \\
\Theta(\Lambda)_G @>{\on{id}}>> \Theta(\Lambda)_G.
\endCD
$$
\end{thm} 

\begin{proof}

The composites of both circuits with the forgetful map $\Theta(\Lambda)_G\to \Theta(\Lambda)$ are isomorphic by
\thmref{t:from bilin to fact}. Hence, it suffices to show that for every simple coroot $\alpha_i$, the composites of 
both circuits with the evaluation map
$$\Theta(\Lambda)_G\mapsto \on{Ge}_A(X)$$
produce isomorphic results.

\medskip

This reduces the verification to the case of $G=SL_2$. However, in the latter case, $\Theta(\Lambda)_G$ is
\emph{discrete}, i.e., the map
$$\Theta(\Lambda)_G\to \on{Quad}(\Lambda,A(-1))^W_{\on{restr}}$$
is an isomorphism, and the assertion follows.

\end{proof} 

\section{Proof of \propref{p:parameterization} by global methods} \label{s:proof of class}

In this section we will supply a proof of \propref{p:parameterization}, which is different from both the 
original (but containing gaps) proof in \cite{Re} and the recent proof in \cite{Zhao}.

\ssec{The simply connected case} \label{ss:simply conn}

In this subsection we will assume that $G$ is simply-connected, and we will prove that \eqref{e:from BG infty} is an isomorphism 
in this case.

\sssec{}

Recall the notations from \secref{sss:global}. A variant of the construction in Sects. \ref{sss:construct gerbe}-\ref{sss:construct gerbe inter next}
defines a map
\begin{equation} \label{e:upper bound}
\Maps_{\on{Ptd}(\on{PreStk})}(B_{\on{et}}(G)\times X,B^4_{\on{et}}(A(1))\to \on{Ge}_A(\Bun_G(\ol{X};D);\on{unit}),
\end{equation} 
(the notation ``$-;\on{unit}$" means ``gerbes trivialized at the unit"), 
so that we have a commutative diagram
\begin{equation} \label{e:global picture}
\CD
\Maps_{\on{Ptd}(\on{PreStk})}(B_{\on{et}}(G)\times X,B^4_{\on{et}}(A(1)) @>{\text{\eqref{e:from BG infty}}}>>   \on{FactGe}_{A}(\Gr_G)  \\
@VVV @VVV \\
\on{Ge}_A(\Bun_G(\ol{X};D);\on{unit}) @>{\sim}>> \on{Ge}_{A}(\Gr_G;\on{unit}). 
\endCD
\end{equation} 

We will show that all arrows in the diagram \eqref{e:global picture} are isomorphisms.

\sssec{}

First, recall that for $G$ simply connected, the pullback map
$$\Maps_{\on{Ptd}(\on{PreStk})}(B_{\on{et}}(G),B^4_{\on{et}}(A(1)) \to
\Maps_{\on{Ptd}(\on{PreStk})}(B_{\on{et}}(G)\times X,B^4_{\on{et}}(A(1))$$
is an isomorphism, and both spaces are discrete with $\pi_0$ given by $H^4_{\on{et}}(B_{\on{et}}(G),A(1))$. 

\sssec{}

Next, we claim that $\on{Ge}_A(\Bun_G(\ol{X};D);\on{unit})$ is also discrete that the left vertical arrow in \eqref{e:global picture}
induces an isomorphism on $\pi_0$. 

\medskip

Indeed, it follows from the Atiyah-Bott formula for the cohomology of $\Bun_G(\ol{X};D),A)$ (see \cite[Theorem 5.4.5]{GL1} or \cite[Theorem 19.1.4]{Ga7}) that
$$H^0_{\on{et}}(\Bun_G(\ol{X};D),A)\simeq A, \,\, H^1_{\on{et}}(\Bun_G(\ol{X};D),A)=0,$$
while the map 
$$H^0_{\on{et}}\left(X,H^4_{\on{et}}(B_{\on{et}}(G),A(1))\right) \to H^2_{\on{et}}(\Bun_G(\ol{X};D),A)$$
induced by \eqref{e:upper bound}, is an isomorphism. 

\sssec{}

Thus, we obtain that the forgetful map
\begin{equation} \label{e:forget factor}
\on{FactGe}_{A}(\Gr_G) \to \on{Ge}_{A}(\Gr_G;\on{unit})
\end{equation} 
realizes $\on{Ge}_{A}(\Gr_G;\on{unit})$ as a retract of$ \on{FactGe}_{A}(\Gr_G)$. Hence, it remains to prove that
the map \eqref{e:forget factor} is fully faithful. 

\medskip

I.e., we have to show that if a gerbe $\CG$ on $\Gr_G$ admits a factorization structure, then this structure is unique. 
However, this follows from the fact that the projection
$$\Gr_G\to \Ran$$
is ind-proper with connected and simply-connected fibers.

\ssec{Retraction for an arbitrary group}

In this subsection we let $G$ be an arbitrary reductive group. 

\sssec{}

Consider the forgetful map
$$\on{FactGe}_{A}(\Gr_G)\to \on{FactGe}_{A}(\Gr_T)\to \Theta(\Lambda).$$

We claim that this map canonically factors via a map
$$\on{FactGe}_{A}(\Gr_G)\to \Theta(\Lambda)_G\to \Theta(\Lambda).$$

First, we claim that the quadratic form $q$ associated to an object of $\on{FactGe}_{A}(\Gr_G)$ belongs
to $\on{Quad}(\Lambda,A(-1))^W_{\on{restr}}$. This follows as in \secref{ss:ss rank 1}. 

\medskip

To construct the identifications 
$$\CG^{\alpha_i} \simeq (\omega_X^{-1})^{q(\alpha_i)},$$
it is enough to consider the case of $G=SL_2$, and the assertion follows from the simply-connected case
established above. 

\sssec{}

Consider the composition
$$\Maps_{\on{Ptd}(\on{PreStk})}(B_{\on{et}}(G) \times X,B^4_{\on{et}}(A(1)))\to \on{FactGe}_{A}(\Gr_G)\to \Theta(\Lambda)_G.$$

It induces an isomorphism on homotopy groups, and hence is an equivalence. Hence, we obtain that the map
\begin{equation} \label{e:fact ge to theta}
\on{FactGe}_{A}(\Gr_G)\to \Theta(\Lambda)_G
\end{equation} 
realizes $\Theta(\Lambda)_G$ as a retract of $\on{FactGe}_{A}(\Gr_G)$. 

\sssec{}

Hence, in order to show that \eqref{e:from BG infty} is an isomorphism,  
it remains to show that the map \eqref{e:fact ge to theta} is fully faithful.

\ssec{The case of a simply-connected derived group} \label{ss:der simpl conn}

In this subsection we will assume that the derived group $G'$ of $G$ is simply-connected. 

\sssec{}

Let $T_0$ be the quotient of $G$ by its derived group $G'$. By the assumption on $G$, the map
$$\pi_{1,\on{alg}}(G)\to \pi_{1,\on{alg}}(T_0)$$
is an isomorphism.

\medskip

We have a commutative diagram
$$
\CD
\on{FactGe}_{A}(\Gr_{G'})  @>{\sim}>>  \Theta(\Lambda)_{G'} \\
@AAA @AAA \\
\on{FactGe}_{A}(\Gr_G) @>>> \Theta(\Lambda)_G \\
@AAA  @AAA \\
\on{FactGe}_{A}(\Gr_{T_0})  @>{\sim}>>  \Theta(\Lambda)_{T_0}
\endCD
$$

Hence, in order to prove that \eqref{e:from BG infty} is an isomorphism, 
it suffices to show that the pullback map
$$\on{FactGe}_A(\Gr_{T_0})\to \on{Fib}(\on{FactGe}_A(\Gr_G)\to \on{FactGe}_A(\Gr_{G'})$$
is an isomorphism.

\sssec{}

Note that the map
$$\Gr_G\to \Gr_{T_0}$$
is ind-proper with fibers that are connected and simply-connected. Hence, the map
$$\on{FactGe}_{A}(\Gr_{T_0}) \to \on{Ge}_{A}(\Gr_{T_0}) \underset{\on{Ge}_{A}(\Gr_G)}\times \on{FactGe}_{A}(\Gr_G)$$
is an isomorphism.

\medskip

We claim that 
\begin{equation} \label{e:fib Gr}
\on{Ge}_{A}(\Gr_{T_0}) \to \on{Ge}_{A}(\Gr_G)\to \on{Ge}_{A}(\Gr_{G'})
\end{equation}
is a fiber sequence.

\medskip

Once we prove this, we can identify 
$$\on{Ge}_{A}(\Gr_{T_0}) \underset{\on{Ge}_{A}(\Gr_G)}\times \on{FactGe}_{A}(\Gr_G)
\simeq \on{Fib}(\on{FactGe}_{A}(\Gr_G)\to \on{Ge}_{A}(\Gr_{G'})),$$
and since
$$\on{FactGe}_{A}(\Gr_{G'})\to \on{Ge}_{A}(\Gr_{G'})$$
is an equivalence (see \secref{ss:simply conn}), further with
$$\on{Fib}(\on{FactGe}_{A}(\Gr_G)\to \on{FactGe}_{A}(\Gr_{G'})),$$
as desired.

\sssec{}

To prove that \eqref{e:fib Gr} is a fiber sequence, it is enough to prove that
$$\on{Ge}_{A}(\Bun_{T_0}(\ol{X};D))\to \on{Ge}_{A}(\Bun_G(\ol{X};D))\to \on{Ge}_{A}(\Bun_{G'}(\ol{X};D))$$
is a fiber sequence.

\medskip

The latter follows from the fact that 
$$\Bun_G(\ol{X};D)\to \Bun_{T_0}(\ol{X};D)$$
is a fibration, locally trivial in the smooth topology with typical fiber $\Bun_{G'}(\ol{X};D)$,
which is simply-connected. 

\ssec{The general case}

We now let $G$ be a general reductive group. 

\sssec{}

Let 
$$1\to T\to \wt{G}\to G\to 1$$
be as in \eqref{e:cover group}. Let $\wt{G}^\bullet$ be the \v{C}ech nerve of the map $\wt{G}\to G$. 

\medskip

Note that the map
$$B_{\on{et}}(\wt{G})\to B_{\on{et}}(G)$$
is an \'etale surjection, and 
$$\Gr_{\wt{G}}\to \Gr_G$$
becomes a surjection after sheafification with respect to the topology generated by finite surjective maps. 

\medskip

Hence the maps
$$\Maps_{\on{Ptd}(\on{PreStk})}(B_{\on{et}}(G) \times X,B^4_{\on{et}}(A(1)))\to \on{Tot}
\left(\Maps_{\on{Ptd}(\on{PreStk})}(B_{\on{et}}(\wt{G}^\bullet) \times X,B^4_{\on{et}}(A(1)))\right)$$
and 
$$\on{FactGe}_{A}(\Gr_G)\to \on{Tot}\left(\on{FactGe}_{A}(\Gr_{\wt{G}^\bullet})\right)$$
are both isomorphisms.

\medskip

Now, the map
$$\Maps_{\on{Ptd}(\on{PreStk})}(B_{\on{et}}(\wt{G}^\bullet) \times X,B^4_{\on{et}}(A(1)))\to
\on{FactGe}_{A}(\Gr_{\wt{G}^\bullet})$$
is a term-wise isomorphism by \secref{ss:der simpl conn}. 

\medskip

Hence, \eqref{e:from BG infty} is also an isomorphism, as desired. 

\section{Twisting of factorization categories by gerbes}  \label{s:twist by sign}

\ssec{The context}

Let $\CC$ be a symmetric monoidal category, and let $A$ be a finite abelian group that 
acts by automorphisms of the identity functor on $\CC$ (viewed as a symmetric monoidal functor).
We will assume that the orders of elements in $A$ are prime to  $\on{char}(k)$. 

\medskip

Up to passing to a colimit, we can write 
$$A=\Hom(\Gamma,E^{\times,\on{tors}}),$$ 
where $\Gamma$ is the dual finite abelian group. 

\medskip

By \secref{sss:splitting}, we can think of a pair $(\CG_A \in \on{Ge}_A(X),\epsilon\in A_{2\on{-tors}})$ 
as a multiplicative factorization gerbe $\CG$ on  $\Gr_{\Gamma\otimes \BG_m}$ with respect to
$E^{\times,\on{tors}}$. 
(Recall also that the multiplicative structure on $\CG$ automatically lifts to a commutative one, see 
Remark \ref{r:com mult arb}.)

\medskip

We will show how to perform a twist of $\on{Fact}(\CC)$ by means of $\CG$ and obtain a new sheaf of 
symmetric monoidal
categories over $\Ran$, denoted $\on{Fact}(\CC)_\CG$, equipped with a factorization structure. 

\medskip

This construction will contain both twisting constructions, mentioned in 
Sects. \ref{sss:tw sym mon gerbe} and \ref{sss:tw sym mon sign}, respectively.

\ssec{Two symmetric monoidal structures on $\Vect^\Gamma$}

Consider the category $\Vect^\Gamma$ of $\Gamma$-graded $E$-vector spaces. 
For $\gamma\in \Gamma$, let $E^\gamma$ denote the corresponding 1-dimensional object in $\Vect^\Gamma$. 

\medskip

The category $\Vect^\Gamma$ has two symmetric monoidal structures. One, denoted by $\otimes$, is given
by convolution along $\gamma$:
$$E^{\gamma_1}\otimes E^{\gamma_2}=E^{\gamma_1+\gamma_2},$$

\medskip

The other one, denoted $\ast$, is given by component-wise tensor product. Note that this symmetric 
is non-unital, unless $\Gamma$ is finite. 

\medskip

Note that these two symmetric monoidal structures are \emph{lax-compatible} in the sense that there
exists a natural transformation
$$(V_1\star W_1) \otimes (V_2\star W_2)\to (V_1\otimes V_2)\star (W_1\otimes W_2)$$
satisfying a homotopy-coherent system of compatibilities.

\ssec{Interpreting the action}

Let us be given an action of $A$ on $\CC$. The datum of such an action is equivalent to that of an 
action on $\CC$ of $\Vect^\Gamma$, equipped with the $\ast$ monoidal structure. 

\medskip

Explicitly, this means that 
every object $c\in \CC$ can be canonically written
as $$c\simeq \underset{\gamma\in \Gamma}\oplus\, c^\gamma,$$
in a way compatible with the symmetric monoidal structure on $\CC$. In the above notations, 
$$E^\gamma\star c=c^\gamma.$$

\medskip

The $\star$-action of $\Vect^\Gamma$ on $\CC$ is lax-compatible with the symmetric
monoidal structure on $\CC$ in the sense that we have a natural transformation
$$(V_1\star c_1)\otimes (V_2\star c_2)\to (V_1\otimes V_2) \star (c_1\otimes c_2),$$
satisfying a homotopy-coherent system of compatibilities.

\ssec{Creating factorization categories}

The assignment
\begin{equation} \label{e:factor functor}
\CC\mapsto \on{Fact}(\CC)
\end{equation} 
is functorial with respect to lax symmetric monoidal functors.

\medskip

Hence, we obtain that $\on{Fact}(\Vect^\Gamma_{\otimes})$ (i.e., the sheaf of symmetric monoidal categories
obtained from $\Vect^\Gamma$ in the $\otimes$ symmetric monoidal structure) acquires another symmetric
monoidal structure given by $\ast$, which is lax-compatible with one given by $\otimes$.

\medskip

Similarly, the $\ast$-action of $\Vect^\Gamma$ on $\CC$ implies that $\on{Fact}(\Vect^\Gamma_{\otimes})$, viewed as
a sheaf of monoidal categories over $\Ran$ with respect to $\ast$, acts on $\on{Fact}(\CC)$. 
This action is lax-compatible with $\otimes$-symmetric monoidal structure on $\on{Fact}(\Vect^\Gamma_{\otimes})$,
and the given one on $\on{Fact}(\CC)$. 

\ssec{Relation to the affine Grassmannian}

We note now that there exists a canonical equivalence of sheaves of categories over $\Ran$
\begin{equation} \label{e:A vs Gr}
\on{Fact}(\Vect^\Gamma_{\otimes})\simeq \Shv(\Gr_{\Gamma\otimes \BG_m})/_{\Ran},
\end{equation} 
compatible with the factorization structures. 

\medskip

Under this equivalence, the $\otimes$-symmetric monoidal structure on $\on{Fact}(\Vect^\Gamma_{\otimes})$ corresponds to
the symmetric monoidal structure on $\Shv(\Gr_{\Gamma\otimes \BG_m})/_{\Ran}$ given by convolution along the
group structure on $\Gr_{\Gamma\otimes \BG_m}$.  The $\ast$-symmetric monoidal structure on $\on{Fact}(\Vect^\Gamma_{\otimes})$ 
corresponds to the symmetric monoidal structure on $\Shv(\Gr_{\Gamma\otimes \BG_m})/_{\Ran}$ given by pointwise
!-tensor product. 

\ssec{The twisting construction}

We obtain that $\on{Fact}(\CC)$ acquires an action of $\Shv(\Gr_{\Gamma\otimes \BG_m})/_{\Ran}$, viewed as
a symmetric monoidal category with respect to the pointwise !-tensor product. 

\medskip

Now, using \thmref{t:1-aff}, we obtain that we can upgrade $\on{Fact}(\CC)$ to a sheaf of categories over 
$\Gr_{\Gamma\otimes \BG_m}$, compatible with the factorization structure. 

\medskip

Hence, the construction of \secref{sss:twist sheaves of categ} allows to twist $\on{Fact}(\CC)$ by any factorization 
$E^{\times,\on{tors}}$-gerbe $\CG$ on $\Gr_{\Gamma\otimes \BG_m}$, and obtain another factorization sheaf of categories, 
to be denoted $\on{Fact}(\CC)_\CG$.

\medskip

Since the action of $\Shv(\Gr_{\Gamma\otimes \BG_m})/_{\Ran}$ on $\on{Fact}(\CC)$ is lax-compatible with
the symmetric monoidal structure on $\Shv(\Gr_{\Gamma\otimes \BG_m})/_{\Ran}$ given by convolution and
the existing symmetric monoidal structure on $\on{Fact}(\CC)$, if $\CG$ carries a \emph{commutative} structure
with respect to the group structure on $\Gr_{\Gamma\otimes \BG_m}$, the twisted sheaf of categories $\on{Fact}(\CC)_\CG$
carries a symmetric monoidal structure.

\end{document}